\documentclass[12pt]{article}

\usepackage{latexsym,amsmath,amscd,amssymb,graphics,amsthm}
\usepackage{enumerate,framed}

\usepackage{graphicx}

\usepackage[square,authoryear]{natbib}

\RequirePackage{color}
\definecolor{dullmagenta}{rgb}{0.4,0,0.4}   % #660066

\PassOptionsToPackage{colorlinks
	   ,urlcolor=dullmagenta
           ,citecolor=dullmagenta
           ,linkcolor=dullmagenta
           ,pagecolor=white
 	   }{hyperref}
\IfFileExists{hyperref.sty}{\RequirePackage{hyperref}}%
 {\typeout{hyperref.sty is not loaded so you will not get active hyperlinks}%
 \providecommand{\href}[2]{#2}%
}

\usepackage{url}

\usepackage[all]{xy}

\newcommand{\bfi}{\bfseries\itshape}

\makeatletter

\@addtoreset{figure}{section}
\def\thefigure{\thesection.\@arabic\c@figure}
\def\fps@figure{h, t}
\@addtoreset{table}{bsection}
\def\thetable{\thesection.\@arabic\c@table}
\def\fps@table{h, t}
\@addtoreset{equation}{section}

\makeatother

\begin{document}

\newtheorem{theorem}{Theorem}[section]
\newtheorem{definition}[theorem]{Definition}
\newtheorem{lemma}[theorem]{Lemma}
\newtheorem{remark}[theorem]{Remark}
\newtheorem{proposition}[theorem]{Proposition}
\newtheorem{corollary}[theorem]{Corollary}
\newtheorem{example}[theorem]{Example}

\def\below#1#2{\mathrel{\mathop{#1}\limits_{#2}}}

%%%%%%%%%%%%%%%%%%%%%%%%%%%%%%%%%%%%%%%%%%%%%%%%%%%%%%%%%%%%%%%%%%%%%%%%%%%%%%%
%%%%%%%%

%%%%%%%%%%%%%%%%%%%%%%%%%%%%%%%%%%%%%%%%%%%%%%%%%%%%%%%%%%%%%%%%%%%%%%%%%%%%%%%
%%%%%%%%

\title{Dirac Reduction for Nonholonomic Mechanical Systems and Semidirect Products}
\author{\hspace{-1cm}\normalsize
\begin{tabular}{cc}
Fran\c{c}ois Gay-Balmaz &
Hiroaki Yoshimura
\\ CNRS and Ecole Normale Sup\'erieure de Paris & Applied Mechanics and Aerospace Engineering
\\ Laboratoire de m\'et\'eorologie dynamique  & Waseda University
\\  24 Rue Lhomond 75005 Paris, France & Okubo, Shinjuku, Tokyo 169-8555, Japan \\ francois.gay-balmaz@lmd.ens.fr & yoshimura@waseda.jp\\
\end{tabular}\\\\
}

\date{}
\maketitle

\makeatother

\maketitle
\vspace{-0.1in}
\begin{center}\abstract{\vspace{2mm}This paper develops the theory of Dirac reduction by symmetry for nonholonomic systems on Lie groups with broken symmetry. The reduction is carried out for the Dirac structures, as well as for the associated Lagrange-Dirac and Hamilton-Dirac dynamical systems. This reduction procedure is accompanied by reduction of the associated variational structures on both Lagrangian and Hamiltonian sides. The reduced dynamical systems obtained are called the implicit Euler-Poincar\'e-Suslov equations with advected parameters and the implicit Lie-Poisson-Suslov equations with advected parameters. The theory is illustrated with the help of finite and infinite dimensional examples. It is shown that equations of motion for second order Rivlin-Ericksen fluids can be formulated as an infinite dimensional nonholonomic system in the framework of the present paper.}
%\vspace{2mm}

\end{center}

\paragraph{keywords:} Dirac structures, reduction by symmetry, nonholonomic systems, variational structures, semidirect products, Rivlin-Ericksen fluids. 

\tableofcontents

\section{Introduction}\label{Intro}
\paragraph{Dirac structures in mechanics.}
Dirac structures are geometric objects that generalize both Poisson brackets and (pre)symplectic structures on manifolds. They were originally developed by \cite{CoWe1988}; \cite{Cour1990} and \cite{Dorfman1993} and were named after Dirac's theory of constraints (\cite{Dirac1950}).

It turns out that
Dirac structures are the appropriate geometric objects for the formulation of equations of motion of nonholonomic Hamiltonian systems on arbitrary configuration manifolds or, more generally, of implicit Hamiltonian systems appearing as differential-algebraic equations (see, for instance, \cite{VaMa1995} and \cite{BC1997}). More recently, the notion of an implicit Lagrangian system, which may be a Lagrangian analogue of an implicit Hamiltonian system, was developed by \cite{YoMa2006a}, where it was shown that nonholonomic mechanical systems and L-C circuits can be formulated as degenerate Lagrangian systems in this context.

For the case of nonholonomic mechanics, given a constraint distribution $ \Delta _Q$ on the configuration manifold $Q$, there exists an associated Dirac structure $D_{ \Delta _Q}$ on the cotangent bundle $T^*Q$ of $Q$, introduced in \cite{YoMa2006a}. This Dirac structure is not integrable, unless the constraint is holonomic.
Given such an induced Dirac structure and a (possibly degenerate) Lagrangian defined on the tangent bundle $TQ$ of $Q$, a Lagrange-Dirac dynamical system can be defined (\cite{YoMa2006a}), which provides a geometric formulation of the equations of motion for the nonholonomic mechanical systems. The associated equations of motion are the so-called {\it implicit Lagrange-d'Alembert equations}. The Lagrange-Dirac system is naturally associated to a variational structure, called the {\it Lagrange-d'Alembert-Pontryagin principle} whose critical curves are precisely the solutions of the implicit Lagrange-d'Alembert equations.

On the Hamiltonian side, dynamical systems associated to Dirac structures were considered in \cite{Dorfman1987, CoWe1988} and \cite{Cour1990}. Further applications of Dirac structures to dynamical systems including L-C circuits and nonholonomic systems were developed by \cite{VaMa1995} and \cite{ BC1997}. The case of the Dirac structure $D_{ \Delta _Q}$ induced by a nonholonomic constraint $\Delta_{Q} \subset TQ$ was considered in \cite{YoMa2007b} to formulate a Hamilton-Dirac system in nonholonomic mechanics.

In the presence of symmetries, there is a well-developed reduction theory for Dirac structures in mechanics for $G$-invariant systems on Lie groups, both in the unconstrained case and in the constrained (nonintegrable) case (\cite{YoMa2007b}), where the associated Dirac reduction process (called Lie-Dirac reduction) yields the geometric framework for the study of the Euler-Poincar\'e-Suslov and Lie-Poisson-Suslov equations. 
For the more general case of a free and proper Lie group action on an arbitrary configuration manifold, a reduction theory called {\it Dirac cotangent bundle reduction} was developed for the canonical Dirac structure (\cite{YoMa2009}). It induces the {\it implicit Lagrange-Poincar\'e equations} on the Lagrangian side and  the {\it implicit Hamilton-Poincar\'e equations} on the Hamiltonian side. Note that this process of Dirac reduction does not only provide a geometric setting for the reduction of the equations of motion (both on the Lagrangian and Hamiltonian case), but also allows for a  corresponding reduction of the associated variational structures.

Other Dirac reduction processes have been developed by \cite{Cour1990}, \cite{Dorfman1993}, \cite{BS2001} and \cite{BCG2007}.

\paragraph{Lagrangian reduction for holonomic mechanical systems on Lie groups.} From the viewpoint of reduction by symmetries in Lagrangian systems, the simplest situation arises when the configuration manifold coincides with the Lie group of symmetries. This is the well-known situation of the Euler-Poincar\'e reduction whose main examples are rigid bodies and incompressible fluids, and this allows to reformulate the Euler-Lagrange equation on the Lie group as a first-order differential equation on the Lie algebra (see, for instance, \cite{MarRat1999}).

In many relevant situations in physics, however, the symmetry of the system is broken by the presence of a parameter held fixed in the Lagrangian description, i.e., before reduction.

After reduction (i.e., in the convective or spatial description) this parameter becomes time dependent and verifies an advection equation; hence it was named \textit{Euler-Poincar\'e reduction with advected parameters\/} for the associated reduction process. It has been developed in \cite{HMR1998a} and applied there to several examples such as heavy tops, compressible fluids and magnetohydrodynamics. The Hamiltonian formulation of the Euler-Poincar\'e reduction with advected parameters can be related to the earlier derived \textit{Lie-Poisson reduction on semidirect products} of \cite{MaRaWe1984}. It is useful to mention that, on the Lagrangian side, the Euler-Poincar\'e equations with advected parameters \textit{are not} the Euler-Poincar\'e equations on a semidirect product. Such a comment is crucial for the Dirac formulation that we will develop subsequently in the paper.

Physical systems with advected parameters can be also approached by using the general methods of Lagrangian reduction in \cite{CeIbMa1987} and \cite{CeMaRa2001a}.

Recently, several generalizations of Euler-Poincar\'e reduction with advection have  been made in order to treat the case of complex fluids (\cite{GBRa2009}), nonabelian charged fluids (\cite{GBRa2011}), rods and molecular strands (\cite{GBHoRa2009}, \cite{ElGBHoPuRa2009}), or nematic systems with broken symmetry (\cite{GBTr2010}).

\paragraph{Lagrangian reduction for nonholonomic systems.} One of the first paper in which a systematic theory of nonholonomic reduction is developed is \cite{Ko1992}. A Hamiltonian version of this theory was developed in \cite{BaSn1993}, while a Lagrangian version was given in \cite{BKMM1996}. Links between these theories and further developments were
given in \cite{KoMa1997,KoMa1998}. The intrinsic geometric formulation of the Lagrangian reduction theory for nonholonomic
systems was given in \cite{CeMaRa2001b}. An application of this reduction theory to the Euler disk was carried out in \cite{CeDi2007}, from which a geometric integrator was derived in \cite{CaCeDiDD2012}. Reduction for nonholonomic systems was generalized to the setting of Lie algebroids in \cite{CoDLMaMa2009}.

Nonholonomic Lagrangian reduction for systems with broken symmetry on Lie groups has not been developed yet and a part of our present work is dedicated to this task, besides the development of the Hamilton-Pontryagin variational formulation and the associated Dirac reduction. The closest relevant work in this direction is \cite{Sch2002}, in which nonholonomic Euler-Poincar\'e reduction was developed in a particular setting well appropriate for several examples of rigid bodies that roll without slipping. Besides the classical examples of rigid bodies with nonholonomic constraints, the nonholonomic reduction that we develop is also motivated from the work of \cite{GBPu2012,GBPu2014} on the dynamics of elastic strings with rolling contact.

\paragraph{Goal of the paper.}
The main goal of this paper is to develop the reduction theory of Dirac structures for holonomic and nonholonomic systems on Lie groups with broken symmetry. This concerns the reduction of Dirac structures as a geometric object, together with the reduction of the associated Lagrange-Dirac and Hamilton-Dirac dynamical systems, as well as the associated variational structures on both the Lagrangian and Hamiltonian sides.

The reduction process for these systems with advected parameters is carried out by considering a parameter dependent constraint $ \Delta ^{ a _0 }_G$ on a Lie group $G$, where the parameter $a _0 $ belongs to the vector space of advected parameters on which the group is assumed to act by representation. Under the appropriate invariance of both a Lagrangian  and constraints under the isotropy subgroup $G_{ a _0 }$ of $a _0 $, we derive the reduced equations of motion (called {\it implicit Euler-Poincar\'e-Suslov equations with advected parameters}) via the Lagrange-d'Alembert-Pontryagin principle, and with their Hamiltonian version via the Hamilton-d'Alembert principle. This approach is then developed in the context of Dirac geometry, by considering the parameter dependent Dirac structure $D_{\Delta ^{a_0}_G}$ on $T^*G$ associated to the constraint $ \Delta ^{ a _0 }_G$ and its reduction by the isotropy subgroup $G_{ a _0 }$. From the viewpoint of dynamics, this allows us to formulate the reduction of the associated Lagrange-Dirac and Hamilton-Dirac systems and also to show that they provide the appropriate geometric framework for the formulation of the Euler-Poincar\'e-Suslov equations with advected parameters and its Hamiltonian version. In particular, we show that, on the Hamiltonian side, the Dirac reduction can be interpreted as the second step of a Dirac reduction by stages for semidirect products, thereby extending the whole existing holonomic reduction theory for semidirect products to the context of the Dirac reduction for nonholonomic mechanics.

\paragraph{Outline of the paper.} The paper is organized as follows. In Section \ref{DDS}, we introduce the required mathematical ingredients for Dirac structures in mechanics and the associated variational structures. We first recall the Lagrange-d'Alembert-Pontryagin principle and the Hamilton-d'Alembert principle for nonholonomic mechanics, together with their reduced versions on Lie groups. Then, we review the definition of Dirac structures, mention the Dirac structure induced from a constraint set, and recall the definitions of Lagrange-Dirac and Hamilton-Dirac dynamical systems. Thus, we explain the reduction of Dirac structures associated with the invariant constraints on Lie groups (the so-called Lie-Dirac reduction), and the reduction of the corresponding Lagrange-Dirac and Hamilton-Dirac systems (the so-called Euler-Poincar\'e-Dirac and Lie-Poisson-Dirac reductions). In Section \ref{SPT}, we review the theory of (holonomic) Euler-Poincar\'e reduction with advected quantities together with its Hamiltonian analogue. In particular, we recall that on the Hamiltonian side the reduction process can be interpreted as the second stage of the ordinary Lie-Poisson reduction for a semidirect product. In Section \ref{VPAP}, we first present the Hamilton-Pontryagin principle and its reduced version for holonomic systems on Lie groups with advected quantities. Then we extend this principle to the case of nonholonomic systems, thereby obtaining the implicit Euler-Poincar\' e-Suslov equations with advected quantities. We also carry out reduction of the variational structure to obtain the implicit Lie-Poisson-Suslov equations with advected parameters  on the Hamiltonian side. Thus, we treat an important class of constraints appearing in several examples of nonholonomic rigid bodies, when the Lie group configuration space is a semidirect product. In Section \ref{Dirac_red_advect}, we carry out the reduction of Dirac structures associated to a parameter dependent nonholonomic constraint on a Lie group, and the reduction of the corresponding Lagrange-Dirac and Hamilton-Dirac dynamical systems. We show that these reduced systems are equivalent to the implicit Euler-Poincar\'e-Suslov equations with advected parameter and their Hamiltonian version. We also note that, on the Hamiltonian side, the Dirac reduction theory developed so far does only parallel partially the holonomic situation, which can be interpreted as the second stage reduction of a Lie-Poisson reduction for semidirect products. This discrepancy is solved in Section \ref{sec_NH_SDP}, where it is shown that a Dirac reduction by stages for semidirect products can be carried out for nonholonomic systems with advected quantities. The associated variational structures are explained. Finally in Section \ref{Examples} we consider several examples that illustrate our theory, such as the heavy top, the incompressible ideal fluid and MHD, the Chaplygin ball, and the Euler disk. We show that the equations of motion for the second order Rivlin-Ericksen fluids is a nonholonomic system in the framework of the theory developed in the present paper.

\section{Dirac dynamical systems}\label{DDS} 

This section contains an extended review of Dirac dynamical systems and the associated variational structures as required mathematical ingredients in this paper. We first recall the expression of the Lagrange-d'Alembert-Pontryagin principle and the Hamilton-d'Alembert principle, both for the unconstrained and constrained cases. In the special situation in which the configuration manifold is a Lie group, we review the reduced version of these variational structures. We then recall the definition of Dirac structures and the associated dynamical systems, both on the Lagrangian and Hamiltonian sides, together with the associated reduction procedures called {\it Lie-Dirac reduction} on a Lie group configuration space.

\subsection{Lagrange-d'Alembert-Pontryagin principle}

Let $Q$ be a manifold thought of as the configuration space of a mechanical system. We denote by $TQ\rightarrow Q$, $T^*Q \rightarrow Q$, and $TQ \oplus T^{\ast}Q \rightarrow Q$, the tangent bundle, the cotangent bundle, and the Pontryagin bundle, respectively, with local coordinates $(q,v) \in TQ$, $(q,p) \in T^{\ast}Q$, and $(q,v,p)\in TQ \oplus T^{\ast}Q$. Let $ \Delta _Q \subset TQ$ be a smooth constraint distribution on $Q$.  Let $L:TQ \to \mathbb{R}$ be the (possibly degenerate) Lagrangian of the system.
\medskip

The {\bfi Lagrange-d'Alembert-Pontryagin principle} is 
\begin{equation}\label{LagDALP}
\delta \int_{t_1}^{t_2}
 \left\{ L(q(t),v(t)) + \left\langle p(t),  \dot{q}(t)-v(t) \right\rangle \right\} \, dt =0,
\end{equation}
where $v$ has to satisfy the constraint $ v(t) \in \Delta _Q (q(t))$ and for variations $(\delta q(t), \delta v(t), \delta p(t))$ of the curves $(q(t),v(t),p(t))$ such that $ \delta q(t)\in \Delta _Q (q(t) )$ and vanishes at the endpoints. The stationarity conditions yield the equations, in local coordinates,
\begin{equation*}
\begin{split}\label{imp_ELdA_Eqn}
p=\frac{\partial L}{\partial v},  \quad \dot{q}=v\in \Delta _Q (q), \quad   \dot{p}-\frac{\partial L}{\partial q}\in \Delta _Q (q)^\circ,
\end{split}
\end{equation*}
called the {\bfi implicit Lagrange-d'Alembert equations} on $TQ\oplus T^{\ast}Q$; see \cite{YoMa2006b}.

Note that the implicit Lagrange-d'Alembert equations include the Lagrange-d'Alembert equations $ \dot{p}-{\partial L}/{\partial q} \in \Delta_{Q}(q)^{\circ}$, the Legendre transformation  $p={\partial L}/{\partial v}$ and the second--order condition $\dot{q}=v \in \Delta_{Q}(q)$. 

In the unconstrained case $\Delta_{Q}=TQ$, the principle \eqref{LagDALP} is called the {\bfi Hamilton-Pontryagin principle} and it recovers the {\it implicit Euler-Lagrange equations}.

\paragraph{Reduction on Lie groups.} When the Lagrangian is invariant under the tangent lifted action of a Lie group $G$ on the configuration manifold $Q$, one can reduce both the equations and the variational structure. We shall focus on the case when $Q=G$ and $G$ acts by left translation. As a consequence, the group $G$ acts on curves $ (g(t),v(t),p(t))$  in $TG\oplus T^{\ast}G$ by simultaneously left translating on each factor by the left-action and its tangent and cotangent lifts, i.e., the action of $h \in G$ is
\begin{equation*}
h \cdot (g(t),v(t),p(t)):=(hg(t),T_{g(t)} L_h \cdot v(t), T^\ast_{hg(t)} L_{h^{-1}} \cdot p(t))=:(hg(t),h v(t), h p(t)),
\end{equation*} 
where $T_{g(t)}L_{h}:T_{g(t)}G \to T_{hg(t)}G$ is the tangent of the left translation map $L_{h}:G \to G;\, g(t)\mapsto hg(t)$ at the point $g(t)$ and $T^{\ast}_{hg(t)}L_{h^{-1}}:T^{\ast}_{g(t)}G \to T^{\ast}_{hg(t)}G$ is the dual of the map $T_{hg(t)}L_{h^{-1}}:T_{hg(t)}G \to T_{g(t)}G$.

Let us assume that the constraint distribution $\Delta_{G} \subset TG$ on $G$ is left invariant under the group action $g \mapsto hg$, which means that for all $h \in G$, the subspace $\Delta_{G}(g) \subset T_{g}G$ is mapped by the tangent map of the group action to the subspace $\Delta_{G}(hg) \subset T_{hg}G$, that is
\begin{equation}\label{invariance_Delta} 
\Delta _G(hg)=h \Delta _G(g), \quad \text{for all $ h \in G$}. 
\end{equation}
By $G$-invariance, $\Delta _G$ is completely determined by its value at the identity, namely, by the vector subspace $ \mathfrak{g}  ^\Delta := \Delta _G (e)\subset \mathfrak{g}$.

Let $L:TG \rightarrow \mathbb{R}$ be a $G$-invariant Lagrangian and let $\ell:\mathfrak{g} \to \mathbb{R}$  given by $\ell:=L|{\mathfrak{g}}$ be the reduced Lagrangian.

The Lagrange-d'Alembert-Pontryagin action integral is invariant under the action of $G$ since $L$ is $G$--invariant and one easily checks that $\langle p(t), \dot{g}(t)-v(t) \rangle$ is also $G$-invariant.  The {\bfi reduced Lagrange-d'Alembert-Pontryagin principle} for nonholonomic mechanics is given by
\begin{equation}\label{RedLagDAPPrin}
\delta \int _{t_0}^{t_1} \left\{\ell(\eta )+ \left\langle \mu ,\xi - \eta  \right\rangle \right\} dt=0, \quad \delta{\xi}=\frac{\partial \zeta}{\partial t}+[\xi, \zeta], \quad \text{where} \quad  \zeta \in \mathfrak{g}  ^\Delta , \quad \eta \in \mathfrak{g}  ^\Delta,
\end{equation}
which yields the {\bfi implicit Euler-Poincar\'e-Suslov equations} on $(\mathfrak{g} \oplus \mathfrak{g}^{\ast}) \times V^{\ast}$: 
\begin{equation}\label{ImpEulPoinSusEqn}
\mu =  \frac{\delta \ell}{\delta \eta}, \quad \xi = \eta\in\mathfrak{g}^{ \Delta}, \quad \dot{\mu} - \operatorname{ad}_{\xi}^{\;\ast} \mu \in \left( \mathfrak{g}^{ \Delta}\right) ^\circ.
\end{equation}
This approach was developed by \cite{YoMa2007b} and is the implicit analogue of the Euler-Poincar\'e-Suslov theory. We refer to \cite{Ko1988} and \cite{Bloch2003} for the details on the original Suslov problem and its generalization. 
\medskip

\begin{remark}[\textbf{Alternative formulation of the Hamilton-Pontryagin principle}]\label{Hybrid_HP}
{\rm
For the unconstrained case, a similar variational principle to the one given in the above was also considered in \cite{BRMa2009}. It reads
\begin{equation}\label{Triv_HPrinciple}
\delta \int_{ t _1 }^{ t _2 } \{ \ell( \xi )+ \left\langle \mu , g ^{-1} \dot g - \xi \right\rangle \} \,dt=0,
\end{equation}
for arbitrary variations $( \delta \xi(t), \delta \mu (t)) $ of $( \xi (t), \mu (t)) \in \mathfrak{g} \oplus \mathfrak{g}^{\ast} $ and variations $\delta g(t)$ of $g(t) \in G$ vanishing at the endpoints and yields the stationarity conditions
\begin{equation*}\label{TrivImpEP}
\xi = g ^{-1} \dot g, \quad \mu = \frac{\delta \ell}{\delta \xi }, \quad \dot{\mu} - \operatorname{ad}^{\ast}_ \xi \mu =0. 
\end{equation*}
Note that  this variational principle in $G \times \mathfrak{g} \oplus  \mathfrak{g}^{\ast}$ can be understood as the {\it trivialized expression} of the Hamilton-Pontryagin principle on $TG \oplus T^{\ast}G$. As opposed to the \textit{reduced Hamilton-Pontryagin principle} (that is \eqref{RedLagDAPPrin} without constraints), the Hamilton-Pontryagin principle \eqref{Triv_HPrinciple} still contains the unreduced variables $g, \dot g$, and is therefore not the reduced expression of the Hamilton-Pontryagin principle on $TG \oplus T^*G$. However, it may have the advantage of involving only unconstrained variations, for instance, in the case of numerical integrations. Of course, both variational principles yield equivalent equations.
}
\end{remark} 

\subsection{Hamilton-d'Alembert phase space principle}

Let $H: T ^\ast Q \rightarrow \mathbb{R}  $ be a Hamiltonian function defined on the cotangent bundle (phase space) of the configuration manifold $Q$. In presence of  a smooth distribution constraint $ \Delta _Q \subset TQ$, the equations of motions are obtained by the
{\bfi Hamilton-d'Alembert principle in phase space}
\[
\delta \int \left\{ \left\langle p(t) , \dot q (t) \right\rangle -H(q(t) , p(t) ) \right\} dt=0,
\]
where $\dot q(t)  \in \Delta _Q(q(t))$, and for variations $ \delta q(t), \delta p(t)$ such that $ \delta q(t)  \in \Delta _Q(q(t)) $ and vanishes at the endpoints; see \cite{YoMa2006b}.
This principle yields the {\bfi Hamilton-d'Alembert equations} for nonholonomic mechanics:
\[
\dot q = \frac{\partial H}{\partial p}\in \Delta _Q , \quad \frac{\partial H}{\partial q}+\dot p \in \Delta ^\circ_Q.
\]
We refer to \cite{BaSn1993}, \cite{VaMa1995} and \cite{Marle1998}, for the description of geometric formalisms for Hamiltonian systems with nonholonomic constraints.

\paragraph{Reduction on Lie groups.} 
Let us consider the case $Q=G$, with $G$ acting on the left by translations and assume that the given distribution $\Delta_{G}$ on $G$ is left-invariant.

Let $H$ be a $G$-invariant Hamiltonian on $T^{\ast}G$ and let $h: \mathfrak{g}  ^\ast \rightarrow \mathbb{R}$ given by $h:=H|_{\mathfrak{g}^{\ast}}$ be the reduced Hamiltonian. The {\bfi reduced Hamilton-d'Alembert principle} is given by
\begin{equation*}\label{RedHamDALPS}
\delta \int_{t_1}^{t_2} \left \{ \left\langle \mu , \xi  \right\rangle -h(\mu) \right\}dt=0, 
\end{equation*}
for $\xi \in \mathfrak{g}  ^\Delta$ and with the variations of the form 
$$
\delta{\xi}=\frac{\partial \zeta}{\partial t}+[\xi, \zeta]. 
$$
This principle yields the {\bfi implicit Lie-Poisson-Suslov equations} for nonholonomic mechanics:
\begin{equation*}\label{ImpLiePoiSusEqn}
\dot \mu  - \operatorname{ad}^*_\xi \mu \in (\mathfrak{g}  ^\Delta)^\circ \quad \text{and} \quad \frac{\delta h}{\delta \mu } =\xi  \in \mathfrak{g}  ^\Delta.
\end{equation*}

For the unconstrained case, this principle is called the {\bfi Lie-Poisson variational principle} (see \cite{CeMaPeRa2003}) and it yields the {\bfi implicit Lie-Poisson equations}
\[\frac{\delta h}{\delta \mu }= \xi , \quad \dot \mu - \operatorname{ad}^{\ast}_ \xi \mu =0.  
\]

\begin{remark}[\textbf{Alternative formulation of the Lie-Poisson variational principle}]\label{Triv_HamPrinPhSp}
{\rm As before, one can develop the trivialized Hamilton principle in phase space as
\begin{equation*}\label{Triv_HP}
\delta \int_{ t _1 }^{ t _2 } \{ \left\langle \mu , g ^{-1} \dot g  \right\rangle - h(\mu) \} \,dt=0,
\end{equation*}
for arbitrary variations $\delta \mu (t) $ of $\mu (t) \in  \mathfrak{g}^{\ast} $ and variations $\delta g(t)$ of $g(t) \in G$ vanishing at the endpoints. It yields the stationarity conditions
\begin{equation*}
g ^{-1} \dot g=\frac{\delta h}{\delta \xi }, \quad  \dot{\mu} = \operatorname{ad}^{\ast}_{g ^{-1} \dot g} \mu. 
\end{equation*}
}
\end{remark} 

\subsection{Dirac dynamical systems}
\paragraph{Linear Dirac structures.} We first recall the definition of a Dirac structure on a vector space $V$, see \cite{CoWe1988}. Let $V^{\ast}$ be the dual space of $V$, and $\langle\cdot \, , \cdot\rangle$
be the natural paring between $V^{\ast}$ and $V$. Define the
symmetric paring
$\langle \! \langle\cdot,\cdot \rangle \!  \rangle$
on $V \oplus V^{\ast}$ by
\begin{equation*}
\langle \! \langle\, (v,\alpha),
(\bar{v},\bar{\alpha}) \,\rangle \!  \rangle
=\langle \alpha, \bar{v} \rangle
+\langle \bar{\alpha}, v \rangle,
\end{equation*}
for $(v,\alpha), (\bar{v},\bar{\alpha}) \in V \oplus V^{\ast}$.
A {\bfi Dirac structure} on $V$ is a subspace $D \subset V \oplus
V^{\ast}$ such that
$D=D^{\perp}$, where $D^{\perp}$ is the orthogonal
of $D$ relative to the pairing
$\langle \! \langle \cdot,\cdot \rangle \!  \rangle$.

\paragraph{Dirac structures on manifolds.}
Now let $P$ be a manifold and let $TP \oplus T^{\ast}P$ denote the Pontryagin bundle. In this paper, we shall call a  subbundle $ D \subset TP \oplus T^{\ast}P$ a {\bfi Dirac structure on the manifold $P$}, if $D(x)$ is a linear Dirac structure on the vector space $T_{x}P$ at each point $x \in P$.

For example, a given two-form $\Omega$ on $P$ together with a distribution $\Delta$ on $P$ determines the Dirac structure $D_ \Delta $ on $P$ defined at $ x \in P$ by
\begin{equation}\label{DiracManifold}
\begin{split}
D(x)=\{ (v_{x}, \alpha_{x}) \in T_{x}P \times T^{\ast}_{x}P
  \; \mid \; & v_{x} \in \Delta(x), \; \mbox{and} \\ 
  & \alpha_{x}(w_{x})=\Omega_{\Delta}(v_{x},w_{x}) \; \;
\mbox{for all} \; \; w_{x} \in \Delta(x) \},
\end{split}
\end{equation}
where $\Omega_{\Delta}$ is the restriction of $\Omega$ to $\Delta$.

A Dirac structure $D$ is said to be {\it integrable} if  the condition
\begin{equation*}\label{ClosedCond}
\langle \pounds_{X_1} \alpha_2, X_3 \rangle
+\langle \pounds_{X_2} \alpha_3, X_1 \rangle+\langle \pounds_{X_3}
\alpha_1, X_2 \rangle=0
\end{equation*}
is satisfied for all pairs of vector fields and one-forms $(X_1, \alpha_1)$,
$(X_2,\alpha_2)$, $(X_3,\alpha_3)$ that take values in $D$,
where $\pounds_{X}$ denotes the Lie derivative  along the vector
field $X$ on $P$. 

\begin{remark}[\textbf{Courant bracket}]{\rm  Let $\Gamma(TP \oplus T^{\ast}P)$ be the space of local sections of $TP \oplus T^{\ast}P$, endowed with the Courant bracket (\cite{Cour1990}) defined by
\[
\begin{split}
\left[(X_{1},\alpha_{1}),(X_{2},\alpha_{2})\right] &:= \left( \left[X_{1}, X_{2} \right],  \pounds_{X_1} \alpha_{2}- \pounds_{X_2} \alpha_{1} + \mathbf{d}\left< \alpha_{2}, X_{1} \right>\right)\\
&\;=\left( \left[X_{1}, X_{2} \right],  \mathbf{i}_{X_1} \mathbf{d}\alpha_{2}- \mathbf{i}_{X_2} \mathbf{d}\alpha_{1} +\mathbf{d} \left< \alpha_{2}, X_{1} \right>\right).
\end{split}
\]
This bracket is skew-symmetric but fails to satisfy the Jacobi identity. As shown in \cite{Dorfman1993}, a Dirac structure $D \subset TP \oplus T^{\ast}P$ is integrable if and only if it is closed under the Courant bracket,
\[
\left[\Gamma(D), \Gamma(D) \right] \subset \Gamma(D).
\]
In this paper however, we shall not use the Courant bracket structure.}
\end{remark} 

\paragraph{Induced Dirac structures on cotangent bundles.} 
One of the most important and interesting Dirac structures in mechanics is the one induced on the cotangent bundle $T^{\ast}Q$ from kinematic constraints, whether holonomic or nonholonomic, given by a distribution $ \Delta _Q$ on the configuration manifold $Q$. 
We define the lifted distribution on $T^{\ast}Q$ by
\begin{equation*}
\Delta_{T^{\ast}Q}
:=( T\pi_{Q})^{-1} \, (\Delta_{Q}) \subset T(T^{\ast}Q),
\end{equation*}
where $\pi_{Q}:T^{\ast}Q \to Q$ is the canonical projection and 
$T\pi_{Q}:T(T^{\ast}Q) \to TQ$ is its tangent map. Let $\Omega$ be the canonical symplectic form on $T^{\ast}Q$.  
The {\it induced Dirac structure} $D_{\Delta_Q}$ on $T^{\ast}Q$ is the subbundle of $T  T^{\ast}Q \oplus T ^{\ast} T^{\ast}Q$, whose fiber is given for each $p_{q} \in
T^{\ast}Q$ by
\begin{align*}
D_{\Delta_Q}(p_{q})
& =\{ (v_{p_{q}}, \alpha_{p_{q}}) \in T_{p_{q}}T^{\ast}Q \times
T^{\ast}_{p_{q}}T^{\ast}Q \mid v_{p_{q}} \in
\Delta_{T^{\ast}Q}(p_{q}),  \; \mbox{and} \;  \nonumber
\\ & \qquad \qquad
\alpha_{p_{q}}(w_{p_{q}}) = \Omega( p _q) (v_{p_{q}},w_{p_{q}}) \;\; \mbox{for
all} \;\; w_{p_{q}} \in \Delta_{T^{\ast}Q}(p_{q})\}.
\end{align*}
It is known that the induced Dirac structure is integrable if and only if the constraint distribution $ \Delta _Q$ is holonomic, see \cite{VaMa1995}, \cite{DaVdS1998}. Therefore, in this paper we are especially interested in the case where the Dirac structure is not integrable.
%\todo{For Hiro: I moved the magenta text here. Before, it was at the end of the next paragraph}

\paragraph{Local representation of the Dirac structure.} Choose local coordinates $ q ^i $ so that $Q$ is locally represented by an open set $U \subset
\mathbb{R}^n$. For each $ q \in U$, the constraint distribution $\Delta_Q$  defines a subspace of $ \mathbb{R}  ^n $, denoted $\Delta(q) \subset \mathbb{R}^n$. Now writing the projection map $\pi_Q: T^{\ast}Q\rightarrow Q $
locally as $(q,p) \mapsto q$, its tangent map is locally given by
$T\pi_Q : (q, p, \dot{q}, \dot{p}) \mapsto (q, \dot{q})$. Thus, we
can locally represent the lifted distribution $\Delta_{T^{\ast}Q}$ as
\[
\Delta_{T^{\ast}Q} \cong \left\{ v _{(q,p)} = (q,p, \dot{q}, \dot{p} )
\mid q \in U, \dot{q} \in \Delta (q) \right\}.
\]
Letting points in $T ^{\ast} T ^{\ast} Q $ be locally denoted by
$\alpha_{(q,p)} = (q,p, \alpha, u )$, where $\alpha$ is a covector
and $w $ is a vector, the local expression for the induced Dirac structure 
is given by
\begin{align}\label{local_induced_Dirac} 
D_{\Delta_Q}(z)  
  & =
\left\{
\left( (q,p, \dot{q}, \dot{p}), (q,p, \alpha, u) \right) \mid
\dot{q} \in \Delta (q), \; u = \dot{q},  \; \mbox{and} \;
\alpha +\dot{p} \in \Delta(q)^{\circ}
\right\}.
\end{align}

\paragraph{Canonical diffeomorphisms.} The iterated tangent and cotangent bundles are crucial objects for the understanding of the interrelation between Lagrangian systems and Hamiltonian systems especially in the context of Dirac structures. In particular, there 
are two canonical diffeomorphisms between $T^{\ast}(TQ)$, $T(T^{\ast}Q)$ and 
$T^{\ast}(T^{\ast}Q)$ that were studied by \cite{Tu1977} in the context of the generalized Legendre transform. These canonical diffeomorphisms, together with the various projection maps involved, are illustrated in the Fig. \ref{BundPic}, below.

\medskip

First there is the canonical diffeomorphism $\kappa_{Q}: T(T^{\ast}Q) \to T^{\ast}(TQ)$, locally given by $(q, p, \delta q, \delta p) \mapsto (q, \delta q, \delta p, p)$,
where $(q,p)$ are local coordinates
of $T^{\ast}Q$ and $(q,p,\delta{q},\delta{p})$ are the corresponding
coordinates of $T(T^{\ast}Q)$, while $(q, \delta q, \delta p, p)$ are
the local coordinates of $T^{\ast}(TQ)$ induced by $\kappa_{Q}$. 
Second, there is the canonical diffeomorphism $\Omega^{\flat}:T(T^{\ast}Q) \to T^{\ast}(T^{\ast}Q)$ associated with the canonical symplectic structure $\Omega$, locally given by $(q,p,\delta{q},\delta{p}) \mapsto (q,p,-\delta{p}, \delta{q})$.
Thus, we can define a diffeomorphism $\gamma_{Q}: T^{\ast}(TQ )\to T^{\ast}(T^{\ast}Q)$ by
$\gamma_{Q}:=\Omega^{\flat} \circ \kappa_{Q}^{-1}$,
locally given by $(q,\delta{q},\delta{p},p) \mapsto (q,p,-\delta{p}, \delta{q})$.

\begin{figure}[h]
\begin{center}
%\hspace{2cm}
\includegraphics[scale=.6]{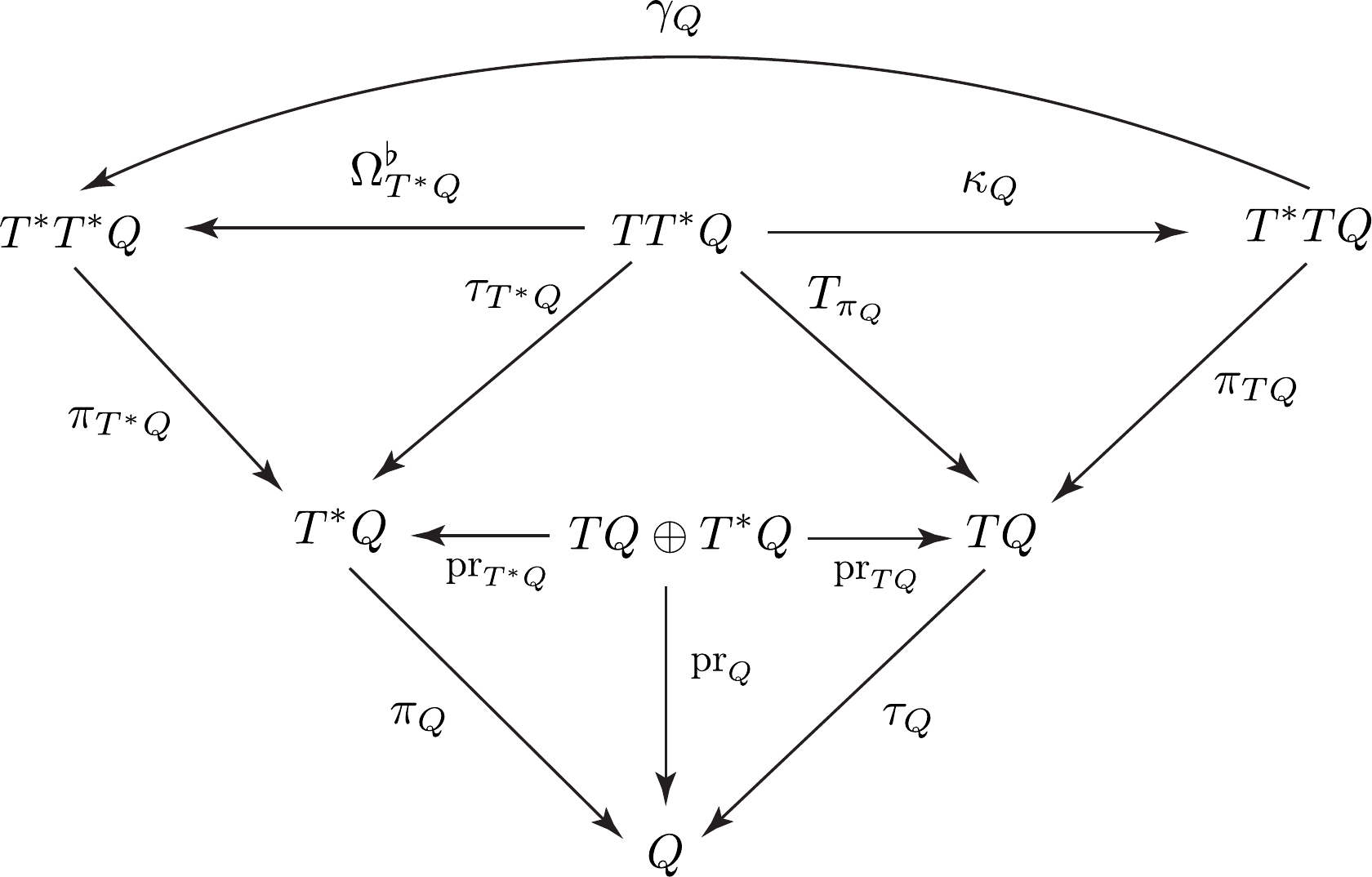}
\caption{A diagram of the canonical diffeomorphisms and bundle projections.}
\label{BundPic}
\end{center}
\end{figure}

\paragraph{Lagrange-Dirac dynamical systems.} Let us quickly review from \cite{YoMa2006a, YoMa2006b} the theory of implicit Lagrangian systems or Lagrange-Dirac systems. 
Let $L:TQ \to \mathbb{R}$ be a Lagrangian, possibly degenerate. The differential $ \mathbf{d} L:TQ \rightarrow T^{\ast}(TQ)$ of $L$ is the one-form on $TQ$ locally given by
\[
\mathbf{d} L(q,v)= \left( q,v, \frac{\partial L}{\partial q}, \frac{\partial L}{\partial v}\right) .  
\]
Using the canonical diffeomorphism $ \gamma_{Q}:T^{\ast}(TQ) \rightarrow T^{\ast}(T^{\ast}Q)$, we define the {\bfi Dirac differential} of $L$ by 
\[
\mathbf{d}_{D} L:= \gamma_{Q} \circ \mathbf{d} L \in \Gamma(T^{\ast}T^{\ast}Q),
\]
which is locally given by
$$
\mathbf{d}_{D} L(q,v)= \left(q,\frac{\partial L}{\partial v}, - \frac{\partial L}{\partial q},  v \right).
$$

\begin{definition}[\bf Lagrange-Dirac dynamical systems]
Let $ \Delta _Q \subset TQ$ be a distribution on $Q$ and consider the induced Dirac structure $D_{\Delta_{Q}}$ on $T^{\ast}Q$. The equations of motion of a {\bfi Lagrange-Dirac dynamical system} (or an implicit Lagrangian system)  $(Q, \Delta_Q, L)$ are given by 
\begin{equation}\label{LDirac_system} 
\left( (q(t),p(t),\dot{q}(t),\dot{p}(t)), \mathbf{d}_{D} L(q(t),v(t))\right)  \in D_{ \Delta _Q }(q(t),p(t)).
\end{equation} 

Any curve $(q(t),v(t),p(t)) \in TQ \oplus T^{\ast}Q,\,t_{1} \le t \le t_{2}$ satisfying \eqref{LDirac_system} is called a {\bfi solution curve} of the Lagrange-Dirac dynamical system.
\end{definition}

It follows from \eqref{local_induced_Dirac}  and \eqref{LDirac_system} that $(q(t),v(t),p(t))$, $t _1 \leq t \leq t _2 $ is a solution curve if and only if it satisfies the implicit Lagrange-d'Alembert equations
\begin{equation*}
p =\frac{\partial L}{\partial v }, \quad \frac{dq}{dt} =v \in \Delta(q), \quad  \frac{dp}{dt} - \frac{\partial L}{\partial q}
\in \Delta (q)^{\circ}.
\end{equation*}

\begin{remark}{\rm Note that the equation $p ={\partial L}/{\partial v }$ arises from the equality of the base points $(q,p)$ and $\left(q,{\partial L}/{\partial v}\right)$ in \eqref{LDirac_system}. }
\end{remark}

\paragraph{Energy conservation for implicit Lagrangian systems.}
Let  $(Q, \Delta_{Q}, L)$ be a  Lagrange-Dirac dynamical system.
Define the energy function $E$ on $TQ \oplus T^{\ast}Q$ by
\[
E(q,v,p)=\langle p, v \rangle -L(q,v).
\]
If $(q(t), v(t),p(t))$ in $TQ \oplus T^{\ast}Q$ is a solution curve of the Lagrange-Dirac system $(Q, \Delta_{Q}, L)$, then the energy $E(q(t),v(t),p(t))$ is constant in time. This is shown as follows:
\[
\frac{d}{dt} E = \left\langle \dot{p}, v \right\rangle
                   + \left\langle p, \dot{v} \right\rangle
                   - \frac{\partial L }{\partial q } \dot{q}
                   -\frac{\partial L }{\partial v } \dot{v}  = \left\langle
       \dot{p} - \frac{\partial L}{\partial q}, v
       \right\rangle,
\]
which vanishes since $\dot{q}=v \in \Delta (q)$ and
since $\dot{p} - {\partial L}/{\partial q} \in \Delta 
(q)^{\circ}$.

\paragraph{Hamilton-Dirac dynamical systems.} 
If the Lagrangian $L$ on $TQ$ is hyperregular, one can define a Hamiltonian $H$ on $T^{\ast}Q$ by Legendre transformation
$$
H:=E \circ \mathbb{F}L^{-1},
$$
where $E(v_{q})=\left<\mathbb{F}L(v_{q}),v_{q}\right>-L(v_{q})$. 

\begin{definition}[\bf Hamilton-Dirac dynamical system]
Given a Hamiltonian $H:T^{\ast}Q \to \mathbb{R}$ and an induced Dirac structure $D_{\Delta_{Q}}$ on $T^{\ast}Q$, the equations of motion of a {\bfi Hamilton-Dirac system} (also called a implicit Hamiltonian system) $(Q, \Delta_{Q}, H)$ are given by
\begin{equation}\label{CondImpHam}
((q(t),p(t),\dot{q}(t),\dot{p}(t)),\mathbf{d}H(q(t),p(t))) \in D_{\Delta_{Q}}(q(t),p(t)).
\end{equation}

Any curve $(q(t),p(t)) \in T^{\ast}Q,\,t_{1} \le t \le t_{2}$ satisfying \eqref{CondImpHam} is called a {\bfi solution curve} of the Hamilton-Dirac dynamical system.
\end{definition}

Using \eqref{local_induced_Dirac}, a curve $(q(t),p(t))$ is a solution curve of the Hamilton-Dirac system $(Q, \Delta_{Q}, H)$ if and only if it verifies the Hamilton-d'Alembert equations
\[
\dot q = \frac{\partial H}{\partial p}\in \Delta , \quad \frac{\partial H}{\partial q}+\dot p \in \Delta ^\circ,
\]
see \cite{YoMa2006b}.
\medskip

This is a special instance of an implicit Hamiltonian system on a Poisson manifold, as developed by \cite{VaMa1995}.

A more general geometric setting for constrained implicit Lagrangian and Hamiltonian systems has been developed in \cite{GrGr2011} via the concept of Dirac algebroids.
 
\subsection{Lie-Dirac reduction}
By {\bfi Lie-Dirac reduction} we mean the reduction of a $G$-invariant Dirac structure on $T^{\ast}G$. This process includes the case of the canonical Dirac structure as well as the case of a Dirac structure induced by a nonholonomic constraint. Theses cases have been developed in \cite{YoMa2007b}. In the present paper we will need two extensions of this theory. First, we will consider the case of Lie-Dirac reduction for an induced Dirac structure, in the presence of an advected parameter. Second we will develop the Lie-Dirac reduction for a more general Dirac structure than the one induced by a nonholonomic constraint.

\paragraph{Invariance and reduction of Dirac structures.}
Before going into details on Lie-Dirac reduction, let us just recall the definition of invariant Dirac structures (see, for instance, \cite{Dorfman1993,LiuWeiXu1998} and \cite{BS2001}). 

Let $P$ be a manifold and $D$ be a Dirac structure on $P$ with a Lie group $G$ acting on $P$. We denote this action by $\Phi:G \times P \to P$ and the action of a group element $h \in G$ on a point $x \in P $ by $h \cdot x= \Phi ( h, x ) = \Phi_{h}(x)$ for $h \in G$ and $x \in P$, so that $\Phi_{h}:P \to P$. Then, a Dirac structure $D \subset TP \oplus T ^{\ast} P $ is $G$--invariant if 
\begin{equation}\label{invariance_Dirac} 
((\Phi_{h})_{\ast}X, (\Phi_{h})_{\ast}\alpha) \in D
\end{equation} 
for all $h \in G$ and $(X, \alpha) \in D$.

Suppose that the action $ \Phi $ is free and proper and consider its natural lift on the Pontryagin bundle $TP \oplus T^*P$. The quotient space $(TP \oplus T^*P)/G= (TP)/G \oplus (T^*P)/G$ is called the {\bfi reduced Pontryagin bundle}.
It is easily checked that the natural lift of the $G$-action to the Pontryagin bundle preserves the
symmetric paring (as well as the Courant bracket).
Similarly as before, a subbundle $D \subset (TP)/G \oplus (T^*P)/G$ is called a {\bfi Dirac subbundle} if for all $x \in P/G$, the vector space $D(x) \subset V \oplus V^\ast $, with $V= ((TP)/G)_x$ is a Dirac structure on $V$.
If we suppose that the Dirac structure $D \subset TP \oplus T^*P$ is $G$-invariant, then the quotient space $D/G$ is easily verified to be a Dirac subbundle of the reduced Pontryagin bundle. The reduction procedure of Dirac structures $D/G  \subset (TP/G) \oplus (T^{\ast}P/G)$ was developed by \cite{YoMa2007b} for the case $P=T^{\ast}G$ and by \cite{YoMa2009} for the case $P=T^{\ast}Q$ in the unconstrained case. The reduced Pontryagin bundle $(TP)/G \oplus (T^*P)/G$ is an example of a Courant algebroid in the general sense of \cite{LiuWeiXu1998}, although the algebroid structure was not explicitly utilized in the reduction procedures of Dirac structures.

\paragraph{Induced Dirac structures on $T^{\ast}G$ and trivialized expressions.} As in \eqref{invariance_Delta}, assume that the constraint distribution $\Delta_{G} \subset TG$ on $G$ is left invariant under the group action.
% $g \mapsto hg$, which means that for all $h \in G$, the subspace $\Delta_{G}(g) \subset T_{g}G$ is mapped by the tangent of the group action to the subspace $\Delta_{G}(hg) \subset T_{hg}G$, that is
%\begin{equation}\label{invariance_Delta} 
%\Delta _G(hg)=h \Delta _G(g), \quad \text{for all $ h \in G$}. 
%\end{equation} 
We denote, as before, by $\Delta_{T^{\ast}G}$ the lifted distribution on $T^{\ast}G$ and by $D_{\Delta_G}$ the induced Dirac structure  on $T^{\ast}G$.

Let $ \bar \lambda :T^*G \rightarrow G \times \mathfrak{g}  ^\ast $ be the left trivialization of the cotangent bundle. The trivialized lifted distribution $\Delta_{G \times \mathfrak{g}^{\ast}}$ on $G \times \mathfrak{g}^\ast $ is defined by $\Delta_{G \times \mathfrak{g}^{\ast}}:=(T\bar{\pi}_{G})^{-1}(\Delta_{G})$, where $\bar{\pi}_{G}: G \times \mathfrak{g}^{\ast} \to G$ is defined such that $\pi_{G}=\bar{\pi}_{G} \circ \bar{\lambda}$. For each $(g, \mu ) \in G \times \mathfrak{g}  ^\ast $, we get
\[
\Delta_{G \times \mathfrak{g}^{\ast}}(g,\mu)=\left\{(v_g ,\rho) \in T_{g}G \times \mathfrak{g}^{\ast} \mid v \in \Delta_G(g) \right\}= \Delta _G(g) \times \{ \mu \} \times \mathfrak{g}  ^\ast .
\]

We define the trivialized Dirac structure on $G \times \mathfrak{g}^{\ast}$ by
\[
\bar{D}_{\Delta_{G}}:=\bar{\lambda}_{\ast}D_{\Delta_{G}}\subset T(G \times\mathfrak{g}  ^\ast ) \oplus T^*(G \times \mathfrak{g}  ^\ast ).
\]
For each $(g,\mu) \in G \times \mathfrak{g}^\ast$, we have
\[
\begin{split}
\bar{D}_{\Delta_{G}}(g,\mu)&=\{ (v_g ,\rho),(\beta_g ,\eta) \in (T_{g}G \times \mathfrak{g}^{\ast}) \times (T_{g}^{\ast}G \times \mathfrak{g}) \mid (v_g ,\rho) \in  \Delta_{G \times \mathfrak{g}^{\ast}}(g,\mu) , \\
\text{and} &\;\; \langle \beta_g , w_g  \rangle +  \langle \sigma, \eta \rangle =\omega(g,\mu)((v_g ,\rho),(w_g ,\sigma)), \;\; \text{for all} \;\;  (w_g ,\sigma) \in \Delta_{G \times \mathfrak{g}^{\ast}}(g,\mu) \},
\end{split}
\]
where $\omega=\bar{\lambda}_{\ast}\Omega$ is the canonical symplectic structure on $G \times \mathfrak{g}^{\ast}$ given, at  each point $(g,\mu) \in G \times \mathfrak{g}^{\ast}$, by
\begin{equation*} 
\label{symGtig}
 \omega(g,\mu)((v_g ,\rho),(w_g ,\sigma)) 
 = \langle \sigma, g ^{-1}v_g  \rangle- \langle \rho, g ^{-1} w_g  \rangle
     + \langle \mu, [g ^{-1}v_g , g ^{-1}w_g ] \rangle,
\end{equation*}
where $(w_g ,\sigma) \in T_{(g,\mu)}(G \times \mathfrak{g}^{\ast}) \cong T_{g}G \times \mathfrak{g}^{\ast}$.

\medskip

\paragraph{Invariance and reduction of the induced Dirac structure.} From the $G$-invariance \eqref{invariance_Delta} of the distribution $ \Delta _G$, we deduce that it is uniquely determined by its value at $e$, which we denote as
\[
\mathfrak{g}  ^ \Delta  := \Delta _G(e)\subset T_e G= \mathfrak{g}  .
\]
Similarly, the distribution $ \Delta _{G \times \mathfrak{g}  ^\ast }\subset T( G \times \mathfrak{g}  ^\ast )$ can be completely determined by its values at $(e, \mu )$, for all $ \mu\in \mathfrak{g}$, given by
\[
\Delta _{G \times \mathfrak{g}  ^\ast }( e , \mu )= \mathfrak{g}  ^ \Delta \times \mathfrak{g}  ^\ast\subset T_{(e, \mu )}(G \times \mathfrak{g}  ^\ast ) \cong \mathfrak{g}  \times \mathfrak{g}  ^\ast .
\]
Let us denote by $\Lambda _h(g, \mu ):=(hg, \mu )$ the left $G$-action on $G \times \mathfrak{g}  ^\ast$ induced, via $\bar \lambda $, by the cotangent lift of left translation by $G$ on $T^*G$. From the invariance of $ \Delta _G$, it follows that the induced Dirac structure $\bar{D}_{\Delta_{G}}$ on $G \times \mathfrak{g}^{\ast}$ is also $G$--invariant, i.e.
\begin{equation*}
((\Lambda_{{h}})_{\ast}X,(\Lambda_{{h}})_{\ast}\alpha) \in \bar{D}_{\Delta_{G}},
\end{equation*}
for all $(X, \alpha) \in \bar{D}_{\Delta_{G}}$ and $ h \in G$, as in \eqref{invariance_Dirac}. In particular, $\bar{D}_{\Delta_{G}}$ can be uniquely determined by its expression at $(e,\mu)$ as
\begin{equation*}
\begin{split}
\bar{D}_{\Delta_{G}}(e,\mu)=&\{ ((\xi,\rho),(\beta ,\eta)) \in W \oplus W^{\ast} \mid  (\xi,\rho) \in  \mathfrak{g}^{\Delta} \oplus \mathfrak{g}^{\ast}, \\
& \; \text{and} \;  \langle \beta , \zeta \rangle +  \langle \sigma, \eta \rangle=\omega(e,\mu)((\xi,\rho),(\zeta,\sigma)) \;\;\text{for all} \;\; (\zeta,\sigma) \in \mathfrak{g}^{\Delta} \oplus \mathfrak{g}^{\ast} \}, \nonumber
\end{split}
\end{equation*}
where $W:= \mathfrak{g}  \times \mathfrak{g}  ^\ast $.

Hence, the reduction of the induced Dirac structure $D_{ \Delta _G}$ can be obtained by taking the quotients by $G$, which is denoted $D_{\Delta_{G}}^{/G}(\mu)$. For each $\mu \in \mathfrak{g}^{\ast}$, we obtain
\begin{equation*}\label{reduced_InducedDirac}
\begin{split}
D_{\Delta_{G}}^{/G}(\mu)=& \{ ((\xi,\rho),(\beta ,\eta)) \in W \oplus W^{\ast} \mid 
(\xi,\rho) \in \mathfrak{g}^{\Delta} \oplus \mathfrak{g}^{\ast}, \\
&\text{and} \;\;  \langle \beta , \zeta \rangle +  \langle \sigma, \eta \rangle =\omega^{/G}(\mu)((\xi,\rho),(\zeta,\sigma)) \;\; \text{for all} \;\; (\zeta,\sigma) \in \mathfrak{g}^{\Delta} \oplus \mathfrak{g}^{\ast}
\},
\end{split}
\end{equation*}
where $\omega^{/G}(\mu)$ is the {\it $\mu$-dependent symplectic structure} on $W$ given by
\begin{equation*}\label{reduced_Omega}
\omega^{/G}(\mu)((\xi,\rho),(\zeta,\sigma))= \left\langle \sigma, \xi \right\rangle
- \left\langle \rho, \zeta \right\rangle + \left\langle \mu, [\xi,\zeta] \right\rangle,
\end{equation*}
where we note that it consists of the canonical two-form as well as the Lie-Poisson structure.

Thus, it follows that for fixed $\mu \in \mathfrak{g}^{\ast}$, the reduced structure $D_{\Delta_{G}}^{/G}(\mu)$ is  given by
\begin{equation}\label{reducedIndDiracAlt}
D_{\Delta_{G}}^{/G}(\mu)= \{ ((\xi,\rho),(\beta ,\eta )) \in W \oplus W^{\ast} \mid  \eta = \xi \in \mathfrak{g}^{\Delta},\;
\beta  + \rho- \mathrm{ad}_{\xi\,}^{\ast}\mu \in (\mathfrak{g}^{\Delta})^{\circ} \}.
\end{equation}
As mentioned earlier, $D_{\Delta_{G}}^{/G}$ is automatically a Dirac structure in the reduced Pontryagin bundle $\left( T(T^*G) \oplus T^*(T^*G) \right) /G\cong \mathfrak{g}  ^\ast \times (W \oplus W ^\ast )\rightarrow \mathfrak{g}  ^\ast $. As to the details, see \cite{YoMa2007b}.

%\begin{figure}[h]
%\begin{center}
%\hspace{2cm}
%\includegraphics[scale=.6]{RM-LieDirac}
%\caption{Lie-Dirac Reduction}
%\label{LDred}
%\end{center}
%\end{figure}

%
\subsection{Euler-Poincar\'e-Dirac reduction}\label{EuPoinSusRed_Section}
In this section, making use of the Lie-Dirac reduction, we will recall the reduction of a Lagrange-Dirac dynamical system $(G,\Delta_{G},L)$, where we assume that the constraint distribution $\Delta_{G} \subset TG$ is left $G$-invariant. In particular, we will show how the reduced Lagrange-Dirac system can be obtained in an {\it implicit version} of the so-called {\it Euler-Poincar\'e-Suslov equations for nonholonomic mechanics} and we also illustrate an example of the Suslov problem of rigid body systems with nonholonomic constraints.

\paragraph{Implicit Lagrangian systems on Lie groups.} Let $G$ be a Lie group, let $D_{\Delta_{G}}$ be the Dirac structure induced from a given distribution $\Delta_{G}$ and let $L$ be a Lagrangian (possibly degenerate) on $TG$. Recall that the equations of motion of a Lagrange-Dirac system $(G,\Delta_{G},L)$ is given by
\[
((g,p,\dot{g},\dot{p}),\mathbf{d}_{D}L(g,v)) \in D_{\Delta_{G}}(g,p).
\]
Recall also that the canonical two-form $\Omega$ on $T^{\ast}G$ is given in coordinates by
\[
\Omega \left( (g,p, u_1, \alpha_1 ), (g,p, u _2, \alpha_2) \right) =
\left\langle \alpha _2, u _1 \right\rangle - \left\langle \alpha _1,
  u _2 \right\rangle,
\]
and the induced Dirac structure may be expressed, in coordinates, by
\begin{align*}
D_{\Delta_G}(g,p) =
\left\{
\left((g,p,\dot{g},\dot{p}),(g,p, \alpha, w) \right) \mid
\dot{g} \in \Delta (g), \;\; w = \dot{g}, \;\; \mbox{and} \;\; \alpha+\dot{p} \in \Delta(g) ^{\circ}
\right\}.
\end{align*}

Writing $\mathbf{d}_DL=(g,p=\partial{L}/\partial{v},-\partial{L}/\partial{g},v)$ and using
the local expressions for the canonical symplectic form and the Dirac differential, the equations of motion $((g,p,\dot{g},\dot{p}),\mathbf{d}_DL(g,p)) \in D_{\Delta_G}(g,p)$ read
\[
\left\langle -\frac{\partial L}{\partial g}, u \right\rangle
+ \left\langle \alpha, v \right\rangle
= \left\langle \alpha, \dot{g} \right\rangle
- \left\langle \dot{p}, u \right\rangle,
\]
for all $u \in \Delta(g)$ and all $\alpha $, where $(u,\alpha)$ are the local representatives of a point in $ T_{(g,p)}T^{\ast}G$. Since this holds for all $u \in \Delta(g)$ and all $\alpha $, it follows 
\begin{equation}\label{ImpLagSys_G}
p = \frac{\partial L}{\partial v }, \quad \dot{g}=v \in \Delta (g), \quad  \dot{p} - \frac{\partial L}{\partial g}
\in \Delta  (g)^{\circ},
\end{equation}
which are the local expressions for the implicit Lagrangian system over $G$.

\paragraph{Left trivialized expressions.}  
Utilizing the left trivializing diffeomorphism 
\[
\lambda:TG \to G \times \mathfrak{g}, \quad  v_{g} \mapsto (g, \eta=g ^{-1} v _g ),
\]
we can trivialize $T ^\ast( TG)$ and $T ^\ast (T ^\ast G)$ as $(G \times \mathfrak{g}) \times (\mathfrak{g}^\ast \oplus \mathfrak{g}^\ast)$ and $(G \times \mathfrak{g}^\ast ) \times (\mathfrak{g}^\ast \oplus \mathfrak{g}) $, respectively. Using these trivializations, the map $ \mathbf{d} L:TG \rightarrow T^*TG$  read
\begin{align*} 
\overline{\mathbf{d} L}: G \times \mathfrak{g} \to (G \times \mathfrak{g}) \times (\mathfrak{g}^{\ast} \oplus \mathfrak{g}^{\ast}), \quad \overline{\mathbf{d} L}(g, \eta ) =\left(g,\eta,g ^{-1} \, \frac{\partial{\bar{L}}}{\partial{g}},\, \frac{\partial{\bar{L}}}{\partial{\eta}} \right),
\end{align*} 
where $\bar L:= L \circ \lambda^{-1}: G \times \mathfrak{g}  \rightarrow  \mathbb{R}$. Further, the map $ \gamma _G: T^*TG \rightarrow T^*T^*G$ is to be
\begin{align*} 
\bar{\gamma}_{G}: (G \times \mathfrak{g}) \times (\mathfrak{g}^\ast \oplus \mathfrak{g}^\ast) \to (G \times \mathfrak{g}^\ast) \times (\mathfrak{g}^\ast \oplus \mathfrak{g}), \quad \bar{\gamma}_{G}(g,\eta,\nu, \mu) = (g,\mu,-\nu,\eta).
\end{align*} 
Then, the trivialization of the Dirac differential $ \mathbf{d} _DL: TG \to T^{\ast}T^{\ast}G$ becomes 
\[
\overline{\mathbf{d}_DL}=\bar{\gamma}_{G} \circ \overline{\mathbf{d}L}: G \times \mathfrak{g}\to (G \times \mathfrak{g}^\ast) \times (\mathfrak{g}^\ast \oplus \mathfrak{g}), \quad 
\overline{\mathbf{d}_DL}(g, \eta )=\left(g,\frac{\partial{\bar{L}}}{\partial{\eta}} ,-g ^{-1} \frac{\partial{\bar{L}}}{\partial{g}},\, \eta \right).
\end{equation*}

Note that $ \bar \gamma _G$ is $G$-invariant and induces the map
\[
(\bar{\gamma}_{G})^{/G}: \mathfrak{g} \times (\mathfrak{g}^\ast \oplus \mathfrak{g}^\ast) \to \mathfrak{g}^\ast \times (\mathfrak{g}^\ast \oplus \mathfrak{g}), \quad (\bar{\gamma}_{G})^{/G}(\eta,\nu, \mu) = (\mu,-\nu,\eta).
\]
If we assume that $L$ is $G$-invariant, with the associated reduced Lagrangian $\ell:\mathfrak{g} \to \mathbb{R}$, then $\overline{ \mathbf{d} L}$ is also $G$-invariant and it induces the quotient map
\[
(\overline{\mathbf{d} L})^{/G}: \mathfrak{g}  \rightarrow  \mathfrak{g} \times (\mathfrak{g}^{\ast} \oplus  \mathfrak{g}^{\ast}),\quad (\overline{\mathbf{d}L})^{/G}( \eta )=\left(\eta,0,\frac{\delta{\ell}}{\delta{\eta}}\right).
\]
This leads to the following definition.

\begin{definition}
The {\bfi reduced Dirac differential} of $\ell: \mathfrak{g}  \rightarrow \mathbb{R}  $ is
\[
\mathbf{d}^{/G}_D\ell:=(\overline{\mathbf{d}_D L})^{/G}=(\bar{\gamma}_{G})^{/G} \circ (\overline{\mathbf{d}L})^{/G}: \mathfrak{g} \rightarrow \mathfrak{g}^{\ast} \times (\mathfrak{g}^{\ast} \oplus \mathfrak{g}),
\]
whose expression is
\begin{equation}\label{diff_L}
\mathbf{d}^{/G}_D\ell(\eta) =\left(\frac{\delta{\ell}}{\delta{\eta}},0,\eta \right).
\end{equation}
\end{definition}

\begin{remark}{\rm Note that the usual differential of the reduced Lagrangian $\ell:\mathfrak{g} \to \mathbb{R}$ is the map
\[
\mathbf{d}\ell: \mathfrak{g} \to \mathfrak{g} \times \mathfrak{g}^{\ast},\quad
\mathbf{d}\ell( \eta )
=\left(\eta, \frac{\delta l}{\delta \eta} \right).
\]
It has therefore a different target space (and, of course, a different expression) from the reduced Dirac differential $\mathbf{d}^{/G}_D\ell:\mathfrak{g} \rightarrow \mathfrak{g}^{\ast} \times (\mathfrak{g}^{\ast} \oplus \mathfrak{g})$.}
\end{remark}

\paragraph{Euler-Poincar\'e-Dirac reduction.} Let us recall the definition of reduction of an implicit Lagrangian system associated with the induced Dirac structure on $T^{\ast}G$.
\begin{definition}[\bf Reduced Lagrange-Dirac dynamical system]\label{EulPoinSusReduction}
Let $(G, \Delta _G,L)$ be a Lagrange-Dirac dynamical system. The equations of motion for the {\bfi reduced Lagrange-Dirac dynamical system} are given by
\begin{equation}\label{condition_reducedImpLagSys_G}
\left( (\mu(t),\xi(t),\dot{\mu}(t)),\mathbf{d}^{/G}_D\ell(\eta(t))\right)  \in D_{\Delta_{G}}^{/G}(\mu(t)).
\end{equation}
In the above, $(\eta(t),\mu(t)) \in \mathfrak{g} \oplus \mathfrak{g}^{\ast}$ is the reduced curve associated to $(g(t), v(t), p(t)) \in TG \oplus T^{\ast}G$ and $(\mu(t),\xi(t),\dot{\mu}(t)) \in \mathfrak{g}^{\ast} \times V$ is the reduced curve associated to $(g(t),p(t),\dot{g}(t),\dot{p}(t))$.
Note also that \eqref{condition_reducedImpLagSys_G} induces the equality of the base points, i.e.  $\mu(t) =  \mathbb{F}{\ell}(\eta(t))$.
\end{definition}

\begin{definition}
Any curve $(\xi(t),\mu(t))$ in $\mathfrak{g} \oplus \mathfrak{g}^{\ast}$ together with $\mu(t) =  \mathbb{F}{\ell}(\eta(t))$ for $\eta(t) \in \mathfrak{g}$ satisfying \eqref{condition_reducedImpLagSys_G} is called a  solution curve of the reduced Lagrange-Dirac dynamical system.
\end{definition}
It follows from equations \eqref{reducedIndDiracAlt}, \eqref{diff_L} and \eqref{condition_reducedImpLagSys_G} that
$(\xi(t),\mu(t))$ is a solution curve of  the reduced Lagrange-Dirac dynamical system if and only if it satisfies
\begin{eqnarray}\label{ImpEPS}
\mu= \frac{\delta {\ell}}{\delta \eta}, \quad
\xi= \eta\in \mathfrak{g}^{\Delta}, \quad
\dot{\mu} - \mathrm{ad}_{\xi}^{\ast}  \mu \in (\mathfrak{g}^{\Delta})^{\circ},
\end{eqnarray}
where $(\mathfrak{g}^{\Delta})^{\circ} \subset \mathfrak{g}^{\ast}$ is the annihilator of the constraint subspace $\mathfrak{g}^{\Delta}$. 
\medskip

The set of equations of motion in equation \eqref{ImpEPS} is the local expression for the reduction of the implicit Lagrange-d'Alembert equations given in equation \eqref{ImpLagSys_G}, which is an implicit analog of  {\it Euler-Poincar\'e-Suslov equations} (see \cite{Bloch2003}). The set of equations of motion in \eqref{ImpEPS} is called {\bfi implicit Euler-Poincar\'e-Suslov equations for nonholonomic mechanics}  (see \cite{YoMa2007b}).
\paragraph{Energy conservation of the reduced Lagrange-Dirac dynamical system.} Let us denote by $(G,\Delta_{G},L)$ a Lagrange-Dirac dynamical system.
Recall that the generalized energy $E$ is defined by $E(g,v,p):=\left(p,v\right>-L(g,v)$, which is $G$-invariant under the action of $G$ since $L$ is $G$--invariant and one easily checks that the momentum function $\langle p(t), v(t) \rangle$ is also $G$-invariant.  Hence one can define the reduced energy $e$ on $\mathfrak{g} \oplus \mathfrak{g}^{\ast}$ by, for $(\eta,\mu) \in \mathfrak{g} \oplus \mathfrak{g}^{\ast}$,
\[
e(\eta,\mu)= \left\langle \mu, \eta \right\rangle -\ell(\eta),
\]
where $\ell=L|\mathfrak{g}$ is the reduced Lagrangian on $\mathfrak{g}$ and $ \left\langle \mu, \eta \right\rangle$ is the reduced momentum function on $\mathfrak{g} \oplus \mathfrak{g}^{\ast}$. We denote by $\mathfrak{g}^{\Delta} \subset \mathfrak{g}$ the reduction of $\Delta_{G} \subset TG$.
Let $(\eta(t),\mu(t))$ be a solution curve of the reduced Lagrange-Dirac dynamical system. Then conservation of energy holds along the solution curve  $(\eta(t),\mu(t))$; that is, the (reduced) energy $e(\eta(t),\mu(t))$ is constant in time, because it follows, by noting that $\mu(t) = (\delta \ell / \delta \eta) (t)$,
\[
\frac{d}{dt}e(\eta,\mu)   = \left\langle \dot{\mu}, \eta \right\rangle
                   + \left\langle \mu, \dot{\eta} \right\rangle
                   -\left\langle \frac{\delta \ell}{\delta \eta}, \dot{\eta} \right\rangle = \left\langle {\mu} - \frac{\delta \ell}{\delta \eta}, \dot{\eta} \right\rangle+
   \left\langle \mathrm{ad}^{\ast}_{\xi}\mu, \eta \right\rangle= \left\langle \mathrm{ad}^{\ast}_{\xi}\mu, \xi \right\rangle,
\]
which vanishes since $\eta= \xi \in\mathfrak{g}^{ \Delta}$ and since $\dot{\mu}- \mathrm{ad}^{\ast}_{\xi}\mu \in (\mathfrak{g}^{ \Delta})^{\circ}$. 

\subsection{Lie-Poisson-Dirac reduction}\label{LiePoisSusRed_Section}
For the case in which a given Lagrangian is regular, or when a Hamiltonian is given, we can develop a Hamiltonian analogue of Euler-Poincar\'e-Dirac reduction. This section gives a reduction procedure for a Hamilton-Dirac dynamical system $(G,\Delta_{G},H)$ associated with an induced Dirac structure on $D_{\Delta_{G}}$,  which we shall call {\it Lie-Poisson-Dirac reduction}. It is also shown that this reduction procedure may be also useful in the analysis of the Suslov problem in nonholonomic mechanics.

\paragraph{Implicit Hamiltonian systems over Lie groups.} Let $G$ be a Lie group, $D_{\Delta_{G}}$ be an induced Dirac structure from a given left invariant distribution $\Delta_{G}$, and $H$ be a $G$-invariant Hamiltonian on $T^{\ast}G$. Recall that equations of motion of  a Hamilton-Dirac dynamical system $(G,\Delta_{G},H)$ are given by
\[
\left( (g(t),p(t),\dot{g}(t),\dot{p}(t)),\mathbf{d}H(g(t),p(t))\right)  \in D_{\Delta_{G}}(g(t),p(t)),
\]
which read, by using the local expressions for the canonical symplectic form and the differential of $H$, namely, $\mathbf{d}H=(g,p,\partial{H}/\partial{g},\partial{H}/\partial{p})$, 
\[
\left\langle \frac{\partial H}{\partial g}, u \right\rangle
+ \left\langle \alpha, \frac{\partial H}{\partial p} \right\rangle
= \left\langle \alpha, \dot{g} \right\rangle
- \left\langle \dot{p}, u \right\rangle,
\]
for all $u \in \Delta(g)$ and all $\alpha $, where $(u,\alpha)$ are the local representatives of a point in $ T_{(g,p)}T^{\ast}G$. It follows that the local expressions for the implicit Hamiltonian system is 
\begin{equation*}
\dot{g}=\frac{\partial H}{\partial p} \in \Delta (g), \quad  \dot{p} + \frac{\partial H}{\partial g}
\in \Delta (g)^{\circ}.
\end{equation*}

\paragraph{Lie-Poisson-Dirac reduction.} Let us develop the reduction procedure for the Hamilton-Dirac dynamical system over a Lie group, which we shall call {\bfi Lie-Poisson-Dirac Reduction}.
The induced Hamiltonian $\bar{H}$ on $G \times \mathfrak{g}^{\ast}$ is $G$--invariant, and it reads 
\[
\bar{H}(g,\mu)=h(\mu),
\]
where $h: \mathfrak{g}^{\ast} \to \mathbb{R}$ is the reduced Hamiltonian as a trivialized expression of $h=H^{/G}=H|{\mathfrak{g}^{\ast}}$. 

\paragraph{Trivialized expressions.}
Employing the trivializing diffeomorphism 
\[
\bar{\lambda}:T^{\ast}G \to G \times \mathfrak{g}^{\ast}, \quad  \alpha _{g} \mapsto (g, \mu=T_{e}^{\ast}L_{g}(\alpha _{g})),
\]
and defining the trivialization $\bar{H}=H \circ \bar{\lambda}^{-1}$, we note that the differential $ \mathbf{d} H: T^*G \rightarrow T^*T^*G$ can be trivialized as the map
\[
\overline{\mathbf{d}H}: G \times \mathfrak{g}^{\ast} \to  (G \times \mathfrak{g}^{\ast}) \times (\mathfrak{g}^{\ast} \oplus \mathfrak{g}),\quad \overline{\mathbf{d}H}( g , \mu )=\left(g,\mu,g ^{-1} \, \frac{\partial{\bar{H}}}{\partial{g}}, \frac{\partial{\bar{H}}}{\partial{\mu}} \right),
\]
Suppose that $H$ is $G$-invariant, with reduced Hamiltonian $h: \mathfrak{g}  ^\ast \rightarrow \mathbb{R}  $. Then the map $\overline{ \mathbf{d} H}$ induces a map $(\mathbf{d}\bar{H})^{/G}$ on the quotient. We define the {\bfi reduction of the differential operator} for $h$ as the map
\begin{equation}\label{red_differential_H} 
\mathbf{d}^{/G}h:=(\overline{\mathbf{d} H})^{/G}: \mathfrak{g}^{\ast} \rightarrow \mathfrak{g}^{\ast} \times (\mathfrak{g}^{\ast} \oplus \mathfrak{g}),\quad \mathbf{d}^{/G}h( \mu )=\left(\mu,0,\frac{\delta{h}}{\delta{\mu}}\right).
\end{equation} 

\begin{definition}[\bf Reduced Hamilton-Dirac dynamical system]\label{LiePoissonSusReduction}
Let $(G,\Delta_{G},H)$ be a Hamilton-Dirac dynamical system. The equations of motion for the {\bfi reduced Hamilton-Dirac dynamical system} may be given by
\begin{equation}\label{condition_reducedImpHamSys_G}
((\mu(t),\xi(t),\dot{\mu}(t)),\mathbf{d}^{/G}h(\mu(t))) \in D_{\Delta_{G}}^{/G}(\mu(t)).
\end{equation}
\end{definition}
\begin{definition}
A curve $(\xi(t),\mu(t))$ in $\mathfrak{g} \oplus \mathfrak{g}^{\ast}$ is a {\bfi solution curve} of the reduced Hamilton-Dirac dynamical system if and only if
it satisfies \eqref{condition_reducedImpHamSys_G}.
\end{definition}

It follows from equations \eqref{reducedIndDiracAlt}, \eqref{red_differential_H} and \eqref{condition_reducedImpHamSys_G} that $(\xi(t),\mu(t)) \in \mathfrak{g} \oplus \mathfrak{g}^{\ast}$ is a solution curve of the reduced Hamilton-Dirac dynamical system if and only if it satisfies
\begin{eqnarray}\label{imp_LPS}
\dot{\mu}- \mathrm{ad}_{\xi}^{\ast}  \mu\in (\mathfrak{g}^{\Delta})^{\circ},\quad \xi= \frac{\delta {h}}{\delta \mu} \in \mathfrak{g}^{\Delta},
\end{eqnarray}
where $(\mathfrak{g}^{\Delta})^{\circ} \subset \mathfrak{g}^{\ast}$ is the annihilator of the constraint subspace $\mathfrak{g}^{\Delta}$.
The set of equations in equation \eqref{imp_LPS} is called {\bfi implicit Lie-Poisson-Suslov equations for nonholonomic mechanics} (see \cite{YoMa2007b}).
%
%%%%%%%%
%%%%%%%

\section{Semidirect product theory}\label{SPT} 

The theory of Lie-Poisson reduction for semidirect products has been systematically developed in \cite{MaRaWe1984}; see also \cite{MaWeRaScSp1983}, and applied to several examples in mechanics such as the heavy top, compressible fluids, and magnetohydrodynamics. One of the main features of this approach is that it allows one to obtain the noncanonical Hamiltonian (or Poisson) structure of the dynamical system by reduction of the canonical symplectic structure on the cotangent bundle of the configuration Lie group. Such a reduction approach to noncanonical Poisson structures has been developed by \cite{MarWein1983} for incompressible perfect fluid motion, based on the earlier work of \cite{Arnol'd1966}.

The Lagrangian side of the {\it semidirect Lie-Poisson reduction} has been developed in \cite{HMR1998a} in connection with several examples in fluid dynamics. We will refer to this reduction theory as \textit{Euler-Poincar\'e reduction with advected parameters} (which is not the same as Euler-Poincar\'e reduction for semidirect products, see Remark \ref{important_remark} later).

\paragraph{Semidirect products.} Let $G$ be a Lie group acting by {\it left} representation on a vector space $V$. We will denote by $v \mapsto  \rho _g(v)$ or simply by $ v \mapsto gv$ the left representation of $g \in G$ on $v \in V$. The {\it semidirect product} $S= G \,\circledS\, V$ is the Cartesian product $S=G \times V$ whose group multiplication is given by 
\[
(g_{1},v_{1}) (g_{2},v_{2})=(g_{1}g_{2},v_{1}+g_{1}v_{2})=(g_{1}g_{2},v_{1}+\rho_{g_{1}} v_{2}).
\]
The identity element is $(e,0)$, where $e$ is the identity in $G$, and the inverse of an element $(g,v)$ is given by $( g ^{-1} , - g ^{-1} v)$. Let $\mathfrak{g}$ be the Lie algebra of $G$. The induced Lie algebra action of $\mathfrak{g}$ on $V$ is given, for $\xi \in \mathfrak{g}$ and $v \in V$, by
\begin{equation*}\label{inf_action_V}
\xi v:=\rho^{\prime}(\xi)v:=\frac{d}{dt}\biggr|_{t=0}\rho_{\mathrm{exp}(\xi t)}v \in V,
\end{equation*} 
where $\rho^{\prime}: \mathfrak{g} \to \mathrm{End}(V)$ is the induced Lie algebra representation. In this paper, we will employ the concatenation notation $\xi v$ most commonly.

The Lie algebra $ \mathfrak{s}$ is the semidirect product $\mathfrak{s}:=\mathfrak{g} \,\circledS\,  V$ of $\mathfrak{g}$ with $V$ endowed with the Lie bracket
\[
\begin{split}
[(\xi_{1},v_{1}),(\xi_{2},v_{2})]
&=([\xi_{1},\xi_{2}], \rho^{\prime}(\xi_{1})v_{2}-\rho^{\prime}(\xi_{2})v_{1})\\
&=([\xi_{1},\xi_{2}], {\xi_{1}}v_{2}-{\xi_{2}}v_{1})=\mathrm{ad}_{(\xi_{1},v_{1})}(\xi_{2},v_{2}),
\end{split}
\]
where $\xi_{1},\xi_{2} \in \mathfrak{g}$ and $v_{1}, v_{2} \in V$.

\medskip

The adjoint action of $(g,v) \in S$ on $(\xi, u) \in\mathfrak{s}$ is given by
\[
\mathrm{Ad}_{(g,v)}(\xi, u)=\left( \mathrm{Ad}_{g} \xi, \rho_{g} u-\rho^{\prime}(\mathrm{Ad}_{g} \xi)u \right),
\]
and the coadjoint action of $(g,v) \in S$ on $(\mu, a) \in \mathfrak{s}^{\ast}$ reads
\[
\mathrm{Ad}^{\ast}_{(g,v)^{-1}}(\mu, a)=\left( \mathrm{Ad}^{\ast}_{g^{-1}} \mu +(\rho^{\prime}_{v})^{\ast}  \rho ^*_{ g^{-1} }a, \rho ^*_{ g^{-1} } a\right) .
\]
In the above,  $\rho_{v}^{\prime}: \mathfrak{g} \to V$ denotes the linear map defined by $\rho_{v}^{\prime}(\xi):=\rho^{\prime}(\xi)v$, the map $(\rho_{v}^{\prime})^{\ast}: V^{\ast}\to \mathfrak{g}^{\ast}$ denotes its dual map,  defined by
\[
\left< (\rho^{\prime}_{v})^{\ast}(a),\xi \right>=\left< a, \rho^{\prime}_{v}(\xi) \right>=\left< a, \xi v\right>,\;\; \text{for}\;\;  \xi \in \mathfrak{g}  , \;\; v \in V, \;\; a \in V ^\ast ,
\]
and $ \rho ^*_{ g^{-1} }: V^\ast  \rightarrow V^\ast $ is the dual map to the inverse representation $ \rho _{ g ^{-1} }$.  We will denote it simply by concatenation as $ ga:= \rho ^\ast _{ g ^{-1} }a$.

\paragraph{Diamond operator.} Following \cite{HMR1998a}, we will use the {\it diamond} notation  $\diamond: V \times V^{\ast} \to \mathfrak{g}^{\ast}$, defined by $(v,a) \mapsto v \diamond a:= (\rho^{\prime}_{v})^{\ast}(a)$. We thus have
\[
\left<v \diamond a, \xi \right>=\left<a, \xi v \right>=-\left<\xi a, v \right>,\;\; \text{for}\;\;  \xi \in \mathfrak{g}  , \;\; v \in V, \;\; a \in V ^\ast ,
\]
where
\[
\xi a :=\frac{d}{dt}\biggr|_{t=0}\rho_{\mathrm{exp}(-t\xi)}^\ast a =\frac{d}{dt}\biggr|_{t=0} \operatorname{exp}(t \xi ) a \in V^\ast .
\]

\paragraph{Right actions.}
When the Lie group acts on the vector space $V$ by a \textit{right} representation $ v \mapsto \rho _g v=vg$, there are some changes in the formulas. The group multiplication of $S= G \,\circledS\,  V$ is given by
\[
(g_{1},v_{1}) (g_{2},v_{2})=(g_{1}g_{2},v_{2}+\rho_{g_{2}}v_{1})=(g_{1}g_{2},v_{2}+v_{1}g_{2})
\]
and the Lie bracket on $\mathfrak{s}=\mathfrak{g}\,\circledS\,  V$ reads
\[
\begin{split}
\mathrm{ad}_{(\xi_{1},v_{1})}(\xi_{2},v_{2})=[(\xi_{1},v_{1}),(\xi_{2},v_{2})]=
([\xi_{1},\xi_{2}], v_{1}\xi_{2}-v_{2}\xi_{1}),
\end{split}
\]
where $ v \mapsto v \xi $ denotes the induced Lie algebra representation on $V$.
The adjoint action of $(g,v) \in S$ on $( \xi , u) \in \mathfrak{s}$ is given by\[
\operatorname{Ad}_{(g,v)}(\xi,u)=(\operatorname{Ad}_g \xi, (u+v\xi)g^{-1})
\]
while the coadjoint action of $(g,v) \in S$ on $(\mu, a) \in \mathfrak{s}^{\ast}$ is
\[
\operatorname{Ad}^{\ast}_{(g,v)^{-1} }(\mu, a) =(\operatorname{Ad}^{\ast}_{g ^{-1} } \mu+(vg^{-1})\diamond (ag^{-1}),ag^{-1}),
\]
where $a \mapsto ag$ denotes the induced right representation of $G$ on $V ^\ast $.  As before, we have used the diamond notation
\[
\left< v \diamond a, \xi \right>=\left< a, v\xi \right>=-\left< a \xi, v\right>,\;\; \text{for}\;\;  \xi \in \mathfrak{g}  , \;\; v \in V, \;\; a \in V ^\ast.
\]
Note that, following our conventions, the adjoint and coadjoint actions are {\it left} representations.

\subsection{Euler-Poincar\'e reduction with advected parameters}\label{EP_Adv} 

In this section we recall from \cite{HMR1998a} the theory of Euler-Poincar\'e reduction with advected parameters. We use the same notation as above concerning the various Lie group and Lie algebra actions arising in the formulas.

\begin{itemize}
\item  Assume that we have a function $L:TG\times V^{\ast}\rightarrow
\mathbb{R}$ which is \textit{left\/} $G$-invariant under the action 
$(v_h,a) \mapsto (gv_h,ga)$, where $g, h \in G$, 
$v _h\in TG$, and $a \in V ^\ast$.
\item In particular, given $a_0\in V^{\ast}$, we define the Lagrangian
$L_{a _0 }:TG\rightarrow\mathbb{R}$ by $L_{a_0}(v_g):=L(v_g,a_0)$. Then $L_{a_0}$
is left invariant under the lift to $TG$ of the left action of $G_{a _0 }$ on
$G$, where $G_{a_0}: = \{g \in G \mid g a _0 = a _0 \}$ is the isotropy group of $a_0$ with respect to the linear action of $G$ on $V ^\ast $.
\item Define the reduced Lagrangian $\ell:\mathfrak{g}\times V^{\ast}\rightarrow\mathbb{R}$ by $\ell: =  L|_{ \mathfrak{g}\times V ^\ast}$. Left $G$-invariance of $L$ yields the formula
\[
\ell(g^{-1}v_g,g^{-1}a_{0})=L(v_g,a_{0}),
\]
for all $g \in G $, $v _g\in T _gG $, and $a_{0} \in V ^\ast$.
\item For a curve $g(t)\in G$ with $g(0)=e$, let $\xi(t):=g(t)^{-1}\dot{g}(t)\in \mathfrak{g}$ 
and define the curve $a(t)\in V^\ast$ as the unique solution of the 
following linear differential equation with time dependent coefficients
\begin{equation}\label{DEq}
\dot{a}+\xi a=0,
\end{equation}
with initial condition $a_0$.
The solution of \eqref{DEq} can then be written as
\begin{equation*}\label{a-evol}
a(t)=g(t)^{-1}a_0.
\end{equation*}
\end{itemize}

\begin{theorem}\label{EPSD} With the preceding notation, the following are equivalent:
\begin{itemize}
\item[\bf{(i)}] With $a_0\in V^\ast$ held fixed, Hamilton's variational principle
\begin{equation*}\label{Hamilton_principle}
\delta\int_{t_0}^{t_1}L_{a_{0}}(g,\dot{g})\mbox{d}t=0,
\end{equation*}
holds, for variations $\delta g(t)$ of $g(t)$ vanishing at the endpoints.
\item[\bf{(ii)}] The curve $g(t)$ satisfies the Euler-Lagrange equations for $L_{a_0}$ on
$G$.
\item[\bf{(iii)}] The constrained variational principle
\begin{equation*}\label{Euler-Poincare_principle}
\delta\int_{t_0}^{t_1}\ell(\xi,a)\mbox{d}t=0,
\end{equation*}
holds on $\mathfrak{g}\times V^{\ast}$, upon using variations of the form
\[
\delta\xi=\frac{\partial\eta}{\partial t}+[\xi,\eta],\quad \delta
a=-\eta a,
\]
where $\eta(t)\in\mathfrak{g}$ vanishes at the endpoints.
\item[\bf{(iv)}] The Euler-Poincar\'e equations with advected parameter hold on
$\mathfrak{g}\times
V^{\ast}$:
\begin{equation}\label{EP}
\frac{\partial}{\partial t}\frac{\delta
\ell}{\delta\xi}=\operatorname{ad}^{\ast}_\xi\frac{\delta \ell}{\delta\xi}+\frac{\delta
\ell}{\delta a}\diamond a.
\end{equation}
\end{itemize}
\end{theorem}

We refer to \cite{HMR1998a} for the proof and several applications. Note that the $G$-invariant function $L:TG \times V ^\ast \rightarrow \mathbb{R}  $ is not the Lagrangian of system. The Lagrangian $L_{a _0 }$ is obtained by fixing a value $a _0 \in V ^\ast $, i.e., $L_{a_0}(v _g ):=L(v _g , a _0 )$ and is only $G_{a _0 }$-invariant.

\subsection{Lie-Poisson reduction on semidirect products}\label{sec_LPSDP} 

We now recall the theory of Lie-Poisson reduction for semidirect product, following \cite{MaRaWe1984}. The setup is given as follows:

\begin{itemize}
\item Assume that we have a function $H: T ^\ast G \times V ^\ast \rightarrow
\mathbb{R}$ which is \textit{left\/} invariant under the action 
$(\alpha_h,a)\mapsto(g\alpha_h ,ga)$, for all $g,h \in G $, 
$\alpha_h\in T ^\ast G $, and $a \in V ^\ast$.
\item In particular, given $a_0\in V^{\ast}$, we define the Hamiltonian $H_{a_0}:T^{\ast}G\rightarrow\mathbb{R}$ by
\[
H_{a_0}(\alpha_g):=H(\alpha_g,a_0).
\]
Then $H_{a_0}$ is left invariant under the cotangent lift to $T^{\ast}G$ of the left action of $G_{a _0 }$ on $G$.
\item Define the reduced Hamiltonian $h:\mathfrak{g}^{\ast}\times V^{\ast}\rightarrow\mathbb{R}$ by $h: = H|_{ \mathfrak{g}^\ast \times V ^\ast}$. Left $G$-invariance of $H$ yields the formula
\[
h(g^{-1}\alpha_g,g^{-1}a_{0})=H(\alpha_g,a_{0}),
\]
for all $g \in G $, $\alpha_g\in T_g ^\ast G$, and $a_{0} \in V ^\ast$.
\end{itemize}

\begin{theorem}\label{LPSD}
For the curves $\alpha(t)\in T^{\ast}_{g(t)}G$ and
$\mu(t):=T^{\ast}L_{g(t)}(\alpha(t))\in\mathfrak{g}^{\ast}$, the following are
equivalent:
\begin{itemize}
\item[\bf{(i)}] The curve $\alpha(t)$ satisfies Hamilton's equations for
$H_{a_0}$ on $T^{\ast}G$.
\item[\bf{(ii)}] The Lie-Poisson equation holds on $\mathfrak{s}^{\ast}$:
\[
\frac{\partial}{\partial t}(\mu,a)=\operatorname{ad}^{\ast}_{\left(\frac{\delta
h}{\delta\mu},\frac{\delta h}{\delta a}\right)}(\mu,a)
=\left(\operatorname{ad}^{\ast}_{\frac {\delta h}{ \delta \mu}}\mu-\frac {\delta
h}{ \delta a}\diamond
a,-\frac {\delta h}{ \delta \mu}a\right),\quad a(0)=a_0,
\]
where $\mathfrak{s}$ is the semidirect product Lie algebra
$\mathfrak{s}=\mathfrak{g}\,\circledS\, V$. The associated Poisson bracket is
the Lie-Poisson bracket on the semidirect product Lie algebra $\mathfrak{s}^{\ast}$,
that is,
\[
\{f,g\}(\mu,a)=-\left\langle\mu,\left[\frac{\delta
f}{\delta\mu},\frac{\delta
g}{\delta\mu}\right]\right\rangle-\left\langle a,\frac{\delta f}{\delta
\mu }\frac{\delta g}{\delta a}-\frac{\delta g}{\delta \mu }\frac{\delta
f}{\delta a}\right\rangle.
\]
\end{itemize}
Like on the Lagrangian side, the evolution of the advected parameter is given by
$a(t)=g(t)^{-1}a_0$.
\end{theorem}

\paragraph{Legendre transform.}
The link with Euler-Poincar\'e reduction for advected parameters recalled in \S\ref{EP_Adv} is the following. Let $L_{ a _0}:TG \rightarrow \mathbb{R}  $ be a given Lagrangian coming from a $G$-invariant function $L: TG \times V ^\ast  \rightarrow \mathbb{R}  $, suppose that the Legendre transformation $\mathbb{F}L_{a_0}$ is
invertible and form the corresponding Hamiltonian
$H_{a_0}=E_{a_0}\circ\mathbb{F}L_{a_0}^{-1}$, where $E_{a_0}$ is the energy of
$L _{a_0}$. Then the function $H: T^\ast G \times V ^\ast
\rightarrow \mathbb{R}$ so defined is $G$-invariant and one can apply
Theorem \ref{LPSD}. At the reduced level, the Hamiltonian $h:\mathfrak{g}^{\ast}\times V^{\ast}\rightarrow\mathbb{R}$ is given in terms of $ \ell$ by
\[
h(\mu,a):=\langle\mu,\xi\rangle-\ell(\xi,a),\quad\mu=\frac{\delta \ell}{\delta\xi}.
\]
Since
\[
\frac{\delta h}{\delta\mu}=\xi\quad\text{and}\quad\frac{\delta h}{\delta
a}=-\frac{\delta \ell}{\delta a},
\]
we see that the Lie-Poisson equations for $h$ on $\mathfrak{s}^\ast $ are
equivalent to the Euler-Poincar\'e equations with advected quantities \eqref{EP} for $\ell$ together with the
advection equation $\dot{a}+a\xi=0$.

\begin{remark}[\textbf{Links with reduction by stages}]\label{Link_RBS}{\rm It is surprising, at first glance, that at the reduced level, on the Hamiltonian side, we recover the ordinary Lie-Poisson equations for the semidirect product $S=G \,\circledS\, V$, whereas we started with a $G$-invariant Hamiltonian $H: T^{\ast}G \times V^{\ast} \rightarrow \mathbb{R}  $ and not a $S$-invariant Hamiltonian on $T^{\ast}S$. This can be easily justified by the process of \textit{Poisson reduction by stages}, \cite{MaMiOrPeRa2007}.
%
%\todo{Hiro: fixed the following equation.}
%
Indeed, the $G$-action on $T ^\ast G \times V ^\ast$ can be seen as the action induced by 
the cotangent lift of left translation on $S$, given by
\begin{align*}\label{action_on_T*S}
T^{\ast}L_{{(g,v)} ^{-1} }(\alpha_h,(u,a)):=\left(g\alpha_h ,(v+gu,ga)\right).
\end{align*}
Thus, we can think of the Hamiltonian $H:T^{\ast}G\times V^{\ast}\to\mathbb{R}$
as being the first stage Poisson reduction of a $S$-invariant Hamiltonian
$\bar{H}:T^{\ast}S\to\mathbb{R}$ by the normal subgroup $\{e\} \times V \subset S$, since
$(T ^\ast S)/(\{e\} \times V)  \cong T ^\ast G \times V^\ast$. Then the second stage of the Poisson reduction applied to $H:T^{\ast}G \times V ^\ast \rightarrow \mathbb{R}$ yields the reduced Hamiltonian $h$ and the reduced Hamilton equations of motion on $ \mathfrak{g}  ^\ast \times V ^\ast $. By the Poisson reduction by stages theorem, we know that the two-stages reduction coincides with the one step reduction of $\bar H$ by $S$ and, therefore, the reduced equations of motion on $ \mathfrak{g}  ^\ast \times V ^\ast $ must recover the Lie-Poisson equations on the semidirect product $ \mathfrak{s}$.

Note that the Hamiltonian $H _{ a _0 }$ does not appear in the process of Poisson reduction by stages. In order to obtain it naturally from a reduction approach, it is necessary to use \textit{symplectic reduction by stages}. Consider the semidirect product
Lie group $S=G\,\circledS\,V$ acting by right translation on its cotangent
bundle $T^{\ast}S$. An equivariant momentum map relative to the canonical symplectic
form is given by
\[
\mathbf{J}_L(\alpha_f,(u,a))=T^{\ast}R_{(f,u)}(\alpha_f,(u,a))=(T^{\ast}_eR_f(\alpha_f)+u\diamond
a,a).
\]
Since $V$ is a closed normal subgroup of $S$, it also acts on $T^{\ast}S$ and has a
momentum map $\mathbf{J}_V : T^{\ast}S\to V^{\ast}$ given by
\[
\mathbf{J}_V(\alpha_f,(u,a))=a.
\]
Reducing $T^{\ast}S$ by $V$ at the value $a_0 $ we get the first stage reduced space
$(T^{\ast}S)_{a_0 }=\mathbf{J}_V^{-1}(a_0 )/V$. One observes that this space is
symplectically diffeomorphic to the canonical symplectic manifold
$(T^{\ast}G,\Omega)$, and that the $S$-invariant Hamiltonian $\bar H$ on $T^{\ast}S$ induces exactly the Hamiltonian $H_{ a _0 }$ on $T^{\ast}G \simeq (T^{\ast}S)_{a_0 }$, as desired. For completeness, we now comment on the second stage symplectic reduction.
The isotropy subgroup $G_{a_0 }$, consisting of
elements of $G$ that leave the point $a_0 $ fixed, acts freely and properly on the first stage symplectic reduced space 
$(T^{\ast}S)_{a_0 }$ and admits an equivariant momentum map $\mathbf{J}_{a_0 } :
(T^{\ast}S)_{a_0 }\to\mathfrak{g}^{\ast}_{a_0 }$ induced from $ \mathbf{J}  _L$, where $\mathfrak{g}_{a_0 }$ is the Lie algebra of
$G_{a_0 }$. Reducing $(T^{\ast}S)_{a_0 }$ at the point $\mu_{a_0 }:=\mu|\mathfrak{g}_{a_0 }$, we get the
second symplectic reduced space
$\left((T^{\ast}S)_{a_0 }\right)_{\mu_{a_0 }}=\mathbf{J}_a^{-1}(\mu_{a_0 })/(G_{a_0 })_{\mu_{a_0 }}$. From the Reduction by Stages Theorem (\cite{MaMiOrPeRa2007}), the second symplectic reduced space
$\left((T^{\ast}S)_{a_0}\right)_{\mu_{a_0 }}$  is symplectically diffeomorphic to the symplectic reduced
space $(T^{\ast}S)_{(\mu,a_0 )}=\mathbf{J}_R^{-1}(\mu,a)/G_{(\mu,a_0 )}$ obtained by symplectic reduction by $S$ at the point $(\mu,a_0 )\in\mathfrak{s}^{\ast}$. The latter is symplectically diffeomorphic to the coadjoint orbit
$\left(\mathcal{O}_{(\mu,a_0 )},\omega_{(\mu,a_0 )}\right)$ of $S$ endowed with its orbit
symplectic form.}
\end{remark}

\begin{remark}[\textbf{Cautionary remark}]\label{important_remark}{\rm  It is important to mention here that the Hamiltonian reduction mentioned above is literally an ordinary Lie-Poisson reduction in the special case when the Lie group is a semidirect product. This is not the case for the Lagrangian side, since it does not coincide with the Euler-Poincar\'e reduction in the special case when the Lie group is a semidirect product. This is why we prefer to call it Euler-Poincar\'e reduction with advected parameters.
This comment is important for the present paper for two reasons. First, we will see that a crucial example of nonholonomic system arises when the Lie group itself is a semidirect product in addition to have advected quantities. So there are two semidirect products arising in this example and, on the Lagrangian side, they have completely different roles; second, we will see that the Dirac reduction processes involved in the Lagrangian and Hamiltonian side are completely different.}
\end{remark}

%%%%%%%
%%%%%%%
%%%%%%%

\section{Variational framework}\label{VPAP} 

The goal of this section is to derive the implicit Euler-Poincar\'e-Suslov equations with advected parameters together with the associated variational structures. We first consider the unconstrained case via the Hamilton-Pontryagin principle, and then explore the nonholonomic case given by a parameter dependent constraint $ \Delta _G ^{a _0 } \subset TG$, by using the Lagrange-d'Alembert-Pontryagin principle. Second we study the case of a nonholonomic system in which a configuration Lie group is a semidirect product and with special classes of constraints arising in several examples. In particular, it is shown that the results of \cite{Sch2002} can be incorporated into the classes of nonholonomic constraints. Finally, we also explore the Hamiltonian side to develop the corresponding implicit Lie-Poisson-Suslov equations with advected parameters, which will be shown to be different from the implicit Lie-Poisson-Suslov equations on semidirect products. We also mention how the theory can be extended to the case in which the parameters are acted on by more general actions, such as affine actions.

\subsection{The unconstrained case}\label{HP_unconstrained} 

Let us consider the Hamilton-Pontryagin principle and its associated reduced version for the case of a parameter dependent Lagrangian $L_{ a _0 }(g, \dot g)=L(g, \dot g, a _0 ):TG \times V^{\ast} \to \mathbb{R}$. Using the same notations and assumptions made in the beginning of \S\ref{EP_Adv}, we state the following theorem.

\begin{theorem} \label{theorem_HP_semidirect}
With the above notations, the following statements are equivalent.
\begin{itemize}
\item[\bf{(i)}]
With $a_{0}$ fixed, the Hamilton-Pontryagin principle
\begin{equation}\label{HP-ActInt}
\delta \int_{t_1}^{t_2}
\left\{ L_{a_{0}}(g(t),v(t)) + \left\langle p(t),  \dot{g}(t)-v(t) \right\rangle \right\} \, dt =0
\end{equation}
holds for variations $(\delta{g}(t), \delta{v}(t),\delta{p}(t))$ of $(g(t),v(t),p(t))$ in $TG \oplus T^{\ast}G$ such that $\delta g$ vanishes at the endpoints.
\item[\bf{(ii)}]
The curve $g(t), \;t \in [t_{1},t_{2}]$ satisfies the implicit Euler-Lagrange equations for $L_{a_{0}}$ on $TG \oplus T^{\ast}G$:
\begin{equation}\label{ImpELEq}
p=\frac{\partial L_{a_{0}}}{\partial v},  \quad \dot{g}=v, \quad   \dot{p}=\frac{\partial L_{a_{0}}}{\partial g}.
\end{equation}
\item[\bf{(iii)}]
The reduced Hamilton-Pontryagin principle
\begin{equation}\label{ConstHP-ActInt}
\delta \int_{t_1}^{t_2} \left\{ \ell(\eta(t), a(t))+ 
\left\langle \mu(t), \xi(t) - \eta(t) \right\rangle \right\} ~dt=0,
\end{equation}
holds under arbitrary variations of $\eta=g^{-1}v ,  \mu=g^{-1}p$ and variations of the form
\[
\delta{\xi}=\frac{\partial \zeta}{\partial t}+[\xi, \zeta], \qquad \delta{a}=-\zeta a,
\]
where $\xi=g^{-1}\dot{g}$ and $\zeta=g^{-1}\delta{g} \in \mathfrak{g}$ vanishes at the endpoints.

\item[\bf{(iv)}]
The implicit Euler-Poincar\'e equations hold on $(\mathfrak{g} \oplus \mathfrak{g}^{\ast}) \times V^{\ast}$:
\begin{equation}\label{implicit_EP}
\mu =  \frac{\delta \ell}{\delta \eta}, \quad \xi = \eta, \quad \dot{\mu} =  \operatorname{ad}_{\xi}^{\;\ast} \mu + \frac{\delta\ell}{\delta{a}} \diamond a.
\end{equation} 
\end{itemize}
\end{theorem}

\begin{proof}
The equivalence of $\bf{(i)}$ and $\bf{(ii)}$ is obvious from the following computations.
It follows from the Hamilton-Pontryagin principle that
\begin{equation*}
\begin{split}
&\delta \int_{t_1}^{t_2}  \left\{ L_{a_{0}}(g(t),v(t)) +\left\langle p(t),  \dot{g}(t)-v(t) \right\rangle
\right\} dt\\ 
&= \int_{t_1}^{t_2} \left\{
\left(-\dot{p}+\frac{\partial L_{a_{0}}}{\partial g}  \right) \delta{g} +\left( -p+\frac{\partial L_{a_{0}}}{\partial v}\right)
\delta{v}+\left(\dot{g}-v  \right) \delta{p} 
\right\} dt +p\,\delta{g}\bigg|_{t_1}^{t_2}=0,
\end{split}
\end{equation*}
holds for all $\delta{g},\delta{v}$ and $\delta{p}$. Keeping the endpoints $g(t_{1})$ and $g(t_{2})$ of $g(t)$ fixed, we obtain the implicit Euler--Lagrange equations on $TG \oplus T^{\ast}G$ for the Lagrangian $L_{a_{0}}: TG \to \mathbb{R}$ as
\begin{equation*}
\begin{split}
p=\frac{\partial L_{a_{0}}}{\partial v},  \quad \dot{g}=v, \quad   \dot{p}=\frac{\partial L_{a_{0}}}{\partial g}.
\end{split}
\end{equation*}

Next, let us see the equivalence of $\bf{(iii)}$ and $\bf{(vi)}$ as follows. The variation of the action integral is given by
\begin{align*} \label{varation_f}
\delta & \int_{t_1}^{t_2} \left\{ \ell(\eta, a)+ 
\left\langle \mu, \xi - \eta \right\rangle \right\} ~dt\\
&=\int_{t_1}^{t_2} \left\{ \left\langle \frac{\delta \ell}{\delta \eta},  \delta \eta\right\rangle +\left\langle \delta a, \frac{\delta \ell}{\delta a} \right\rangle  
+\left\langle \delta{\mu}, \xi-\eta\right\rangle + \left\langle \mu, \delta \xi - \delta \eta\right\rangle  \right\}~dt  \nonumber \\ 
& =   \int_{t_1}^{t_2} \left\{ \left\langle \frac{\delta\ell}{\delta \eta} -\mu, \delta \eta\right\rangle -\left\langle \zeta a, \frac{\delta \ell}{\delta a}\right\rangle 
+ \left\langle \delta{\mu}, \xi-\eta \right\rangle  + \left\langle  \mu, \dot{\zeta} + \mbox{ad}_{\xi}  \zeta\right\rangle  \right\} ~dt \nonumber \\
&=  \int_{t_1}^{t_2} \left\{ \left<\frac{\delta \ell}{\delta \eta} -\mu, \delta \eta \right>
+ \left\langle \delta{\mu}, \xi-\eta \right\rangle + \left\langle - \dot{\mu} + \mbox{ad}_{\xi}^{\ast}  \mu + \frac{\delta \ell}{\delta a} \diamond a,  \zeta\right\rangle  \right\} ~dt  + \mu \cdot \zeta\bigg|_{t_1}^{t_2},
\end{align*}
where the variation of $\xi$ is given by $\delta \xi=\dot{\zeta}+ [\xi,\zeta]$, where $\zeta=g^{-1}  \delta g$, so that $\zeta $ is an arbitrary curve in $\mathfrak{g}$ satisfying $\zeta(t_{1})=\zeta(t_{2})=0$. Then, the variation of the action integral vanishes for any $\delta{\eta} \in \mathfrak{g}$, $\zeta \in \mathfrak{g}$ and $\delta{\mu} \in \mathfrak{g}^{\ast}$ if and only if
\begin{equation}\label{Imp_Euler_PoincareEqn}
\begin{split}
\mu =  \frac{\delta \ell}{\delta \eta}, \quad \xi & = \eta, \quad \dot{\mu} =  \mbox{ad}_{\xi}^{*} \mu + \frac{\delta \ell}{\delta a} \diamond a.
\end{split}
\end{equation}

Now, the equivalence of $\bf{(i)}$ and $\bf{(iii)}$ is clear, since we first note that the $G$-invariance of $L:TG \times V^{\ast} \to \mathbb{R}$ and $\left\langle p_{g},  \dot{g}-v \right\rangle$ together with the definition of $a(t)=g^{-1}(t)a_{0}$ imply that the integrands in 
equations \eqref{HP-ActInt} are \eqref{ConstHP-ActInt} equal. All variations $\delta{g}(t) \in TG$ vanishing at the endpoints induce variations of $\delta \xi(t) \in \mathfrak{g}$ of $\xi(t)$ of the form $\delta\xi=\dot{\zeta} + [\xi, \zeta]$ with $\zeta(t) \in \mathfrak{g}$ vanishing at the endpoints. The variations of $a(t)=g^{-1}(t)a_{0}$ is given by
\[
\delta{a}(t)=\delta{g}(t)^{-1}a_{0}=-g(t)^{-1}\delta{g}(t)g(t)^{-1}a_{0}=-\zeta(t)a(t).
\]
Conversely, if the variation of $a(t)$ is defined by $\delta{a}(t)=-\zeta(t) a(t)$, then the variation of $g(t)a(t)=a_{0}$ vanishes, which is consistent with the fact that $ a _0 $ is held fixed in $L_{a_{0}}$.
\end{proof}

\medskip

We shall call the set of equations \eqref{Imp_Euler_PoincareEqn} {\bfi implicit Euler-Poincar\'e equations with advected parameters}  on $(\mathfrak{g} \oplus \mathfrak{g}^{\ast}) \times V^{\ast}$. They are the reduced formulation of the implicit Euler-Lagrange equations  \eqref{ImpELEq} on $TG \oplus T^{\ast}G$ associated to the parameter dependent Lagrangian $L_{a_{0}} : TG \to \mathbb{R}$.

\medskip

\begin{remark}[\textbf{Alternative formulation of the Hamilton-Pontryagin principle}]\label{Hybrid_HP_advected} {\rm Similarly with Remark \ref{Hybrid_HP}, we can consider the following variational principle similar to ${\bf (iii)}$ above, namely
\[
\delta \int_{ t _1 }^{ t _2 } \{ \ell( \xi,a )+ \left\langle \mu , g ^{-1} \dot g - \xi \right\rangle+ \left\langle v, g ^{-1} a _0 - a \right\rangle  \} \,dt=0,
\]
for arbitrary variations $ \delta \xi(t) $, $ \delta \mu (t) $, $ \delta v(t)$, $ \delta a(t)$ of $ \xi (t) $, $ \mu (t) $, $v(t)$, $a(t) $ and variations $\delta g(t)$ of $g(t)$ vanishing at the endpoints. It yields the stationarity conditions
\[
\xi = g ^{-1} \dot g, \quad a= g ^{-1} a _0 , \quad \mu = \frac{\delta \ell}{\delta \xi },\quad v= \frac{\delta \ell}{\delta a},  \quad \dot{\mu} =  \mbox{ad}_{\xi}^{*} \mu + v\diamond a,
\]
which are equivalent to \eqref{implicit_EP}. 
}
\end{remark} 

\subsection{The case of systems with nonholonomic constraints}\label{HP_constrained} 
%
%\todo{Added the period in the following.}
%
As before, we assume that $L:TG \times V ^\ast \rightarrow \mathbb{R}  $ is a $G$-invariant Lagrangian such that $L_{a_0}(g, \dot g):=L(g, \dot g, a_0)$, where $a_0$ is a parameter that can be arbitrarily chosen in $V^{\ast}$.

\paragraph{Invariance assumption on nonholonomic constraints.}In addition, we assume that the dynamics given by the Lagrangian $L_{a_0}$ is constrained by a distribution $\Delta_G^{ a _0 }\subset TG$ that depends on $a _0 $. Suppose that the distribution has the following invariance property:
\begin{equation}\label{invariance_assumption}
\Delta _G (hg, ha_0)=h \Delta _G (g, a _0), \quad \text{for all $h\in G$},
\end{equation} 
where we employed the notation $\Delta _G^{a_0}(g)= \Delta _G (g, a_0)\subset T_gG$. By the invariance assumption, the family of distributions $ \Delta _G^{ a _0 }$ is completely determined by the family
\begin{equation}\label{def_g_delta}
\mathfrak{g}  ^{ \Delta }(a):= \Delta _G (e,a)\subset \mathfrak{g} 
\end{equation} 
of vector subspaces of $ \mathfrak{g}  $, parametrized by $a=g^{-1}a_{0}\in V^{\ast}$. As will be shown later, the $G$-invariance assumption \eqref{invariance_assumption} is verified in many important examples, such as Chaplygin's ball and Euler's disk.

\paragraph{Reduction theorem.} We now formulate the generalization of Theorem \ref{theorem_HP_semidirect} to the case with nonholonomic constraints.

\begin{theorem}\label{thm_HP_constrained} With the above notations, the following statements are equivalent.
\begin{itemize}
\item[\bf{(i)}]
With $a_{0}$ fixed, the Lagrange-d'Alembert-Pontryagin principle
\begin{equation*}\label{HP-ActInt_constrained}
\delta \int_{t_1}^{t_2}
\left\{ L_{a_{0}}(g(t),v(t)) + \left\langle p(t),  \dot{g}(t)-v(t) \right\rangle \right\} \, dt =0
\end{equation*}
holds for variations $(\delta{g}(t), \delta{v}(t),\delta{p}(t))$ of $(g(t),v(t),p(t))$ in $TG \oplus T^{\ast}G$, where $ \delta g$ vanishes at the endpoints, and with $ \delta g\in \Delta _G^{a_0}(g) $ and with the constraint $ v\in \Delta _G^{a_0 }(g)$ .
\item[\bf{(ii)}]
The curve $g(t)$ satisfies the implicit Lagrange-d'Alembert equations for $L_{a_{0}}$ on $TG \oplus T^{\ast}G$
\begin{equation}\label{ImpELEq_constrained}
p=\frac{\partial L_{a_{0}}}{\partial v},  \quad \dot{g}=v\in \Delta ^{a_0}_G(g), \quad   \dot{p}-\frac{\partial L_{a_{0}}}{\partial g}\in  \Delta ^{a_0}_{G}(g) ^\circ.
\end{equation}
\item[\bf{(iii)}]
The reduced Lagrange-d'Alembert-Pontryagin principle
\begin{equation}\label{ConstHP-ActInt_constrained}
\delta \int_{t_1}^{t_2} \left\{ \ell(\eta(t), a(t))+ 
\left\langle \mu(t), \xi(t) - \eta(t) \right\rangle \right\} ~dt=0
\end{equation}
holds, under arbitrary variations of $ \eta , \mu $ and variations of the form
\[
\delta{\xi}=\frac{\partial \zeta}{\partial t}+[\xi, \zeta], \qquad \delta{a}=-\zeta a,
\]
where $\zeta \in \mathfrak{g}^{ \Delta}(a)$ vanishes at the endpoints and $ \eta $ verifies the constraint $ \eta \in \mathfrak{g}^{ \Delta}(a)$.

\item[\bf{(iv)}]
The implicit Euler-Poincar\'e-Suslov equations with advected parameters hold on $(\mathfrak{g} \oplus \mathfrak{g}^{\ast}) \times V^{\ast}$:
\begin{equation}\label{implicit_EP_nonholonomic}
\mu =  \frac{\delta \ell}{\delta \eta}, \quad \xi = \eta\in\mathfrak{g}^{ \Delta}(a) , \quad \dot{\mu} - \operatorname{ad}_{\xi}^{\;\ast} \mu - \frac{\delta\ell}{\delta{a}} \diamond a\in \left( \mathfrak{g}^{ \Delta}(a)\right) ^\circ.
\end{equation} 
\end{itemize}
\end{theorem}
\begin{proof}
These equivalences follow exactly in the same way as in the preceding theorem, but by taking account of the constraints $ \delta g\in \Delta _G^{ a _0 } (g)$, $v\in  \Delta _G^{ a _0 } (g)$ and $\zeta \in \mathfrak{g}^{ \Delta}(a)$  and $ \eta \in \mathfrak{g}^{ \Delta}(a)$.
\end{proof}

Note that the implicit Euler-Poincar\'e-Suslov equations with advected parameters on $(\mathfrak{g} \oplus \mathfrak{g}^{\ast}) \times V^{\ast}$ in \eqref{implicit_EP_nonholonomic} are the reduced formulation of the implicit Lagrange-d'Alembert equations  \eqref{ImpELEq_constrained} on $TG \oplus T^{\ast}G$ associated to the parameter dependent Lagrangian $L_{a_{0}} : TG \to \mathbb{R}$ and the nonholonomic constraint $ \Delta _G ^{ a _0 }$.
\medskip

\begin{remark}\label{Hybrid_HP_advected_constraints}{\rm The alternative variational structure mentioned in Remark \ref{Hybrid_HP_advected} can be easily generalized to the case with constraints by imposing $\dot g \in \Delta _G^{a _0 } (g) $ and $\delta  g \in \Delta _G^{a _0 }(g)$,
i.e. $ g ^{-1} \dot g, g ^{-1} \delta  g \in \mathfrak{g}  ^\Delta (a)$, in which case we get the stationarity conditions
\[
\xi = g ^{-1} \dot g\in \mathfrak{g}  ^\Delta (a), \quad a= g ^{-1} a _0 , \quad \mu = \frac{\delta \ell}{\delta \xi },\quad v= \frac{\delta \ell}{\delta a},  \quad \dot{\mu} - \mbox{ad}_{\xi}^{*} \mu -v\diamond a\in (\mathfrak{g}  ^\Delta (a))^\circ. 
\]
}
\end{remark} 
\begin{remark}\rm
Note that when there is no advected quantities, we have a $G$-invariant Lagrangian $L: TG \rightarrow \mathbb{R}$ and the constraint is simply given by $\Delta _G \subset TG$. It follows from the invariance assumption \eqref{invariance_assumption} that the distribution $ \Delta _G$ is $G$-invariant. In this case, \eqref{ConstHP-ActInt_constrained} and \eqref{implicit_EP_nonholonomic} in Theorem \ref{thm_HP_constrained} recover \eqref{RedLagDAPPrin} and \eqref{ImpEulPoinSusEqn}. 
\end{remark}

\subsection{Rolling ball type constraints on semidirect products}\label{subsec_the_sdp_case}

A situation frequently encountered in examples is the case where $G$ itself is a semidirect product of a Lie group $K$ and the vector space $V$; i.e., the Lie group $G$ in the above theory is given by $G= K \,\circledS\, V$. The Lie algebra of $G$ is the semidirect product $\mathfrak{g}=\mathfrak{k} \,\circledS\,
 V$.
 In this particular situation, the implicit Euler-Poincar\'e-Suslov equations with advected parameters \eqref{implicit_EP_nonholonomic} acquire a more specific formulation well-suited for the study of the rigid bodies with nonholonomic constraints as in \cite{Sch2002}.

In this case, the Lagrangian is $L:TG \times V^{\ast} \rightarrow \mathbb{R}$; $L=L(k, \dot k, x, \dot x, a_0)$ and the constraint is
\[
\Delta _{G}^{a_0}\subset T{G}, \quad \Delta _{G}^{a_0}(k,x)= \Delta _{G}(k,x,a_0)\in T_{(k,x)}G.
\]
Since the group multiplication on $G$ reads $(h,z)(k,x)=(hk,z+hx)$, the invariance property reads
\begin{equation}\label{inv_prop_G}
\Delta _G\left( hk,z+hx, h a_0\right) =TL_{(h,z)}\left(  \Delta _G(k,x,a_0)\right).
\end{equation} 
The family of distributions $ \Delta _G^{ a _0 }\subset TG$ is completely determined by the family of vector subspaces of $ \mathfrak{g} $,  parametrized by $a\in V^{\ast}$ as
\begin{equation}\label{g_Delta_a} 
\mathfrak{g}  ^{ \Delta }(a):= \Delta _G (e,0,a)\subset \mathfrak{g}. 
\end{equation}  

\paragraph{Particular class of constraints.} Following \cite{Sch2002}, we will consider three particular cases of such constraints, going from a basic situation to more general constraints:
\begin{itemize}
\item[\bf (I)]
{\bf The basic case:} The parameter dependent distribution is given by
\begin{equation}\label{case_I}
\Delta _G \left( k,x,a _0 \right) =\left \{( k, x,  \dot k,\dot x)\in T_{(k,x)}G\mid  \dot x= \dot k \left( k^{-1} a _0\right) ^\sharp \right\},
\end{equation} 
where $ \sharp: V^{\ast} \rightarrow V$ is the sharp map associated to a given inner product on $V$, and where $\dot k \left( k^{-1} a _0\right) ^\sharp$ denotes the induced action of $\dot k \in TK$ on $ \left( k^{-1} a_0\right) ^\sharp\in V$.
\item[\bf (II)]
%\underline{The intermediate case:}
{\bf The intermediate case:}  The parameter dependent distribution is given by
\begin{equation}\label{case_II}
\Delta _G \left( k,x,a _0 \right) =\left \{( k, x, \dot k, \dot x)\in T_{(k,x)}G \mid  \dot x= \dot k \phi (k^{-1} a _0) \right\},
\end{equation} 
where $ \phi  :V^{\ast} \rightarrow V$ is an arbitrary smooth function, and $\dot k\phi (a) $ denotes the induced action of $ \dot k \in TK$ on $\phi(a)\in V$.
\item[\bf (III)]
%\underline{The general case:}
{\bf The general case:}  The parameter dependent distribution is given by
\begin{equation}\label{case_III}
\Delta _G \left( k,x,a _0 \right) =\left \{( k, x, \dot k, \dot x)\in T_{(k,x)}G \mid  \dot x= A(k,x,a _0)\cdot \dot k \right\},
\end{equation} 
where $A(k,x,a_0):T_kK \rightarrow V$ is linear and satisfies the invariance property
\[
A(hk,z+hx,h a_0)\cdot h \dot k=h \left( A(k,x,a_0)\cdot \dot k\right) , \quad \text{for all} \quad (h,z)\in G. 
\]
\end{itemize}

One observes that the invariance property \eqref{inv_prop_G} is verified in each cases, and that one recovers \eqref{case_I} and \eqref{case_II} from \eqref{case_III}, by taking $A(k,x,a _0 ) \cdot \dot k:= \dot k\left(k^{-1} a_0\right) ^\sharp$ and $A(k,x,a _0 ) \cdot \dot k:= \dot k \phi(k^{-1}  a_0)$, respectively. 
The vector subspaces $ \mathfrak{g} ^ \Delta (a)$ defined in \eqref{g_Delta_a} are, respectively, given by
\begin{align}
{\bf (I) \;\,\quad} \mathfrak{g} ^ \Delta (a)&=\{( \xi  , X)\in \mathfrak{g} \mid X= \xi a^\sharp\},\notag \\%\label{case_I_red}\\
{\bf (II)\; \quad} \mathfrak{g} ^ \Delta (a)&=\{( \xi  , X)\in \mathfrak{g} \mid X= \xi \phi  (a)\}, \notag \\%\label{case_II_red}\\
{\bf (III) \quad} \mathfrak{g} ^ \Delta (a)&=\{( \xi  , X)\in \mathfrak{g} \mid X= \alpha(a) \cdot \xi \},\label{case_III_red}
\end{align} 
where we denoted $\alpha(a):= A(e,0,a)$.
Note that our convention slightly differs from the one used in \cite{Sch2002} since he employs $G= K \,\circledS\, V$ and $a \in V $, whereas we use $G= K \,\circledS\, V$ and $ a \in V ^\ast $.

\paragraph{Reduction theorem.} Since the invariance assumption is verified, we can now apply Theorem \ref{thm_HP_constrained} to this special case when the Lie group $G$ is the semidirect product $G= K \,\circledS\, V$. It yields the Lagrange-d'Alembert-Pontryagin principle for these special classes of constraints. We will apply this Theorem to several examples in Section \ref{Examples} later. 
We denote by $(k,x,v,w)$ and $(k,x,p,\pi )$ the variables in $TG$ and $T^{\ast}G$, respectively, and we indicate the reduced variables by
\begin{align*} 
( \xi  , X)&:= (k,x) ^{-1} (k,x, \dot k, \dot x)= ( k^{-1} \dot k, k^{-1} \dot x)\in \mathfrak{g},\\
(\eta , Y)&:= (k,x) ^{-1} (k,x,v,w)= ( k^{-1} v, k^{-1} w)\in \mathfrak{g} ,\\
( \zeta  , Z)&:= (k,x) ^{-1} (k,x, \delta k, \delta x)= ( k^{-1} \delta k, k^{-1} \delta x)\in \mathfrak{g}\\
( \mu , b)&:= (k,x) ^{-1} (k,x,p,\pi )=( k^{-1} p, k^{-1} \pi )\in \mathfrak{g}^\ast.
\end{align*} 

\begin{theorem}\label{SDP_Schneider}  With the above notations, the following statements are equivalent.
\begin{itemize}
\item[\bf{(i)}]
With $a_{0}$ fixed, the Lagrange-d'Alembert-Pontryagin principle
\begin{equation}\label{HP-ActInt_constrained_SDP}
\delta \int_{t_1}^{t_2}
\left\{ L_{a_{0}}(k(t), x(t), v(t), w(t)) + \left\langle p(t),  \dot{k}(t)-v(t) \right\rangle+\left\langle \pi(t),  \dot{x}(t)-w(t) \right\rangle \right\} \, dt =0
\end{equation}
holds for variations $(\delta{k}, \delta x,\delta{v},\delta w,\delta{p}, \delta \pi)$ of  $(k,x,v,w,p, \pi )$ in $TG\oplus T^{\ast}G$, where $\delta k, \delta x$ vanish at endpoints, and with $(\delta k, \delta x)\in \Delta _G^{a_0}(k,x) $ and with the constraint $(v,w)\in \Delta _G^{a_0 }(k,x)$.
\item[\bf{(ii)}]
The curve $(k(t), x (t) ) \in G$ satisfies the implicit Euler-Lagrange equations for $L_{a_{0}}$ on $TG \oplus T^{\ast}G$:
\begin{equation*}\label{ImpELEq_constrained_SDP}
\begin{split}
&(\dot{k}, \dot x)=(v,w)\in \Delta ^{a_0}_G(k,x), \quad \left( \dot{p}-\frac{\partial L_{a_{0}}}{\partial k},\dot{ \pi }-\frac{\partial L_{a_{0}}}{\partial x} \right) \in  \Delta ^{a_0}_{G}(k,x)  ^\circ,\\
&p=\frac{\partial L_{a_{0}}}{\partial v},  \;\; \pi=\frac{\partial L_{a_{0}}}{\partial w}.
\end{split}
\end{equation*}

\item[\bf{(iii)}]
The reduced Lagrange-d'Alembert-Pontryagin principle
\begin{equation}\label{ConstHP-ActInt_constrained_SDP}
\delta \int_{t_1}^{t_2} \left\{ \ell(\eta(t), Y(t) , a(t))+ 
\left\langle \mu(t), \xi(t) - \eta(t) \right\rangle+\left\langle b(t), X(t) - Y(t) \right\rangle \right\} ~dt=0,
\end{equation}
holds on $(\mathfrak{g} \oplus \mathfrak{g}^{\ast}) \times V^{\ast}$, under arbitrary variations of $ (\eta , Y), ( \mu ,b)$ and variations of the form
\[
\delta{\xi}=\frac{\partial \zeta}{\partial t}+[\xi, \zeta], \quad \delta X=\frac{\partial Z}{\partial t}+ \xi Z- \zeta X,\quad \delta{a}=-\zeta a,
\]
where $(\zeta,Z) \in \mathfrak{g}^{ \Delta}(a)$ vanishes at the endpoints and $ (\eta,Y) $ verifies the constraint $(\eta ,Y)\in \mathfrak{g}^{ \Delta}(a)$.

\item[\bf{(iv)}]
The implicit Euler-Poincar\'e-Suslov equations with advected parameters hold on $(\mathfrak{g} \oplus \mathfrak{g}^{\ast}) \times V^{\ast}$:
\begin{equation}\label{implicit_EP_constrained_SDP}
\left\{
\begin{array}{l}
\vspace{0.2cm}\displaystyle\mu =  \frac{\delta \ell}{\delta \eta},\quad  b= \frac{\delta \ell}{\delta Y}, \quad (\xi,X) =( \eta,Y)\in\mathfrak{g}^{ \Delta}(a),\\
\displaystyle\left( \dot{\mu} - \operatorname{ad}_{\xi}^{\;\ast} \mu+X \diamond b - \frac{\delta\ell}{\delta{a}} \diamond a, \dot b + \xi b\right) \in \left( \mathfrak{g}^{ \Delta}(a)\right) ^\circ.
\end{array}
\right.
\end{equation} 
\end{itemize}
\end{theorem}

Note that \eqref{implicit_EP_constrained_SDP} yields the motion equation
\[
\left( \partial _t \frac{\delta \ell}{\delta \xi }  - \operatorname{ad}_{\xi}^{\;\ast} \frac{\delta \ell}{\delta \xi } +X \diamond \frac{\delta \ell}{\delta X} - \frac{\delta\ell}{\delta{a}} \diamond a, \partial _t \frac{\delta \ell}{\delta X} + \xi \frac{\delta \ell}{\delta X}\right) \in \left( \mathfrak{g}^{ \Delta}(a)\right) ^\circ, \quad (\xi,X)\in\mathfrak{g}^{ \Delta}(a).
\]

\paragraph{Application to the particular cases.} As we shall see, equation \eqref{implicit_EP_constrained_SDP} can be made more explicit when the constraint is given by the particular cases mentioned above.
We shall thus reformulate here the statement $\bf{(iv)}$ in the particular cases when the constraints are given by \eqref{case_I}--\eqref{case_III}.
\begin{itemize}
\item[\bf (I)]
{\bf The basic case:\;}
By applying \eqref{ConstHP-ActInt_constrained_SDP}, we find
\[
\int_{t _1 } ^{t _2 } \left \{ \left\langle \dot{\mu} - \operatorname{ad}_{\xi}^{\;\ast} \mu+X \diamond b - \frac{\delta\ell}{\delta{a}} \diamond a, \zeta \right\rangle+\left\langle \dot b + \xi b, Z \right\rangle  \right \} dt.
\]
Using that the constraint $( \zeta , Z)\in \mathfrak{g}^ \Delta (a)$ reads $Z= \zeta a ^\sharp$, we get the equation
\[
\dot{\mu} - \operatorname{ad}_{\xi}^{\;\ast} \mu+X \diamond b - \frac{\delta\ell}{\delta{a}}\diamond  a+ a^\sharp \diamond \left(   \dot b + \xi b\right) =0,
\]
which can be rearranged as
\[
\partial _t \left( \mu + a^\sharp \diamond b \right) - \operatorname{ad}_{\xi}^{\;\ast} \left( \mu + a^\sharp \diamond b \right)=\frac{\delta\ell}{\delta{a}} \diamond a+ \partial _t a^\sharp \diamond  b, 
\]
where we used the formula $ \operatorname{ad}^{\ast}_ \xi ( v \diamond a)= - \xi v \diamond a- v \diamond \xi a $, for all $v\in V$, $a\in V^{\ast}$, $ \xi \in \mathfrak{k}$, and the constraint $X=Y= \eta a^\sharp = \xi a^\sharp$. Thus the implicit Euler-Poincar\'e-Suslov equations with advected parameter \eqref{implicit_EP_constrained_SDP} read 
\[
\left\{
\begin{array}{l}
\vspace{0.2cm}\displaystyle\mu =  \frac{\delta \ell}{\delta \eta},\quad  b= \frac{\delta \ell}{\delta Y}, \quad (\xi,X) =( \eta,Y), \quad Y= \eta a^\sharp,\\
\displaystyle\partial _t \left( \mu + a^\sharp \diamond b \right) - \operatorname{ad}_{\xi}^{\;\ast} \left( \mu + a^\sharp \diamond b \right)=\frac{\delta\ell}{\delta{a}} \diamond a+ \partial _t a^\sharp \diamond  b.
\end{array}
\right.
\]

This yields the motion equation
\[
\partial _t \left( \frac{\delta \ell}{\delta \xi } + a^\sharp \diamond \frac{\delta \ell}{\delta X } \right) - \operatorname{ad}_{\xi}^{\;\ast} \left( \frac{\delta \ell}{\delta \xi } + a^\sharp \diamond \frac{\delta \ell}{\delta X } \right)=\frac{\delta\ell}{\delta{a}} \diamond a+ \partial _t a^\sharp \diamond  \frac{\delta \ell}{\delta X }, \quad X= \xi a^\sharp.
\]

\item[\bf (II)]
{\bf The intermediate case:\;}The same computation as above yields the implicit Euler-Poincar\'e-Suslov equations with advected parameters
\begin{equation}\label{case_II_equations} 
\left\{
\begin{array}{l}
\vspace{0.2cm}\displaystyle\mu =  \frac{\delta \ell}{\delta \eta},\quad  b= \frac{\delta \ell}{\delta Y}, \quad (\xi,X) =( \eta,Y), \quad Y= \eta \phi(a),\\
\displaystyle\partial _t \left( \mu + \phi(a) \diamond b \right) - \operatorname{ad}_{\xi}^{\;\ast} \left( \mu + \phi(a) \diamond b \right)=\frac{\delta\ell}{\delta{a}} \diamond a+ \partial _t \phi(a) \diamond  b
\end{array}
\right.
\end{equation} 
and the motion equation
\[
\partial _t \left( \frac{\delta \ell}{\delta \xi } + \phi (a)  \diamond \frac{\delta \ell}{\delta X } \right) - \operatorname{ad}_{\xi}^{\;\ast} \left( \frac{\delta \ell}{\delta \xi } + \phi (a)  \diamond \frac{\delta \ell}{\delta X } \right)=\frac{\delta\ell}{\delta{a}} \diamond a+ \partial _t \phi (a)  \diamond  \frac{\delta \ell}{\delta X }, \quad X= \xi \phi (a) .
\]

\item[\bf (III)]
{\bf The general case:\;}Introducing the notation for the reduced map $\alpha(a):= A(e,0,a): \mathfrak{k}  \rightarrow V$ and applying the reduced Lagrange-d'Alembert-Pontryagin principle \eqref{ConstHP-ActInt_constrained_SDP},  we obtain
\[
\partial _t \mu  - \operatorname{ad}_{\xi}^{\;\ast} \mu+X \diamond b - \frac{\delta\ell}{\delta{a}}\diamond  a+ \alpha(a)^{\ast} \cdot  \left(\partial _t b + \xi b\right) =0,
\] 
which can be rearranged as
\begin{align*} 
\partial _t \left( \mu +\alpha(a)^{\ast}\cdot b \right) - \operatorname{ad}_{\xi}^{\;\ast} \left( \mu + \alpha(a)^{\ast} \cdot b \right)=&\frac{\delta\ell}{\delta{a}} \diamond a+ (\partial _t \alpha(a)^{\ast})\cdot  b\\
&- \left( \alpha(a) \cdot \xi \right) \diamond b - \alpha(a)^{\ast} \cdot \xi b- \operatorname{ad}^{\ast}_ \xi \left(\alpha(a)^{\ast}\cdot b \right).
\end{align*}
Thus, we obtain the implicit Euler-Poincar\'e-Suslov equation with parameters
\[
\left\{
\begin{array}{l}
\vspace{0.2cm}\displaystyle\mu =  \frac{\delta \ell}{\delta \eta},\quad  b= \frac{\delta \ell}{\delta Y}, \quad (\xi,X) =( \eta,Y), \quad Y=\alpha(a) \cdot \eta, \\
\vspace{0.2cm}\displaystyle\partial _t \left( \mu +\alpha(a)^{\ast} \cdot b \right) - \operatorname{ad}_{\xi}^{\;\ast} \left( \mu + \alpha(a)^{\ast} \cdot b \right)\\
\qquad \qquad \qquad \displaystyle=\frac{\delta\ell}{\delta{a}} \diamond a+ (\partial _t \alpha(a)^{\ast} )\cdot  b- \left( \alpha(a) \cdot \xi \right) \diamond b - \alpha(a)^{\ast} \cdot \xi b- \operatorname{ad}^{\ast}_ \xi \left( \alpha(a)^{\ast} \cdot b \right)
\end{array}
\right.
\]
and the motion equation
\begin{align*} 
&\partial _t \left( \frac{\delta \ell}{\delta \xi } + \alpha(a)^{\ast} \cdot\frac{\delta \ell}{\delta X } \right) - \operatorname{ad}_{\xi}^{\;\ast} \left( \frac{\delta \ell}{\delta \xi } +\alpha(a)^{\ast} \cdot\frac{\delta \ell}{\delta X } \right)\\
& \qquad =\frac{\delta\ell}{\delta{a}} \diamond a+ \left((\partial _t \alpha(a)^{\ast} )\cdot-  \left(\alpha(a) \cdot \xi \right) \diamond  - \alpha(a)^{\ast} \cdot \xi - \operatorname{ad}^{\ast}_ \xi \alpha(a)^{\ast} \cdot \right)\frac{\delta \ell}{\delta X},
\end{align*} 
where $X= \alpha(a)\cdot \xi$.  Note that the last three terms cancel when $\alpha(a) \cdot \xi = \xi \phi (a)$, for any smooth function $ \phi : V^{\ast} \rightarrow V$.
\end{itemize}
These equations consistently recover the ones derived in \cite{Sch2002}; see especially equations (17), (21), (23), and (25).

\begin{remark}[\textbf{Alternative formulation of the Lagrange-d'Alembert-Pontryagin principle}]\label{Hybrid_HP_advected_constraints_sdp}{\rm In the context of Theorem \ref{SDP_Schneider}, the variational structure in Remark \ref{Hybrid_HP_advected_constraints} may be given by
\begin{equation}\label{alternative_Schneider} 
\delta \int_{ t _1 }^{ t _2 } \{ \ell( \xi,X,a )+ \left\langle \mu , k^{-1} \dot k - \xi \right\rangle+ \left\langle b, k^{-1} \dot x-X \right\rangle + \left\langle v, k^{-1} a _0 - a \right\rangle  \} \,dt=0,
\end{equation} 
with $( \dot k, \dot x) \in \Delta ^{ a _0 }_G(k,x)$,
for arbitrary variations $ \delta \xi(t) $, $ \delta X(t)$, $ \delta \mu (t) $, $ \delta b(t)$, $ \delta v(t)$, $ \delta a(t)$ of $ \xi (t) $, $X(t)$, $ \mu (t) $, $b(t)$, $v(t)$, $a(t) $ and variations $\delta k(t)$, $ \delta x(t)$ of $k(t)$, $x(t)$ vanishing at the endpoints, together with the constraint
\[
( \delta k, \delta  x) \in \Delta^{ a _0 }_G(k,x).
\]
It yields the equations of motion
\[
(\xi ,X)=( k^{-1} \dot k,  k^{-1} \dot x) \in \mathfrak{g} ^ \Delta (a), \quad a= k^{-1} a _0 , \quad \mu = \frac{\delta \ell}{\delta \xi },\quad b= \frac{\delta \ell}{\delta X}, \quad v= \frac{\delta \ell}{\delta a},
\]
\[
\left( \dot{\mu} - \operatorname{ad}_{\xi}^{\;\ast} \mu+X \diamond b - v  \diamond a, \dot b + \xi b\right) \in \left( \mathfrak{g}^{ \Delta}(a)\right) ^\circ. 
\]
In the particular case when the constraint subspace $ \mathfrak{g} ^ \Delta (a)$ is of the form \eqref{case_III_red}, this alternative formulation \eqref{alternative_Schneider} recovers the one developed in \cite[\S12.2]{Ho2008}.}
\end{remark}

\subsection{Generalization to arbitrary advected quantities}\label{generalization}

As was mentioned in \S\ref{Intro}, it has been recently observed that several mechanical systems, such as complex fluids, geometrically exact rods, as well as symmetry breaking phenomena in condensed matter physics require a more general setting for advected quantities (\cite{GBRa2009}, \cite{GBTr2010}, \cite{GBHoRa2009}, \cite{ElGBHoPuRa2009}, \cite{GBPu2012,GBPu2014}), both in the context of constrained and unconstrained systems.

In this case, one has to consider advected parameters in an arbitrary manifold $Q$, on which a Lie group $G$ acts. We denote by $ \Phi :G \times Q \rightarrow Q$, $(g,q) \mapsto \Phi _g(q)$ this action, assumed to be a left action.
Let $L:TG \times Q \rightarrow \mathbb{R}  $ be a $G$-invariant function, and consider the Lagrangian $L_{q_0}:TG \rightarrow \mathbb{R}  $, defined by $L_{q_0}(g, \dot g):= L(g, \dot g, q _0 )$, obtained by fixing the value $ q_0 $ of the parameter. As above, we assume that there is a constraint distribution
\[
\Delta _G^{q_0}\subset TG, \quad \Delta _G^{q_0}(g)= \Delta _G(g,q_0)\in T_{g}G
\]
depending on the parameter $q_0\in Q$, and we assume the invariance property
\[
\Delta _G (hg, h q_0 )= h \Delta _G (g, q _0 ), \quad \text{for all $h \in G$}. 
\]

One observes that Theorem \ref{thm_HP_constrained} generalizes easily to this case, therefore we shall not present it in full details. For example, after reduction the parameter in $Q$ is given by $q= \Phi _{ g ^{-1} }( q_0 )$ and thus verifies the (generalized) advection equation $ \dot q+ \xi _Q (q)=0$, where $ \xi _Q $ denotes the infinitesimal generator of the $G$ action. In this situation, the nonholonomic Euler-Poincar\'e equations with advected parameters \eqref{implicit_EP_nonholonomic} generalize to
\begin{equation*}\label{implicit_EP_constrained_gen}
\mu =  \frac{\delta \ell}{\delta \eta}, \quad \xi = \eta\in\mathfrak{g}^{ \Delta}(q) , \quad \dot{\mu} - \operatorname{ad}_{\xi}^{\;\ast} \mu + \mathbf{J} \left( \frac{\delta\ell}{\delta{q}}\right) \in \left( \mathfrak{g}^{ \Delta}(q)\right) ^\circ,
\end{equation*} 
where
\[
\mathbf{J}:T^{\ast}Q \rightarrow \mathfrak{g}  ^\ast , \quad \left\langle \mathbf{J} ( \alpha _q ) , \xi \right\rangle = \left\langle \alpha _q , \xi _Q ( q) \right\rangle 
\]
is the {\it cotangent bundle momentum map} associated to the action of $G$ on $Q$. These equations are the implicit and nonholonomic version of the Euler-Poincar\'e equations for symmetry breaking studied in \cite{GBTr2010}.

An important particular example is the case of an affine action of $G$ on $Q=V^{\ast}$, given by $ a \mapsto  ga+c(g)$, see \cite{GBRa2009}, where $g \mapsto ga$ denotes a left representation and $g \mapsto c(g)$ is a cocycle. This setting is useful for complex fluids and rod dynamics. In this case $ \mathbf{J} (a,v)=- a \diamond v+ \mathbf{d} c ^T(v)$ so that one obtains the {\bfi implicit affine Euler-Poincar\'e-Suslov equations}:
\begin{equation*}\label{implicit_affine_EP_constrained}
\mu =  \frac{\delta \ell}{\delta \eta}, \quad \xi = \eta\in\mathfrak{g}^{ \Delta}(a) , \quad \dot{\mu} - \operatorname{ad}_{\xi}^{\;\ast} \mu  - a \diamond \frac{\delta \ell}{\delta a} + \mathbf{d} c ^T\left( \frac{\delta \ell}{\delta a}\right)   \in \left( \mathfrak{g}^{ \Delta}(a)\right) ^\circ.
\end{equation*}

%%%%%%%%%%%%%%%%%%%%%%%%%

\subsection{Variational framework on the Hamiltonian side}\label{Hamiltonian_side}

Recall from \S\ref{sec_LPSDP} that, on the Hamiltonian side, we consider a $G$-invariant Hamiltonian $H: T ^\ast G \times V ^\ast \rightarrow \mathbb{R}  $ and that, fixing the value of the parameter $ a _0 \in V ^\ast $, we get the Hamiltonian $H_{ a _0}: T ^\ast G \rightarrow \mathbb{R}  $ of the mechanical system. In some situations, the Hamiltonian $H_{ a _0 }$ is obtained by Legendre transformation of the Lagrangian $L_{ a _0 }: TG \rightarrow \mathbb{R}$ supposed to be hyperregular. However, we will not make such an assumption, since our theory is independent of the existence of such a Lagrangian.

Recall that by $G$-invariance $H$ induces a reduced Hamiltonian $h: \mathfrak{g}  ^\ast \times V ^\ast \rightarrow \mathbb{R}$, and that $H$ can be seen as the reduced expression of a $S$-invariant Hamiltonian $\bar H: T ^\ast S \rightarrow \mathbb{R}  $. The reduction processes $\bar H \rightarrow H \rightarrow h$ can be understood as a Poisson reduction by stages.

\medskip

With $a_{0}$ fixed, the Hamilton equations for a mechanical systems with nonholonomic constraint $ \Delta ^{ a _0 } _G \subset TG$ and Hamiltonian $H_{ a _0 }$ can be obtained by the \textit{Hamilton-d'Alembert principle}
\begin{equation}\label{PS_principleHa0}
\delta \int_{ t _1 }^{ t _2 } \left\{ \left\langle p, \dot g \right\rangle -H_{ a _0 }(g,p)\right\} dt=0,\quad \dot g \in \Delta ^{ a _0 }_G(g),
\end{equation} 
for variations $ \delta g (t) $ of $g(t)$ vanishing at the endpoints and such that $\delta g\in \Delta _G^{a_0}$, and for arbitrary variations $ \delta p (t) $ of $p(t)$. We thus obtain the implicit Hamilton-d'Alembert equations
\[
\dot g= \frac{\partial H_{a_0}}{\partial p}\in \Delta^{a _0 } _G(g), \quad \dot p+ \frac{\partial H_{a_0}}{\partial g}\in \Delta _G^{a _0 } (g)^\circ.  
\]
Using the $G$-invariance of $H$, the reduction of \eqref{PS_principleHa0} induces the {\bfi reduced Hamilton-d'Alembert principle}
\begin{equation}\label{Lie_Poisson_VP}
\delta \int_{t_1}^{t_2} \left \{ \left\langle \mu , \xi  \right\rangle -h(\mu , a) \right\}dt=0,\quad \xi \in \mathfrak{g}  ^ \Delta (a)
\end{equation}
with $ \mu = g ^{-1} p$, $ \xi = g ^{-1}\dot g$, and $ a = g ^{-1} a _0 $. We thus obtain that the variation of $ \mu $ is arbitrary, whereas the variations of $ \xi $ and $a$ are computed to be $\delta \xi  = \partial _t \zeta +[ \xi , \zeta ]$ and $\delta a=- \zeta a$, where $ \zeta = g ^{-1} \delta g$ is an arbitrary curve vanishing at the endpoints and such that $ \zeta \in \mathfrak{g}  ^ \Delta (a)$.

From the reduced Hamilton-d'Alembert principle \eqref{Lie_Poisson_VP}, we get the {\bfi implicit Lie-Poisson-Suslov equations with advected parameters}:
\[
\xi  =  \frac{\delta h}{\delta \mu }\in \mathfrak{g}  ^ \Delta (a), \quad \dot{\mu} - \operatorname{ad}_{\xi}^{\;\ast} \mu +\frac{\delta h}{\delta a}  \diamond a\in ( \mathfrak{g}  ^ \Delta (a)) ^\circ.
\]

We can therefore formulate the following theorem, which provides the Hamiltonian counterpart of Theorem \ref{theorem_HP_semidirect} and \ref{thm_HP_constrained}.

\begin{theorem}\label{HPS_principle_first}  Let $H:T^*G \times V ^\ast \rightarrow \mathbb{R}  $ be a $G$-invariant Hamiltonian and let $ \Delta ^{a _0 }_G \subset TG$ be a family of distribution verifying the $G$-invariance assumption \eqref{invariance_assumption}. Fix a parameter $ a _0 \in V ^\ast $, consider a curve $(g(t),p(t)) \in T^*G$, $t \in [t_{1},t_{2}]$, and define the curves $ \mu (t)= g (t) ^{-1} p (t) $ and $a (t) = g(t) ^{-1} a _0 $. Then the following are equivalent.
\begin{itemize}
\item[\bf{(i)}]
With $a_{0}$ fixed, the Hamilton-d'Alembert principle in phase space
\[
\delta \int_{ t _1 }^{ t _2 } \left\{ \left\langle p, \dot g \right\rangle -H_{ a _0 }(g,p)\right\} dt=0,\quad \dot g \in \Delta ^{ a _0 }_G(g)
\]
holds, for variations $ \delta g (t) $ of $g(t)$ vanishing at the endpoints and such that $\delta g\in \Delta _G^{a_0}$, and for arbitrary variations $ \delta p (t) $ of $p(t)$.
\item[\bf{(ii)}]
The curve $(g(t),p(t) )\in T^*G$, $t \in [t_{1},t_{2}]$ satisfies the implicit Hamilton-d'Alembert equations:
\[
\dot g= \frac{\partial H_{a_0}}{\partial p}\in \Delta ^{a _0 }_G(g), \quad \dot p+ \frac{\partial H_{a_0}}{\partial g}\in \Delta _G^{a _0 } (g)^\circ.  
\]
\item[\bf{(iii)}] The reduced Hamilton-d'Alembert principle with advected quantities
\begin{equation}\label{Lie_Poisson_VPSDP_first} 
\delta \int_{t_1}^{t_2} \left \{ \left\langle \mu , \xi  \right\rangle -h(\mu , a) \right\}dt=0,\quad \xi \in \mathfrak{g}  ^ \Delta (a)
\end{equation} 
holds, for arbitrary variations $ \delta \mu (t) $ and variations $ \delta \xi (t) $ and $ \delta a(t) $ of the form $\delta \xi  = \partial _t \zeta +[ \xi , \zeta ]$ and $\delta a=- \zeta a$, where $ \zeta \in \mathfrak{g}  ^ \Delta (a)$ and vanishes at the endpoints.
\item[\bf{(iv)}] The implicit Lie-Poisson-Suslov equations with advected parameters
\begin{equation}\label{implicit_LP_SDP} 
\xi  =  \frac{\delta h}{\delta \mu }\in \mathfrak{g}  ^ \Delta (a), \quad \dot{\mu} - \operatorname{ad}_{\xi}^{\;\ast} \mu +\frac{\delta h}{\delta a}  \diamond a\in ( \mathfrak{g}  ^ \Delta (a)) ^\circ
\end{equation} 
hold.
\end{itemize}
\end{theorem}

\begin{remark}[\textbf{On the advection equation}]{\rm Note that, similarly with Theorem \ref{theorem_HP_semidirect} and \ref{thm_HP_constrained}, the relation $ a (t) = g (t) ^{-1} a _0 $ is assumed as an hypothesis for the preceding theorem. Equivalently, the advection equation $\partial _t a + \xi a =0$, with initial condition $a (t _1 ) = a _0 $, is assumed to hold, for $ \xi(t)  = g ^{-1}(t)  \dot g (t) $. The Dirac formulation of this reduction process will be formulated in \S\ref{subsec_preliminary_comments}.

We will later consider a version of the previous theorem that includes the advection equation as a consequence of the variational structure in phase space, without having to assume it as an hypothesis. In this case, one has to formulate the phase space principle on $T^*S$ for the $S$-invariant Hamiltonian $\bar H: T^*S \rightarrow \mathbb{R}  $ and for an appropriate distribution constraint on $T^*S$. This principle, together with the Dirac formulation, will be the subject of \S\ref{sec_NH_SDP}.}
\end{remark} 

\begin{remark}[\textbf{Cautionary remark}]\label{not_suslovSDP} {\rm  It is important to observe that while the equations \eqref{implicit_LP_SDP} together with the advection equation $ \partial _t a + \xi a =0$ are a nonholonomic version of the (ordinary) Lie-Poisson equations on the semidirect product $ \mathfrak{s} = \mathfrak{g}  \,\circledS\, V$, they are \textit{not} the Lie-Poisson-Suslov equations on the semidirect product  $ \mathfrak{s} = \mathfrak{g}  \,\circledS\, V$. This is why we use the terminology \textit{Lie-Poisson-Suslov equations with advected parameters} for \eqref{implicit_LP_SDP}. The Lie-Poisson-Suslov equations on the semidirect product $ \mathfrak{s} $ have to be associated to a $S$-invariant constraint $ \Delta _S \subset TS$, yielding the constraint $ \mathfrak{s} ^ \Delta \subset \mathfrak{s} $ at the identity, and are thus given by
\[
\left( \frac{\delta h}{\delta \mu },\frac{\delta h}{\delta a }\right)  \in\mathfrak{s}^{ \Delta} , \quad \left( \dot{\mu} - \operatorname{ad}_{\frac{\delta h}{\delta \mu }}^{\;\ast} \mu +\frac{\delta h}{\delta{a}} \diamond a, \dot a + a \frac{\delta h}{\delta \mu } \right) \in (\mathfrak{s}^{ \Delta} )^\circ.
\]
In general these equations are different from \eqref{implicit_LP_SDP}  since they include also constraints on $\frac{\delta h}{\delta a}$ and do not allow $a$-dependence in the constraints.

This comment will be illustrated in \S\ref{sec_NH_SDP} later, on the Hamiltonian side, by the fact that the Dirac structure on $T^{\ast}S$ that one has to start with is not induced from a distribution $ \Delta  _S $ on $S$.
}
\end{remark} 

\begin{remark}[\textbf{The case of arbitrary advected quantities}]{\rm One can easily adapt Theorem \ref{HPS_principle_first} to the general setting mentioned in \S\ref{generalization}. The corresponding nonholonomic Lie-Poisson equations are
\[
\xi  =  \frac{\delta h}{\delta \mu }\in \mathfrak{g}  ^ \Delta (q), \quad \dot{\mu} - \operatorname{ad}_{\xi}^{\;\ast} \mu + \mathbf{J} \left( \frac{\delta\ell}{\delta{q}}\right) \in \left( \mathfrak{g}^{ \Delta}(q)\right) ^\circ.
\]}
\end{remark}

\begin{remark}[\textbf{The case of rolling ball type constraint on semidirect products}]\label{sdp_case_Ham}{\rm One can easily write Theorem \ref{HPS_principle_first} in the special case when the Lie group $G$ is itself a semidirect product $G=K \,\circledS\,
 V$, as described in \S\ref{subsec_the_sdp_case}. For example, in this situation the Hamilton-d'Alembert principle reads
\[
\delta \int_{ t _1 }^{ t _2 } \left\{ \langle p, \dot k \rangle+ \left\langle \pi , \dot x \right\rangle  -H_{ a _0 }(k,p, x , \pi )\right\} dt=0,\quad (\dot k, \dot x) \in \Delta ^{ a _0 }_G(k,x),
\]
for variations $ \delta k(t) , \delta x (t) $ of $k(t), x (t) $ vanishing at the endpoints and such that $(\delta k, \delta x)\in \Delta _G^{a_0}$, and for arbitrary variations $ \delta p (t) , \delta \pi (t) $ of $p(t), \pi (t) $.

The reduced Hamilton-d'Alembert principle with advected quantities reads
\begin{equation*}\label{Lie_Poisson_VPSDP_first_SDP0} 
\delta \int_{t_1}^{t_2} \left \{ \left\langle \mu , \xi  \right\rangle + \left\langle b, X \right\rangle -h(\mu ,b, a) \right\}dt=0,\quad (\xi,X) \in \mathfrak{g}  ^ \Delta (a),
\end{equation*} 
for arbitrary variations $ \delta \mu (t) , \delta b (t) $ and variations $ \delta \xi (t) $, $ \delta X(t) $ and $ \delta a(t) $ of the form $\delta{\xi}=\partial _t \zeta+[\xi, \zeta]$, $\delta X=\partial _t Z+ \xi Z- \zeta X$ and $\delta a=- \zeta a$, where $( \zeta,Z) \in \mathfrak{g}^ \Delta (a)$ and vanishes at the endpoints. One gets the implicit Lie-Poisson-Suslov equations with advected parameters:
\begin{equation*}\label{implicit_LP_SDP_SDP} 
(\xi,X)  =  \left( \frac{\delta h}{\delta \mu }, \frac{\delta h}{\delta b } \right) \in \mathfrak{g}^ \Delta (a), \;\; \left(\partial _t \mu  - \operatorname{ad}_{\xi}^{\;\ast} \mu + X \diamond b+\frac{\delta h}{\delta a}  \diamond a, \partial _t b + \xi b \right) \in ( \mathfrak{g}^ \Delta (a)) ^\circ.
\end{equation*}}
\end{remark}
 
\section{Lie-Dirac reduction with advected quantities}\label{Dirac_red_advect}
In this section, we shall investigate the reduction of an induced Dirac structure on
the cotangent bundle, for the case in which the configuration manifold is given by a Lie
group and with a parameter dependent constraint distribution given by $ \Delta _G ^{a_0}\subset TG$, with the invariance condition given in \eqref{invariance_assumption}: 
\[
\Delta _G (hg, ha_{0})=h \Delta _G (g,a_{0}).
\]
Then, we shall  formulate the reduction of the corresponding Lagrange-Dirac and Hamilton-Dirac dynamical systems, and finally show that the reduced formulations provide the Euler-Poincar\'e-Suslov and Lie-Poisson-Suslov equations with advected parameters for nonholonomic mechanics.

\subsection{Lie-Dirac reduction of the induced Dirac structures}

Recall that given a parameter dependent constraint distribution $ \Delta _G ^{a_0}\subset TG$, the associated Dirac structure $D_{ \Delta ^{a_0}_G}\subset T(T^{\ast}G) \oplus T^{\ast}(T^{\ast}G)$ is given by
\begin{align*} 
D_{ \Delta ^{a_0}_G}(p_g)&=\left\{(v_{p_g}, \alpha _{p _g })\in T_{p_g}(T^{\ast}G) \oplus T^{\ast}_{p_g}(T^{\ast}G)\mid v_{p_g}\in \Delta ^{a_0}_{T^{\ast}G}( p_g),\right.\\
&\quad \left. \phantom{T^{\ast}_{p_g}(T^{\ast}G)}\left\langle \alpha _{p _g }, w _{p _g } \right\rangle = \Omega (p_g) \left( v _{p _g }, w_{p _g } \right),\;\;\text{for all}\;\; w_{p_g}\in  \Delta ^{a_0}_{T^{\ast}G}( p_g)\right\},
\end{align*}
where
\[
\Delta ^{a_0}_{T^{\ast}G}( p_g):= \left( T_{p_g} \pi \right) ^{-1} \left( \Delta ^{a_0}_G(g) \right) \subset T_{p_g}T^{\ast}G.
\]

\paragraph{Trivialization.} In order to implement the reduction process, we shall first trivialize the expression of $D_{ \Delta ^{ a _0 }_G}$. The trivialized Dirac structure, denoted $\bar D_{ \Delta ^{a_0}_G}\subset T(G \times \mathfrak{g}  ^\ast )\oplus T^{\ast}( G \times \mathfrak{g}  ^\ast )$ reads
\begin{align*} 
&\bar D_{ \Delta ^{a_0}_G}(g, \mu )=\left\{\left .\left( (v_g, \rho ),(\beta _g, \eta  ) \right)\in (T_g G \times \mathfrak{g}  ^\ast) \oplus (T^{\ast}_gG \times \mathfrak{g}  ^\ast)\,\right|\, v_g\in \Delta ^{a_0}_{G}(g),\right.\\
&\left. \phantom{T^{\ast}_{p_g}(T^{\ast}G)}\left\langle \beta _g , w _g  \right\rangle + \left\langle \eta , \sigma \right\rangle = \omega (g, \mu ) \left((v_g, (\mu , \rho )),(w_g, ( \mu , \sigma )  \right), \;\;\text{for all}\;\;w_g\in  \Delta ^{a_0}_G (g) ,\; \sigma \in \mathfrak{g}  ^\ast \right\},
\end{align*}
where $\omega  \in \Omega ^2( G \times \mathfrak{g}  ^\ast )$ is the trivialized canonical symplectic form, given by 
\[
\omega  (g, \mu ) \left((v_g, \rho ),(w_g, \sigma )  \right)= \left\langle \sigma , g ^{-1} v _g \right\rangle -\left\langle \rho , g ^{-1} w _g \right\rangle + \left\langle \mu , [ g ^{-1} v _g , g ^{-1} w _g ] \right\rangle .
\]
Therefore, we have the equivalence
\begin{equation}\label{condition_D} 
\begin{array}{c} 
\left(( v_g ,\rho ),(\beta _g,  \eta  ) \right)\in \bar D_{ \Delta ^{a_0}_G}(g, \mu )\\
\Longleftrightarrow\\
g ^{-1} \beta _g + \rho - \operatorname{ad}^{\ast}_{g ^{-1} v _g } \mu \in g ^{-1} \Delta _{G}^{a_0}(g)^\circ \quad \text{and} \quad \eta = g ^{-1} v _g \in g ^{-1} \Delta _{G}^{a_0}(g).\\
\end{array} 
\end{equation}
By the invariance property \eqref{invariance_assumption} of $ \Delta _G^{a_0}$ and the definition \eqref{def_g_delta} of $\mathfrak{g}  ^\Delta(a)$, we have
\[
g ^{-1} \Delta _{G}^{a_0}(g)= g ^{-1} \Delta _G(g, a _0 )= \Delta _G(e, g ^{-1} a_0)= \mathfrak{g}  ^\Delta (a).
\]
So we can rewrite \eqref{condition_D} as
\begin{equation}\label{condition_D_new}
g ^{-1} \beta _g + \rho - \operatorname{ad}^{\ast}_{g ^{-1} v _g } \mu \in (\mathfrak{g}  ^\Delta (a))^\circ \quad \text{and} \quad \eta = g ^{-1} v _g \in \mathfrak{g}  ^\Delta (a), \quad a:= g ^{-1} a _0 .
\end{equation}

\paragraph{Reduction.} We shall now use that both $D_{ \Delta ^{a_0}_G}$ and $\bar D_{ \Delta ^{a_0}_G}$ are $G_{a_0}$-invariant. Note that we have the diffeomorphism
\[
\left( T(T^{\ast}G)\oplus T^{\ast}(T^{\ast}G) \right) /G_{a_0}\simeq \left( T(G \times \mathfrak{g}  ^\ast )\oplus T^{\ast}( G \times \mathfrak{g}  ^\ast )\right) /G_{a_0}\simeq \left( G/G_{a_0} \times \mathfrak{g}  ^\ast\right)  \times W \oplus W^{\ast},
\]
where $W:= \mathfrak{g}  \times  \mathfrak{g}  ^\ast $. Identifying the quotient $G/G_{a_0}$ with $ \operatorname{Orb}(a_0)\subset V^{\ast}$ of $a_0$, via the orbit map
\begin{equation}\label{Orbit_map_a0} 
gG_{a_0}=[g]_{G_{a_{0}}}\in G/G_{a_0} \mapsto g a_0\in \operatorname{Orb}(a_0),
\end{equation} 
we obtain the reduced Dirac structure $D^{/G_{a_0}}_{\Delta_G^{a_0}}:=\bar D_{ \Delta ^{a_0}_G}/G_{a_0} \subset \left( \operatorname{Orb}(a_0)\times \mathfrak{g}  ^\ast\right)  \times W \oplus W^{\ast}$ on the reduced Pontryagin bundle.
Using the expression of $\bar D_{ \Delta ^{a_0}_G}/G_{a_0}$ given earlier, we obtain the following result.

\begin{proposition} 
The reduced Dirac structure $D^{/G_{a_0}}_{\Delta_G^{a_0}}$ associated to the induced Dirac structure $ D_{ \Delta ^{ a _0 }}\subset T(T^*G )\oplus T^*(T^*G)$ is given at $(a,\mu )\in \operatorname{Orb}(a_0)  \times \mathfrak{g}^{\ast}$ by
\begin{align*} 
D^{/G_{a_0}}_{\Delta_G^{a_0}}(a, \mu)&=\left\{ \left( (a, \mu, \xi , \rho ),(a, \mu, \beta , \eta ) \right) \in ( \operatorname{Orb}(a_0) \times \mathfrak{g}^{\ast})\times W\oplus W^{\ast} \mid \xi \in \mathfrak{g}  ^{ \Delta }(a),\right .\\
& \qquad \qquad \left . \left\langle \beta , \zeta  \right\rangle + \left\langle \eta , \sigma \right\rangle = \left\langle \sigma , \xi  \right\rangle - \left\langle \rho , \zeta  \right\rangle + \left\langle \mu , [\xi ,\zeta ] \right\rangle , \;\forall \; \zeta \in \mathfrak{g}  ^{ \Delta }(a), \sigma \in \mathfrak{g}  ^\ast \right\}.
\end{align*}
Thus we have the equivalence
\begin{equation}\label{condition_red_dirac} 
\begin{array}{c} 
\left( ( \mu , a, \xi , \rho ),( \mu ,a, \beta , \eta ) \right) \in D^{/G_{a_0}}_{\Delta_G^{a_0}}\\
\Longleftrightarrow\\
\beta + \rho - \operatorname{ad}^{\ast}_\xi \mu \in (\mathfrak{g}  ^\Delta (a))^\circ \;\; \text{and} \;\; \eta = \xi  \in \mathfrak{g}  ^\Delta (a).
\end{array} 
\end{equation} 
\end{proposition} 
\begin{figure}[h]
\begin{center}
%\hspace{2cm}
\includegraphics[scale=.63]{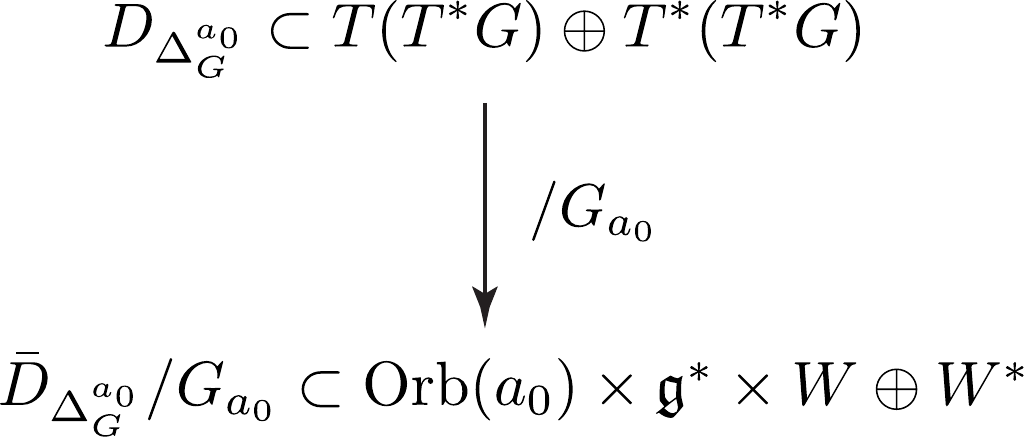}
\caption{Lie-Dirac Reduction by $G_{ a _0 }$}
\label{DiracReduction}
\end{center}
\end{figure}

\subsection{Euler-Poincar\'e-Dirac reduction with advected parameters}

Knowing the expression of the reduced Dirac structure, we shall now consider the reduction of the Lagrange-Dirac system associated to $L_{ a _0 }$ and $ \Delta _G ^{ a _0 }$. Recall from \eqref{LDirac_system} that the equations of motion of a Lagrange-Dirac system $(G, \Delta ^{ a _0 }_G, L_{ a _0 })$ read
\begin{equation*}\label{LDirac_system_recall} 
\left(( g,p,\dot g, \dot p), \mathbf{d}_DL_{a_0}(g,v) \right) \in D_{ \Delta ^{a_0}_G}(g,p),
\end{equation*} 
for a curve $(g(t), v(t), p(t)) \in TG \oplus T^*G$. Recall also that the Dirac differential of $L_{ a _0 }$ reads
\[
\mathbf{d}_DL_{ a _0 }(g,v)= \left( g, \frac{\partial L_{ a _0 }}{\partial v} , - \frac{\partial L_{ a _0 }}{\partial g}, v\right).
\]
A curve is a solution of the Lagrange-Dirac system if and only if it verifies the implicit Euler-Lagrange-d'Alembert equations
\begin{equation}\label{implicit_EL_Dirac}
p= \frac{\partial L_{a_0}}{\partial v}, \quad  \dot g=v\in \Delta ^{a_0}_G(g), \quad \dot p- \frac{\partial L_{a_0}}{\partial g}  \in \Delta _G(g)^\circ.
\end{equation} 

\paragraph{Trivialization.}
In order to implement the reduction process, we shall first write the trivialized version of the Lagrange-Dirac system $(G, \Delta ^{ a _0 }_G, L_{ a _0 })$. In doing this, we have to remember that $L_{ a _0 }$ is not $G$-invariant but only $G_{ a _0 }$-invariant. Let us consider the induced Lagrangian $\bar L_{a_0}:G \times \mathfrak{g}  \rightarrow \mathbb{R}  $, defined by
\[
L_{a_0}(g, v)= \bar L_{a_0}(g, g ^{-1} v).
\]
The trivialized expression of the Dirac differential reads
\[
\overline{\mathbf{d}_D L_{a_0}}(g, \eta )= \left( g, \frac{\partial \bar L_{a_0}}{\partial \eta }, - g ^{-1} \frac{\partial \bar L_{a_0}}{\partial g} , \eta \right) \in G \times \mathfrak{g}  ^\ast \times W^\ast.
\]
Using the expression \eqref{condition_D_new} of the trivialized Dirac structure, one obtains that the trivialized implicit Lagrangian system, given by
\begin{equation}\label{implicit_Lagr_syst_trivialized}
\left(( g, \mu , \xi , \dot \mu ), \overline{\mathbf{d}_DL_{a_0}}(g, \eta ) \right) \in \bar D_{ \Delta _G^{a_0}}(g, \mu )
\end{equation} 
with $\mu = g ^{-1} p$, $\xi = g ^{-1} \dot g$, yields the equations
\[
\left\{
\begin{array}{l}
\vspace{0.2cm}\displaystyle\dot \mu - \operatorname{ad}^{\ast}_ \xi \mu - g ^{-1} \frac{\partial \bar L_{a_0}}{\partial g} \in (\mathfrak{g}  ^\Delta(a))^\circ, \quad \mu = \frac{\partial \bar L_{a_0}}{\partial \eta }\\
\eta = \xi \in \mathfrak{g}  ^\Delta(a).
\end{array}
\right.
\]
This system is the trivialized version of the implicit Lagrange-d'Alembert equations \eqref{implicit_EL_Dirac}.

\paragraph{Reduction.} We shall now reduce the system \eqref{implicit_Lagr_syst_trivialized} by using the $G_{a  _0 }$-invariance. Using the orbit map \eqref{Orbit_map_a0}, we get the system
\[
\left((a, \mu , \xi , \dot \mu ), \overline{\mathbf{d}_DL}^{/G_{a_0}}_{a_0}(a, \eta ) \right) \in D^{/G_{a_0}}_{ \Delta _G^{a_0}}(a, \mu),
\]
where $(a, \mu , \xi , \dot \mu ) \in  \operatorname{Orb}(a_0) \times \mathfrak{g}  ^\ast \times W$ and $\overline{\mathbf{d}_D L_{a_{0}}}^{/G_{a_0}}(a, \eta )\in \operatorname{Orb}(a_0) \times \mathfrak{g}  ^\ast \times W^{\ast}$ is given by
\[
\overline{\mathbf{d}_D L_{a_{0}}}^{/G_{a_0}}(a, \eta)= \left(a, \frac{\partial \bar L_{a_0}}{\partial \eta }, - g ^{-1} \frac{\partial \bar L_{a_0}}{\partial g} , \eta \right) \in \operatorname{Orb}(a_0)  \times \mathfrak{g}  ^\ast \times W^{\ast},
\]
where we note that the third component is well defined, that is, does not depend on $g$ such that $g ^{-1} a_0=a$. We now observe that the reduced Lagrangian $\ell$ is related to $L_{a_0}$ and $\bar L_{a_0}$ by
\[
\bar L_{a_0}\left( g, g ^{-1} v\right) =L_{a_0}\left( g, v\right) = \ell\left( g^{-1} a_0, g ^{-1}v\right).
\]
We thus have
\[
\frac{\partial \bar L_{a_0}}{\partial \eta }=\frac{\delta \ell}{\delta \eta }, \quad  \frac{\partial \bar L_{a_0}}{\partial g }=g \left( \frac{\delta \ell}{\delta a}\diamond a \right),
\]
where $a= g ^{-1} a _0 $.
The second equality is proved as follows
\[
\left\langle \frac{\partial \bar L_{a_0}}{\partial g}, \delta g \right\rangle = \left\langle \frac{\delta \ell}{\delta a}, - g ^{-1} \delta g g ^{-1} a _0 \right\rangle = - \left\langle \frac{\delta \ell}{\delta a}, g ^{-1} \delta g a \right\rangle = \left\langle \frac{\delta \ell}{\delta a} \diamond a, g ^{-1} \delta g \right\rangle.  
\]
We therefore obtain the following definition for the reduced Dirac differential for systems on Lie group with advected quantities.

\begin{definition} Given $a _0 \in V ^\ast $, the $G_{a _0 }$-{\bfi reduced Dirac differential} of $\ell: \operatorname{Orb}(a_ 0 )  \times  \mathfrak{g} \rightarrow \mathbb{R}  $ is
\[
\mathbf{d} _D^{/ G_{a _0 }}\ell: \operatorname{Orb}(a_0 )  \times \mathfrak{g} \rightarrow  \operatorname{Orb}(a_0 ) \times \mathfrak{g}  ^\ast  \times W^\ast, \quad \mathbf{d} _D^{/ G_{a _0 }}\ell(a, \eta)= \left( a ,\frac{\delta \ell}{\delta \eta },  - \frac{\delta \ell}{\delta a} \diamond a, \eta \right).
\]
The reduced Lagrange-Dirac system associated to $(G, \Delta ^{ a _0 }_G, L_{ a _0 })$ is
\begin{equation}\label{implicit_Lagr_syst_reduced}
\left((a, \mu , \xi , \dot \mu ), \mathbf{d} _D^{/ G_{a _0 }}\ell( a, \eta ) \right) \in D^{/G_{a_0}}_{ \Delta _G^{a_0}}(a, \mu).
\end{equation} 
\end{definition}

Using \eqref{condition_red_dirac}, we get the following result.

\begin{proposition} Assume that the advection equation $\dot a+ \xi a =0$ is verified. The curve $t \mapsto (a(t) , \mu(t) , \xi(t)  )$ is a solution curve of the reduced Lagrange-Dirac system \eqref{implicit_Lagr_syst_reduced} 
if and only if it verifies the {\bfi implicit Euler-Poincar\'e-Suslov equations with advected parameters}:
\begin{equation*}\label{ImpEulPoinSusAd} 
-  \frac{\delta \ell}{\delta a}\diamond a + \dot \mu  - \operatorname{ad}^{\ast}_\xi \mu \in (\mathfrak{g}  ^\Delta (a))^\circ, \quad \eta = \xi  \in \mathfrak{g}  ^\Delta (a), \quad \text{and} \quad \mu =\frac{\delta \ell}{\delta \eta}.
\end{equation*} 
\end{proposition} 

This shows that the implicit Euler-Poincar\'e-Suslov equations with advected parameters, obtained earlier in \eqref{implicit_EP_nonholonomic} via the reduced Lagrange-d'Alembert-Pontryagin principle, can be naturally formulated  in the context of Dirac reduction.

\begin{remark}[\textbf{On the advection equation}]{\rm 
Note that the advection equation $\dot a + a \xi =0$ is not included in the reduced Lagrange-Dirac system. It is given a priori from the definition $a= g ^{-1} a _0 $. This is consistent with both the process of Euler-Poincar\'e reduction with advected parameters of \cite{HMR1998a} (see \S\ref{EP_Adv}) and with the reduction of the Hamilton-Pontryagin principle with advected parameters that we developed above (see \S\ref{HP_unconstrained} and \S\ref{HP_constrained}).}
\end{remark}

\subsection{Preliminary comments on the Hamilton-Dirac reduction}\label{subsec_preliminary_comments}

In this section we make some relevant comments concerning the choice of an appropriate Dirac reduction approach on the Hamiltonian side. 

Consider the Hamilton-Dirac system
\begin{equation}\label{HD_syst} 
\left( (g,p, \dot g, \dot p), \mathbf{d} H_{a_0}(g,p)\right) \in D_{ \Delta ^{a_0}_G}(g,p),
\end{equation} 
associated to the Hamiltonian $H_{ a _0 } : T^{\ast}G\rightarrow \mathbb{R}  $ and the constraint $ \Delta _{G}^{ a _0 }$.
Recall that the coordinate representation yields the Hamilton-d'Alembert equations
\begin{equation}\label{equ_motion_HdA}
 \dot g \in \Delta _G^{ a _0 }(g) , \quad \dot p + \frac{\partial H_{ a _0 }}{\partial g}\in \Delta _G^{ a _0 }(g) ^\circ, \quad  \frac{\partial H_{a_0}}{\partial p}=\dot g.
\end{equation} 

The trivialization of \eqref{HD_syst} yields the Hamilton-Dirac system in the form
\[
\left(( g,\mu , \xi , \dot \mu ), \overline{\mathbf{d}  H_{a_0}}(g, \mu )\right) \in \bar D_{ \Delta ^{a_0}_G}(g,\mu ),
\]
where $\overline{\mathbf{d}  H_{a_0}}\in G \times \mathfrak{g}  ^\ast \times \mathfrak{g}^\ast   \times \mathfrak{g}$ is the trivialized expression of $ \mathbf{d} H_{ a _0 } \in T^{\ast}T^{\ast}G$, given by
\[
\overline{\mathbf{d} H_{ a _0 }}(g, \mu )= \left( g, \mu , g ^{-1} \frac{\partial \bar H_{ a _0 }}{\partial g}, \frac{\partial \bar H_{ a _0 }}{\partial \mu }  \right),
\]
where $H_{a_0}(g, p)= \bar H_{a_0}(g, g ^{-1} p)$.

A reduction of this system yields
\[
\left((a, \mu , \xi , \dot \mu ), \overline{\mathbf{d}  H_{a_0}}^{/G_{ a _0 }}(a, \mu )\right) \in \bar D^{/G_{ a _0 }}_{ \Delta ^{a_0}_G}(g,\mu ),
\]
where $(a, \mu , \xi , \dot \mu ) \in  \operatorname{Orb}(a_0) \times \mathfrak{g}  ^\ast \times \mathfrak{g}  \times \mathfrak{g}^\ast $, and the reduced differential $\overline{\mathbf{d}  H_{a_0}}^{/G_{ a _0 }}(a, \mu )\in \operatorname{Orb}(a_0) \times \mathfrak{g}  ^\ast \times \mathfrak{g}  ^\ast \times \mathfrak{g}$ is given, in terms of the reduced Hamiltonian $h$, by
\begin{align*} 
\overline{\mathbf{d}  H_{a_0}}^{/G_{ a _0 }}(a, \mu )&= \left(a, \mu , g ^{-1} \frac{\partial \bar H_{a_0}}{\partial g }, \frac{\partial \bar H_{a_0}}{\partial \mu } \right) \\
&= \left( a, \mu , \frac{\delta h}{\delta a}\diamond a, \frac{\delta h}{\delta \mu }  \right) \in \operatorname{Orb}(a_0)  \times \mathfrak{g}  ^\ast \times \mathfrak{g}  ^\ast \times \mathfrak{g}.
\end{align*}
We therefore give the following definition for the reduced differential for systems on Lie groups with advected quantities.

\begin{definition} Given $a _0 \in V ^\ast $, the $G_{a _0 }$-{\bfi reduced Dirac differential} of $h:\operatorname{Orb}(a_0 ) \times \mathfrak{g}^\ast   \rightarrow \mathbb{R}  $ is
\[
\mathbf{d} ^{/ G_{a _0 }}h: \operatorname{Orb}(a_0 ) \times \mathfrak{g}^\ast   \rightarrow  \operatorname{Orb}(a_0 ) \times \mathfrak{g}^\ast  \times W^\ast, \quad \mathbf{d}^{/ G_{a _0 }}h(a, \mu )= \left( a, \mu , \frac{\delta h}{\delta a}\diamond a, \frac{\delta h}{\delta \mu }  \right).
\]
The reduced Hamilton-Dirac system associated to $(G, \Delta ^{ a _0 }_G, H_{ a _0 })$ is
\begin{equation}\label{implicit_Ham_syst_reduced}
\left((a, \mu , \xi , \dot \mu ), \mathbf{d}^{/ G_{a _0 }}h(a, \mu ) \right) \in D^{/G_{a_0}}_{ \Delta _G^{a_0}}( a, \mu),
\end{equation} 
\end{definition}
Using \eqref{condition_red_dirac} we obtain the following proposition.

\begin{proposition}  Assume that the advection equation $\dot a+ \xi a =0$ is verified. The curve $t \mapsto (a(t) , \mu(t) , \xi(t)  )$ is a solution curve of the reduced Hamilton-Dirac system \eqref{implicit_Ham_syst_reduced} 
if and only is it verifies the {\bfi implicit Lie-Poisson-Suslov equations with advected parameters}
\begin{equation}\label{ImpLiePoiSusAd} 
\frac{\delta h}{\delta a}\diamond a + \dot \mu  - \operatorname{ad}^{\ast}_\xi \mu \in (\mathfrak{g}  ^\Delta (a))^\circ, \quad \frac{\delta h}{\delta \mu }= \xi \in \mathfrak{g}  ^\Delta (a). 
\end{equation} 
\end{proposition} 

\begin{figure}[h]
\begin{center}
%\hspace{2cm}
\includegraphics[scale=.7]{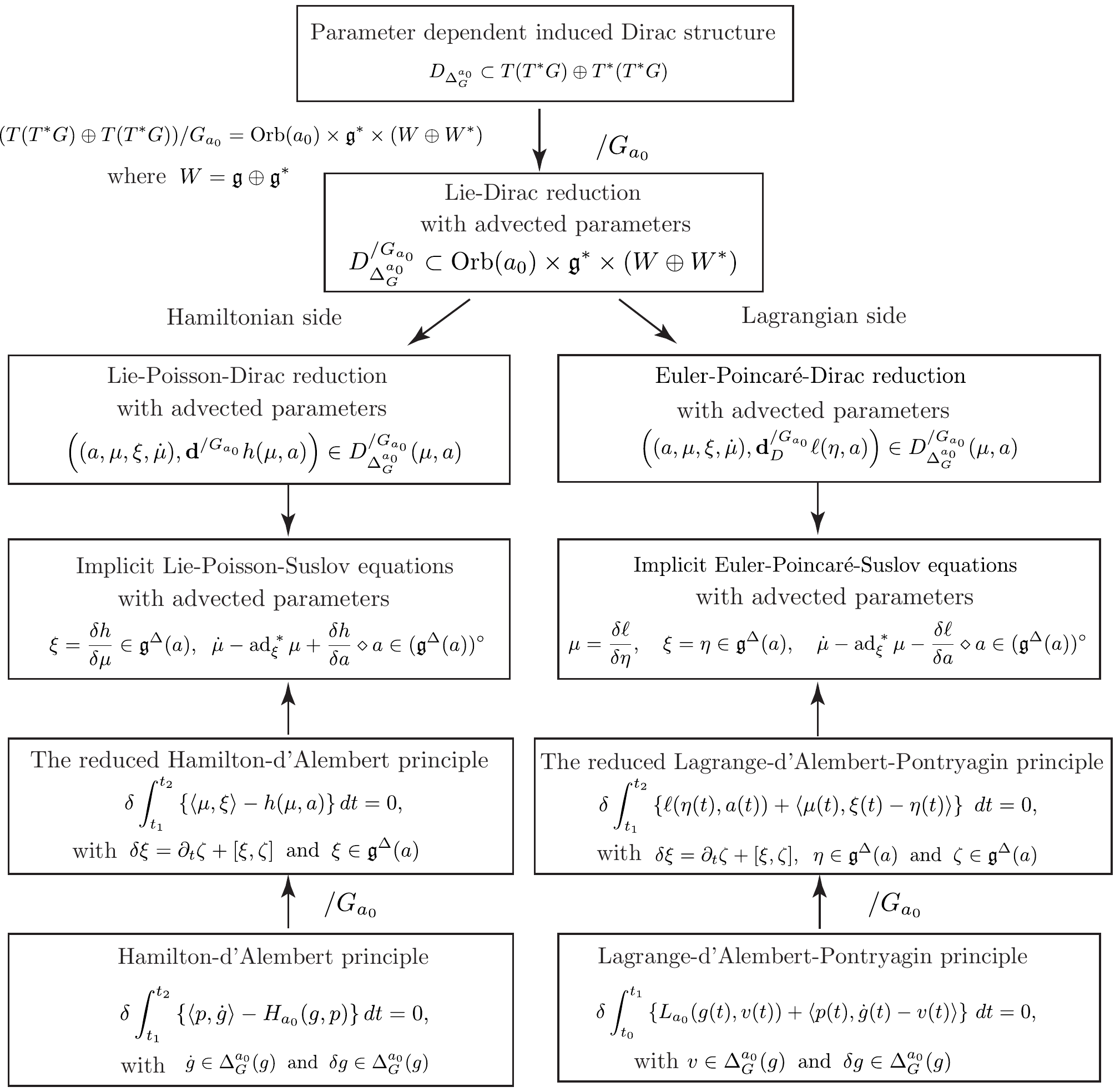}
\caption{Summary of Lie-Dirac reduction with advected parameters}
\label{Diagram_summary}
\end{center}
\end{figure}

Recall that the equations \eqref{ImpLiePoiSusAd} were obtained in \eqref{implicit_LP_SDP} via the reduced Hamilton-d'Alembert principle with advected quantities. They are here obtained via Dirac reduction.

Recall also that that this system is equivalent to its Lagrangian counterpart \eqref{implicit_EP_nonholonomic}, when the Hamiltonian is associated to the Lagrangian $\ell$ assumed to be hyperregular. This reduction process is however not fully satisfactory on the Hamiltonian side, since one has to assume that the advection equation $\dot a + a \xi =0$ holds a priori and one does not obtain it naturally from the reduction approach. However, in order to develop a consistent Dirac reduction approach on the Hamiltonian side, we have to recover, as a particular case, the ordinary Lie-Poisson reduction for semidirect product, that yields directly the advection equation $\dot a + a \frac{\delta h}{\delta \mu }=0$ without assuming it a priori.

At first glance, one could think that it is then enough to apply the Lie-Dirac reduction approach developed in \cite{YoMa2007b} for induced invariant Dirac structures on Lie groups, to the particular case when the Lie group is the semidirect product $S= G \,\circledS\, V$. This is however not the case, as we will mention in details below. In particular, such approach would yield the Lie-Poisson-Suslov equations associated with the semidirect product $ \mathfrak{s}= \mathfrak{g}  \,\circledS\, V$, which are not the equations \eqref{implicit_LP_constrained} we aim to obtain, as we already noted in Remark \ref{not_suslovSDP}.

\begin{remark}[\textbf{The case of the rolling ball type constraint}]\label{rolling_ball_remark} {\rm The Dirac reduction approach developed in the present Section evidently applies to the rolling ball type constraints described in \S\ref{subsec_the_sdp_case}. One starts with the Dirac structure $ D_{ \Delta ^{a _0 }_{G}} \subset T(T^*G) \oplus T^*(T^*G)$, where $G= K \,\circledS\, V$, and 
obtains by reduction the Dirac structure $D^{/G_{ a _0 }}_{ \Delta _G^{a _0 }}\subset \operatorname{Orb}( a _0 ) \times \mathfrak{g}^\ast \times W \oplus W ^\ast  $, $W = \mathfrak{g}  \times \mathfrak{g}  ^\ast $. We have the equivalence
\begin{equation*}\label{condition_red_dirac_rollingtype} 
\begin{array}{c} 
\left( ( a,\mu ,b,  \xi,X , \rho , \sigma ),( a,\mu ,b, \beta , \gamma , \eta, Y ) \right) \in D^{/G_{a_0}}_{\Delta_G^{a_0}}\\
\Longleftrightarrow\\
\left( \beta + \rho - \operatorname{ad}^{\ast}_\xi \mu+ X \diamond b, \gamma + \sigma - \xi b \right)  \in (\mathfrak{g}  ^\Delta (a))^\circ \quad \text{and} \quad (\eta, Y) = (\xi, X)  \in \mathfrak{g}  ^\Delta (a).
\end{array} 
\end{equation*}
The reduced Lagrange-Dirac system reads
\[
\left( ( a,\mu ,b, \xi , X, \dot { \mu },\dot b),\mathbf{d} _D ^{/ G_{ a _0 }}\ell ( a, \eta , Y)\right) \in D^{/G_{a _0 }}_{ \Delta _G^{a _0 }}( a,\mu ,b).
\]
}
\end{remark}

%%%
%%%
%%%
\section{Lie-Dirac reduction for nonholonomic systems on semi\-direct products}\label{sec_NH_SDP} 
As we already mentioned above, a satisfactory Dirac reduction approach on the Hamiltonian side should extend the properties of the Lie-Poisson reduction for semidirect products. First, the reduced Dirac structure should be obtained, similarly with the Poisson structure, from a Dirac structure on $T^*S$, where $S=G \,\circledS\, V$, and not from a Dirac structure on $T^*G$ as in the previous section (such approach was however satisfactory for the Lagrangian side). Second, the advection equation for the variable $a$ should be incorporated directly in the reduced geometric object, as in the Poisson case, and not assumed a priori as above (and on the Lagrangian side). In particular, a Dirac reduction by stages should be available in this context, as will be shown later.

Also, contrary to the usual setting in nonholonomic mechanics, one cannot start with the Dirac structure $D_{ \Delta _S}$ on $T^*S$ induced from a given $S$-invariant distribution $ \Delta  _S \subset TS$ via the lifted distribution $\Delta_{T^{\ast}S}=( T\pi_{S})^{-1} \, (\Delta_{S}) \subset T(T^{\ast}S) $ since this would yield the Lie-Poisson-Suslov equations on $\mathfrak{s} ^\ast $ which are not the desired equations in our context. We shall need the Dirac structure associated to another $S$-invariant distribution $ \Delta _{T^*S}\subset T(T^*S)$ constructed from the parameter dependent $G$-invariant distribution $ \Delta _G(g,a)\subset TG$.

In order to develop this approach systematically, we present below the Lie-Dirac reduction for $S$-invariant Dirac structures on the cotangent bundle of a Lie group $T^*S$, induced by  an arbitrary distribution $ \Delta _{T^*S} \subset T(T^*S)$ on $T^*S$.

\subsection{Lie-Dirac reduction: a more general case}\label{Lie_Dirac_gen}

In this section, we use the notation $S$ for the Lie group, since later in \S\ref{subsec_LDRBS}, we will apply this reduction to the special case of a semidirect product $S= G \,\circledS \, V$. The results of this section apply to an arbitrary Lie group $S$, not necessarily given by a semidirect product.

Let $ \Delta _{T^*S}\subset T(T^*S)$ be an arbitrary distribution on $T^*S$ (not necessarily induced by a distribution on $S$) and consider the Dirac structure $D_{ \Delta_{T^*}}$ associated with $ \Delta _{T^*S}\subset T(T^*S)$ and the canonical symplectic form, i.e., for each $p_{s} \in T^{\ast}S$,
\begin{align*} 
D_{ \Delta_{T^*S}}(p_s)&=\left\{(v_{p_s}, \alpha _{p _s })\in T_{p_s}(T^*S) \oplus T^*_{p_s}(T^*S)\mid v_{p_s}\in \Delta _{T^*S}( p_s), \right.\\
&\left.\hspace{3cm} \left\langle \alpha _{p_s }, w _{p_s } \right\rangle = \Omega (p_s) \left( v _{p_s }, w_{p _s } \right), \;\; \text{for all}\; \; w_{p_s}\in  \Delta_{T^*S}( p_s)\right\}.
\end{align*}

\paragraph{Trivialization.} Let us denote by $ \Delta _{S \times \mathfrak{s}  ^\ast }$, the trivialized expression of the distribution $ \Delta _{T^*S}$. We have $\Delta _{S \times \mathfrak{s}  ^\ast }(s, \mu )\in T_{( s, \mu )}(S \times \mathfrak{s}  ^\ast )$, for all $( s, \mu ) \in S \times \mathfrak{s}  ^\ast $.
The trivialized expression $\bar D_{\Delta _{S \times \mathfrak{s}  ^\ast }}\subset T(S \times \mathfrak{s}  ^\ast )\oplus T^*( S \times \mathfrak{s}  ^\ast )$ of the Dirac structure $D_{ \Delta_{T^*S}}$ reads
\begin{align*} 
\bar D_{\Delta _{S \times \mathfrak{s}  ^\ast }}(s, \mu )&=\left\{\left.\left((v_s, \rho ),(\beta _s, \eta  ) \right)\in (T_sS \times \mathfrak{s}  ^\ast) \oplus (T^*_sS \times \mathfrak{s} )\,\right|\, (v_s, \rho )\in \Delta _{S \times \mathfrak{s}  ^\ast },\right.\\
&\;\;\left .\left\langle \beta _g , w _g  \right\rangle + \left\langle \eta , \sigma \right\rangle = \omega (s, \mu ) \left((v_s,\rho ),(w_s,  \sigma )  \right),\;\; \text{for all}\; \; (w_s, \sigma )\in  \Delta _{S \times \mathfrak{s}  ^\ast } ,\; \sigma \in \mathfrak{s}  ^\ast \right\},
\end{align*}
where $\omega  \in \Omega ^2( S \times \mathfrak{s}  ^\ast )$ is the trivialized canonical symplectic form. We thus have the equivalence
\[ 
\begin{array}{c} 
\left( (v_s, \rho ),(\beta _s , \eta  ) \right)\in \bar D_{ \Delta _{S \times \mathfrak{s}  ^\ast }}(s, \mu )\\
\Longleftrightarrow\\
\left( -s\rho +s\operatorname{ad}^*_{s ^{-1} v _s } \mu - \beta _s, s ^{-1} v_s- \eta \right) \in \Delta _{S \times \mathfrak{s}  ^\ast }(s, \mu )^\circ \quad \text{and} \quad (v_s, \rho)\in \Delta _{S \times \mathfrak{s}  ^\ast }(s, \mu ).
\end{array} 
\]

\paragraph{Reduction.} We suppose that $ \Delta _{T^*S}$ (and hence $\Delta _{S \times \mathfrak{s}  ^\ast }$) is $S$-invariant. Thus, the distribution is completely determined by $\Delta _{S \times \mathfrak{s}  ^\ast }(e, \mu )\in T_{(e, \mu )}( S \times \mathfrak{s} ^\ast )\cong \mathfrak{s}  \times \mathfrak{s}  ^\ast $ and we have $\Delta _{S \times \mathfrak{s}  ^\ast }(s, \mu )= s \Delta _{S \times \mathfrak{s}  ^\ast }(e, \mu ) $. 

Defining the reduced Dirac structure as usual by $D^{/S}_{{\Delta} _{S \times \mathfrak{s}  ^\ast }}:= \bar D_{\Delta _{S \times \mathfrak{s}  ^\ast }}/S$, we get
\begin{align*} 
D^{/S}_{{\Delta} _{S \times \mathfrak{s}  ^\ast }}( \mu )&=\left\{ \left( ( \mu , \xi , \rho ),( \mu, \beta , \eta ) \right) \in \mathfrak{s}^\ast \times( W\oplus W^* )\mid ( \xi , \rho )\in \Delta _{S \times \mathfrak{s}  ^\ast }(e, \mu ),\right .\\
& \qquad  \left . \left\langle \beta , \zeta  \right\rangle + \left\langle \eta , \sigma \right\rangle = \left\langle \sigma , \xi  \right\rangle - \left\langle \rho , \zeta  \right\rangle + \left\langle \mu , [\xi ,\zeta ] \right\rangle,\;\; \text{for all}\; \; ( \zeta , \sigma )\in \Delta _{S \times \mathfrak{s}  ^\ast }(e, \mu ) \right\},
\end{align*} 
where $W= \mathfrak{s}  \times \mathfrak{s}  ^\ast $.
We thus have the equivalence
\begin{equation}\label{condition_red_dirac_gen} 
\begin{array}{c} 
\left( ( \mu , \xi , \rho ),( \mu ,\beta , \eta ) \right) \in D^{/S}_{{\Delta} _{S \times \mathfrak{s}  ^\ast }}( \mu )\\
\Longleftrightarrow\\
\left( - \rho + \operatorname{ad}^*_\xi \mu - \beta , \xi  - \eta \right) \in \left({\Delta} _{S \times \mathfrak{s}  ^\ast }(e, \mu )\right) ^\circ \quad \text{and} \quad (\xi , \rho )\in \Delta _{G \times \mathfrak{s}  ^\ast }(e, \mu ).
\end{array} 
\end{equation}

\paragraph{Equations of motions.} Let $H: T^*S \rightarrow \mathbb{R}  $ be a Hamiltonian function. The Hamilton-Dirac system
\[
\left( (s,p,\dot s, \dot p), \mathbf{d} H(s,p) \right) \in D_{\Delta _{T^*S}}(s,p)
\]
yields locally the conditions
\begin{equation}\label{equ_motion} 
( \dot s, \dot p)\in \Delta _{T^*S}(s,p), \quad \left( \frac{\partial H}{\partial s}+\dot p, \frac{\partial H}{\partial p}- \dot s  \right) \in  \Delta _{T^*S}(s,p) ^\circ.
\end{equation} 
Assuming that $H$ is $S$-invariant and proceeding as in \S\ref{LiePoisSusRed_Section}, we get the reduced Hamilton-Dirac system
\[
(( \mu , \xi , \dot \mu ), ( \mathbf{d} ^{/S}h( \mu ))\in D^{/S}_{{\Delta} _{S \times \mathfrak{s}  ^\ast }}( \mu ), 
\]
where $\mathbf{d}^{/S}h:=(\overline{\mathbf{d}H})^{/S}: \mathfrak{s}^{\ast} \rightarrow \mathfrak{s}^{\ast} \times (\mathfrak{s}^{\ast} \oplus \mathfrak{s})$ is given by $\mathbf{d}^{/S}h( \mu )=\left(\mu,0,\frac{\delta{h}}{\delta{\mu}}\right)$. We thus get the equations
\begin{equation}\label{LP_Suslov_gen} 
\left( - \dot \mu  + \operatorname{ad}^*_\xi \mu , \xi  - \frac{\delta h}{\delta \mu }  \right) \in {\Delta} _{S \times \mathfrak{s}  ^\ast }(e, \mu ) ^\circ \quad \text{and} \quad ( \xi , \dot \mu  )\in \Delta _{S \times \mathfrak{s}  ^\ast }(e, \mu ).
\end{equation}

\begin{remark}[\textbf{Recovering the case of the lifted distribution}]
{\rm In the particular case when $ \Delta _{T^*S}$ is the lifted distribution associated to a given $S$-invariant distribution $ \Delta _S\subset TS$, that is, $ \Delta _{T^*S}= \left( T  \pi \right) ^{-1} ( \Delta _S)$, we have ${\Delta} _{S \times \mathfrak{s}  ^\ast }= \mathfrak{s}  ^\Delta \times \mathfrak{s}  ^\ast $ and thus $({\Delta} _{S \times \mathfrak{s}  ^\ast })^\circ= \left( \mathfrak{s}  ^\Delta \right) ^\circ \times \{0\}$, so that \eqref{condition_red_dirac_gen} consistently recovers
\[
\beta + \rho - \operatorname{ad}^*_\xi \mu \in (\mathfrak{s}  ^\Delta)^\circ, \quad  \eta = \xi  \in \mathfrak{s}  ^\Delta
\]
and hence \eqref{LP_Suslov_gen} recovers the implicit Lie-Poisson-Suslov equations \eqref{imp_LPS}:
\[
\dot \mu  - \operatorname{ad}^*_\xi \mu \in (\mathfrak{s}  ^\Delta)^\circ, \quad \frac{\delta h}{\delta \mu } =\xi  \in \mathfrak{s}  ^\Delta.
\]
}
\end{remark}

\subsection{Lie-Dirac reduction by stages on semidirect products}\label{subsec_LDRBS}

We now apply the Lie-Dirac reduction described in \S\ref{Lie_Dirac_gen} to the particular case when the Lie group $S$ is the semidirect product $S= G \,\circledS\, V$. We then show that by choosing an appropriate $S$-invariant distribution $ \Delta _{T^*S}\subset T(T^*S)$ constructed from the $G$-invariant and parameter dependent distribution $ \Delta _G(g,a)\subset TG$, we recover an implicit version of the nonholonomic Lie-Poisson equations that naturally includes the advection equation, without having to formulate it as an hypothesis. In order to illustrate the analogy with the usual Poisson reduction for semidirect products (see Theorem \ref{LPSD} and Remark \ref{Link_RBS}) our approach will be done using a reduction by stages. We first reduce by the normal subgroup $V$ of $S$, and then by $G=S/V$.

\paragraph{The case of an arbitrary invariant distribution.} In order to find the appropriate $S$-invariant distribution needed on $ T^*S$, we first consider an arbitrary $S$-invariant distribution $ \Delta _{T^*S}\subset  T(T^*S)$ and its associated Dirac structure, for $(g,p,u,a_0 ) \in T^{*}G \times T^{\ast}V \cong T^{\ast}S$, 
\[
D_{\Delta _{T^*S}}(g,p,u,a_0 )\subset T_{(g,p,u,a_0 )}T^*S\oplus T^*_{(g,p,u,a_0 )}T^*S.
\]
Note that here $ a _0 $ denotes an arbitrary element in $ V ^\ast $ and is not fixed.
Given a Hamiltonian $H: T^*S \rightarrow \mathbb{R}  $, we know from \eqref{equ_motion} that the associated Hamilton-Dirac system
\[
\left( (g,p,\dot g,\dot p, u , a _0,\dot u, \dot a _0  ), \mathbf{d} H(g,p,u , a _0 ) \right) \in D_{\Delta _{T^*S}}(g,p,u , a _0 ),
\]
yields locally the conditions
\begin{equation}\label{equ_motion_SDP} 
( \dot g, \dot p, \dot u, \dot a _0 )\in \Delta _{T^*S}, \quad \left( \frac{\partial H}{\partial g}+\dot p, \frac{\partial H}{\partial u}+ \dot a _0  , \frac{\partial H}{\partial p}- \dot g, \frac{\partial H}{\partial a _0 }- \dot u  \right) \in \left(  \Delta _{T^*S}\right) ^\circ.
\end{equation} 

As before, let us denote by $\bar{D}_{ \Delta _{S \times \mathfrak{s} ^\ast}}$
the trivialized Dirac structure induced on $S \times \mathfrak{s}^\ast $. From \eqref{LP_Suslov_gen}, we know that a reduction by $S$ yields the Dirac system
\begin{equation}\label{Dirac_eqn_S} 
\left(( \mu , a, \xi , v, \dot \mu , \dot a), \mathbf{d} ^{/S}h( \mu , a) \right) \in D^{/S}_{\Delta _{S \times \mathfrak{s}  ^\ast }}( \mu ,a)
\end{equation} 
with the associated equations
\begin{equation*}\label{LP_Suslov_SDP_gen}
\left( -( \dot \mu , \dot a)+ \operatorname{ad}^*_{( \xi , v)}( \mu , a), ( \xi , v)- \left( \frac{\delta h}{\delta \mu } , \frac{\delta h}{\delta a} \right)  \right) \in \Delta _{S \times \mathfrak{s} ^\ast }(e,0, \mu , a)^\circ
\end{equation*} 
and
\begin{equation*}
( \mu , a, \xi , v, \dot \mu , \dot a)\in \Delta _{S \times \mathfrak{s} ^\ast }(e,0, \mu , a).
\end{equation*} 
\paragraph{Definition of the distribution on $T^*S$.} 
In order to recover from  \eqref{Dirac_eqn_S}  the implicit Lie-Poisson-Suslov equations \eqref{implicit_LP_SDP} together with the advection equation $ \partial _t a + \xi a=0$, we need to choose the $S$-invariant distribution $ \Delta _{T^*S}\subset T(T^*S)$ in such a way that the equality
\begin{equation}\label{Formula_reduced_DeltaS} 
\Delta _{S \times \mathfrak{s} ^\ast }(e,0, \mu , a)= \mathfrak{g}  ^\Delta (a) \times T_0V \times T_{( \mu , a)}\mathfrak{s} ^\ast \subset T_{(e,0, \mu , a)}(S \times \mathfrak{s} ^\ast )
\end{equation} 
holds. By $S$-invariance, this means that, at $(g,u, \mu , a)\in S \times \mathfrak{s} ^\ast $, we have
\begin{align*} 
\Delta _{S \times \mathfrak{s} ^\ast }(g,u, \mu , a)&= (g,u) \left( \mathfrak{g}  ^\Delta (a) \times T _0 V \times T_{( \mu , a)}\mathfrak{s} ^\ast\right) \\
&= \Delta _G(g, ga) \times T _u V \times T_{( \mu , a)}\mathfrak{s} ^\ast \subset T_{(g, u, \mu , a)}(S \times \mathfrak{s} ^\ast ).
\end{align*} 
Using the expression for the trivialization $T^*S \rightarrow S \times \mathfrak{s} ^\ast $ given by, 
\begin{align*} 
&(g,p,u, a _0 )\in T^*S\\
& \quad \mapsto (g,u, \mu , a)=\left( (g,u), (g,u) ^{-1} (g,p,u, a _0 ) \right)= \left( (g,u), g ^{-1} p , g ^{-1} a_0  \right)  \in S \times \mathfrak{s} ^\ast,
\end{align*}
we obtain that $ \Delta _{T^*S}$ has to be defined as follows.

\begin{definition}\label{Definition_delta_star}  The distribution $ \Delta _{T^*S}\subset T(T^*S)$ on $T^*S$ associated to a given constraint distribution $ \Delta _G^{ a _0 }\subset TG$ on $G$ is defined as
\begin{equation}\label{Def_Delta_TstarS} 
\Delta _{T^*S}:= (T \pi _S) ^{-1} (\Delta ^{a _0 }_G \times TV)=\left( (T \pi _G ) ^{-1}  \Delta _G^{a_0}\right)  \times T(T^*V).
\end{equation} 
\end{definition} 

In more details, we have
\begin{equation*}\begin{aligned} 
\Delta _{T^*S}(g,p,u,a_0 )&= \left( T_{(g,p,u, a _0 )} \pi _S\right) ^{-1} \left( \Delta _G^{a_0}(g) \times T _u V \right)\\
&= \left( \left( T_{(g,p)} \pi_G \right) ^{-1}\left( \Delta _G^{a_0}(g)\right) \right)  \times T_{(u, a _0 )}T^*V \subset T_{(g,p,u,a_0 )} T^*S,
\end{aligned}
\end{equation*}  
where it is important to note that the vector space $\Delta _{T^*S}(g,p,u,a_0 )$ depends on $a_{0} \in V^{\ast}$.
In this case, the implicit Hamiltonian system \eqref{equ_motion_SDP} reads
\begin{equation*}\label{equ_motion_SDP_particular} 
 \dot g \in \Delta _G^{ a _0 }(g) , \quad \frac{\partial H}{\partial g}+\dot p\in \Delta _G^{ a _0 }(g) ^\circ, \quad \frac{\partial H}{\partial u}+ \dot a _0 =0 ,\quad  \frac{\partial H}{\partial p}- \dot g=0, \quad \frac{\partial H}{\partial a _0 }- \dot u  =0.
\end{equation*} 

\paragraph{Dirac reduction by stages: first stage.} The first stage reduction consists in reducing the Dirac structure by the normal subgroup $V$ of $S$. We note that the subgroup action of $w \in V$ on $S$ and $T^*S$ reads
\[
w(g,u)=(g, w+u), \quad\text{and}\quad  w (g, p, u, a_0)= (g,p,w+u,a_0).
\]
It follows that the Hamiltonian $H:T^*S \rightarrow \mathbb{R}  $, $H=H(g,p,u,a_0)$ is $V$-invariant if and only if it does not depend on the variable $u$. Similarly, $ \Delta _{T^*S}$ is $V$-invariant if and only if the vector space $ \Delta _{T^*S}(g,p,u,a _0 )$ does not depend on the variable $u$. For such a Hamiltonian, we have
\[
\mathbf{d} H:T^*S \rightarrow T(T^*S), \quad \mathbf{d} H(g,p,u,a_0)
=\left(g,p,\frac{\partial H}{\partial g} , \frac{\partial H}{\partial p} ,u,a_0,0, \frac{\partial H}{\partial a _0 } \right).
\]
Using the equalities
\begin{align*} 
T^*S/V=T^*G \times V^*, \quad &T(T^*S)/V= T(T^*G) \times V^* \times V  \times V ^\ast\\
&T^*(T^*S)/V= T^*(T^*G) \times V^* \times V^*  \times V,
\end{align*} 
we obtain the reduced derivative $\mathbf{d} ^{/V} H: T^*G \times V^* \rightarrow T^*(T^*G) \times V^* \times V^*  \times V$
\[
\mathbf{d} ^{/V} H(g,p,a _0 )=\left(g,p,\frac{\partial H}{\partial g} , \frac{\partial H}{\partial p} ,a_0,0, \frac{\partial H}{\partial a _0 } \right).
\]
Note that we have the equality
\[
\left( T(T^*S) \oplus T^*(T^*S) \right) /V= T( T^*G) \times V^* \times V \times V^* \oplus T^*(T^*G) \times V^* \times V^* \times V
\]
as vector bundles over $T^*G \times V^*$.

Since both $ \Delta _{T^*S}$  and the canonical symplectic form on $T^*S$ do not depend on $u$, the fiber $D_{ \Delta _{T^*S}}(g,p,u, a _0 )$ of the associated Dirac structure $D_{ \Delta _{T^*S}}$ does not depend on $u$, either. Therefore, since the $V$-action does not affect the vector fibers above $T^*S$, the reduced Dirac structure $D^{/V}_{ \Delta _{T^*S}}:= \left( D_{ \Delta _{T^*S}} \right) /V$, with
\begin{equation}\label{Dirac_eqn_V} 
D^{/V}_{ \Delta _{T^*S}}(g,p,a_0)\subset T_{(p,g)}( T^*G) \times \{a _0 \} \times V \times V^* \oplus T^*_{(g,p)}(T^*G) \times \{a _0 \} \times V^* \times V,
\end{equation} 
has the same expression as the Dirac structure $D_{ \Delta _{T^*S}}$.

The reduced Hamilton-Dirac system reads
\[
\left( ( g,p,\dot g, \dot p, a _0 ,\dot u,\dot a _0), \mathbf{d} ^{/V}H (g,p,a _0 ) \right) \in D^{/V}_{ \Delta _{T^*S}}(g,p, a _0 )
\]
and the associated equations are
\begin{equation*}\label{equ_motion_SDP_redV} 
( \dot g, \dot p, \dot u, \dot a _0 )\in \Delta _{T^*S}^{/V}(g,p, a _0 ), \quad \left( \frac{\partial H}{\partial g}+\dot p, \dot a _0  , \frac{\partial H}{\partial p}- \dot g, \frac{\partial H}{\partial a _0 }- \dot u  \right) \in  \Delta^{/V} _{T^*S}(g,p, a _0) ^\circ ,
\end{equation*}
where $(g,p, a _0 )\in T^* G \times V^*$, and we denote by $ \Delta _{T^*S}^{/V}(g,p, a _0 )$ the quotient of $ \Delta _{T^*S}(g,p, u, a _0 )$, in which $ \Delta _{T^*S}$ does not depend on the variable $u \in V$.

For the particular case in which the distribution $ \Delta _{T^*S}\subset T(T^*S)$ is induced by the distribution $ \Delta _{G}^{ a _0 }\subset TG$ via \eqref{Def_Delta_TstarS}, the equations are
\begin{equation*}\label{equ_motion_SDP_particular_redV} 
 \dot g \in \Delta _G^{ a _0 }(g) , \quad \frac{\partial H}{\partial g}+\dot p\in \Delta _G^{ a _0 }(g) ^\circ, \quad  \dot a _0 =0 ,\quad  \frac{\partial H}{\partial p}- \dot g=0, \quad \frac{\partial H}{\partial a _0 }- \dot u  =0.
\end{equation*} 
Since $\dot a_0=0$, and $H$ does not depend on $u$, the last equation decouples from the others, and the first four equations are equivalent to the implicit Hamilton-d'Alembert equations \eqref{equ_motion_HdA} for $H_{a_0}$.

\paragraph{Dirac reduction by stages: second stage.} The expression of the reduced Dirac structure $D^{/S}_{ \Delta _{T^*S}}$ can be either obtained by reducing the Dirac structure $D^{/V}_{ \Delta _{T^*S}}$ by the group $G$, or by reducing the Dirac structure $D_{ \Delta _{T^*S}}$ by the group $S$. This follows from the fact that $V$ is a normal subgroup of $S=G\,\circledS \,V$. We shall explain the reduction of $D_{ \Delta _{T^*S}}$ by $S$ in the following.

By applying the general result obtained in \eqref{condition_red_dirac_gen} for $S$ to the semidirect product $S= G \,\circledS\, V$, we obtain the following description of the reduced Dirac structure $D^{/S}_{ \Delta _{T^*S}}\subset \mathfrak{s} ^\ast \times (W \oplus W ^\ast )$, with $W= \mathfrak{s} \times \mathfrak{s} ^\ast $:
\begin{equation}\label{condition_red_dirac_S} 
\begin{array}{c} 
\left( ( \mu , a, \xi , w, \rho  , b),( \mu , a, \beta  , c, \eta  , v)\right) \in D^{/S}_{{\Delta} _{S \times \mathfrak{s}  ^\ast }}( \mu,a )\\
\Longleftrightarrow\\
\left(- (\rho,b) + \operatorname{ad}^*_{( \xi ,w)}(\mu,a) - (\beta,c) , ( \xi ,w)-( \eta ,v)  \right) \in {\Delta} _{S \times \mathfrak{s}  ^\ast }(e, \mu )^\circ\\
\text{and} \quad ( \xi ,w, \rho ,b )\in \Delta _{S\times \mathfrak{s}^\ast }(e,0, \mu ,a),\\
\Longleftrightarrow\\
-\rho+ \operatorname{ad}^*_ \xi\mu- w \diamond a - \beta \in (\mathfrak{g}  ^ \Delta (a))^\circ ,\;\; b +  \xi a + c =0, \;\;  \xi = \eta ,\;\;  w=v, \quad  \xi \in \mathfrak{g}  ^ \Delta (a),
\end{array} 
\end{equation} 
where in the second equivalence we used $\Delta _{S \times \mathfrak{s} ^\ast }(e,0,\mu , a)=\mathfrak{g}  ^\Delta (a) \times T _0 V \times T_{( \mu , a)}\mathfrak{s} ^\ast$, see \eqref{Formula_reduced_DeltaS}, and $\Delta _{S \times \mathfrak{s} ^\ast }(e,0,\mu , a)^\circ =(\mathfrak{g}  ^\Delta (a))^ \circ \times \{0\}\times  \{0\}$, as well as the expression of the operator $ \operatorname{ad}^*$ for the semidirect product. 

From \eqref{condition_red_dirac_S} it follows that the solution curves of the reduced Hamilton-Dirac system 
\[
\left(( \mu , a, \xi , w, \dot \mu , \dot a), \mathbf{d} ^{/S}h( \mu , a) \right) \in D^{/S}_{\Delta _{S \times \mathfrak{s}  ^\ast }}( \mu ,a)
\]
verify the equations
\[
\dot \mu - \operatorname{ad}^*_ \xi \mu + w \diamond a \in ( \mathfrak{g}  ^ \Delta (a))^\circ, \;\; \dot a + \xi a=0,\;\; \xi = \frac{\delta h}{\delta \mu }, \;\; w= \frac{\delta h}{\delta a},\;\; \xi \in \mathfrak{g}  ^ \Delta (a).
\]
These are exactly the desired equations, namely, the implicit Lie-Poisson version of the Euler-Poincar\'e-Suslov equations \eqref{implicit_EP_nonholonomic} (compare with \eqref{ImpLiePoiSusAd}) which contains in addition the advection equation for the advected parameter $a$.

As we already mentioned, these equations cannot be obtained by reduction of the $S$-invariant Dirac structure in $T(T^*S)\oplus T^*(T^*S)$ induced by a distribution $ \Delta _S \subset TS$. Indeed, we had to use the Dirac structure induced by the distribution $ \Delta _{T^*S}\subset T(T^*S)$ given in \eqref{Def_Delta_TstarS}, which is not of the lifted from $\left( T \pi  \right) ^{-1} ( \Delta _S)$.

The results obtained in this section are briefly summarized in the following theorem.

\begin{theorem} Let $ \Delta_G^{ a _0 }\subset TG$, $ a _0 \in V ^\ast $ be a family of constraint distributions on $G$ such that 
$ \Delta _G(hg, ha)= h \Delta _G(g,a)$, for all $h \in G$. Consider the semidirect product $S= G \,\circledS\, V$ and let $ \Delta _{T^*S}$ be the distribution on $T^*S$ associated with $ \Delta _G^{ a _0 }$, as defined in \eqref{Def_Delta_TstarS}. Let $H:T^*S \rightarrow \mathbb{R}  $ be a $S$-invariant Hamiltonian and let $h: \mathfrak{s} ^\ast \rightarrow \mathbb{R}  $ be the associated reduced Hamiltonian. Then we have the following results.
\begin{itemize}
\item[\bf{(i)}] The distribution $ \Delta _{T^*S}$ and the associated Dirac structure $D_{ \Delta _{T^*S}}$ on $T^*S$ are $S$-invariant.
\item[\bf{(ii)}] The first stage reduction yields the reduced Dirac structure $D^{/V}_{ \Delta _{T^*S}}$ given in \eqref{Dirac_eqn_V} on the first stage reduced Pontryagin bundle $(T(T^*S )\oplus T^*(T^*S))/V$.
\item[\bf{(iii)}] The solution curves of the reduced Hamilton-Dirac system associated to $D^{/V}_{ \Delta _{T^*S}}$ verify the  implicit Hamilton-d'Alembert equations
\[
\dot g \in \Delta _G^{ a _0 }(g) , \quad \frac{\partial H}{\partial g}+\dot p\in \Delta _G^{ a _0 }(g) ^\circ, \quad  \dot a _0 =0 ,\quad  \frac{\partial H}{\partial p}- \dot g=0, \quad \frac{\partial H}{\partial a _0 }- \dot u  =0.
\]
\item[\bf{(iv)}] The second stage reduction, namely the reduction of $D^{/V}_{ \Delta _{T^*S}}$ by $G$,  yields the reduced Dirac structure $D^{/S}_{\Delta _{T^*S }}$ given in \eqref{condition_red_dirac_S} on the second stage reduced Pontryagin bundle $(T(T^*S) \oplus  T^*(T^*S))/V/G=(T(T^*S )\oplus  T^*(T^*S))/S$. This follows from the fact that $V$ is a normal subgroup of $S=G \,\circledS\, V$.
\item[\bf{(v)}] The solution curves of the reduced Hamilton-Dirac system associated to $D^{/S}_{ \Delta _{T^*S}}$ verify the implicit Lie-Poisson-Suslov equations with advected parameters together with the advection equation
\[
\dot \mu - \operatorname{ad}^*_ \xi \mu + w \diamond a\in \left( \mathfrak{g}  ^ \Delta (a) \right) ^\circ, \quad \dot a + \xi a =0, \quad \xi = \frac{\delta h}{\delta \mu }, \quad w= \frac{\delta h}{\delta a} , \quad 
\xi \in \mathfrak{g}  ^\Delta(a).
\]
\end{itemize} 
\end{theorem}

\begin{figure}[h]
\begin{center}
%\hspace{2cm}
\includegraphics[scale=.5]{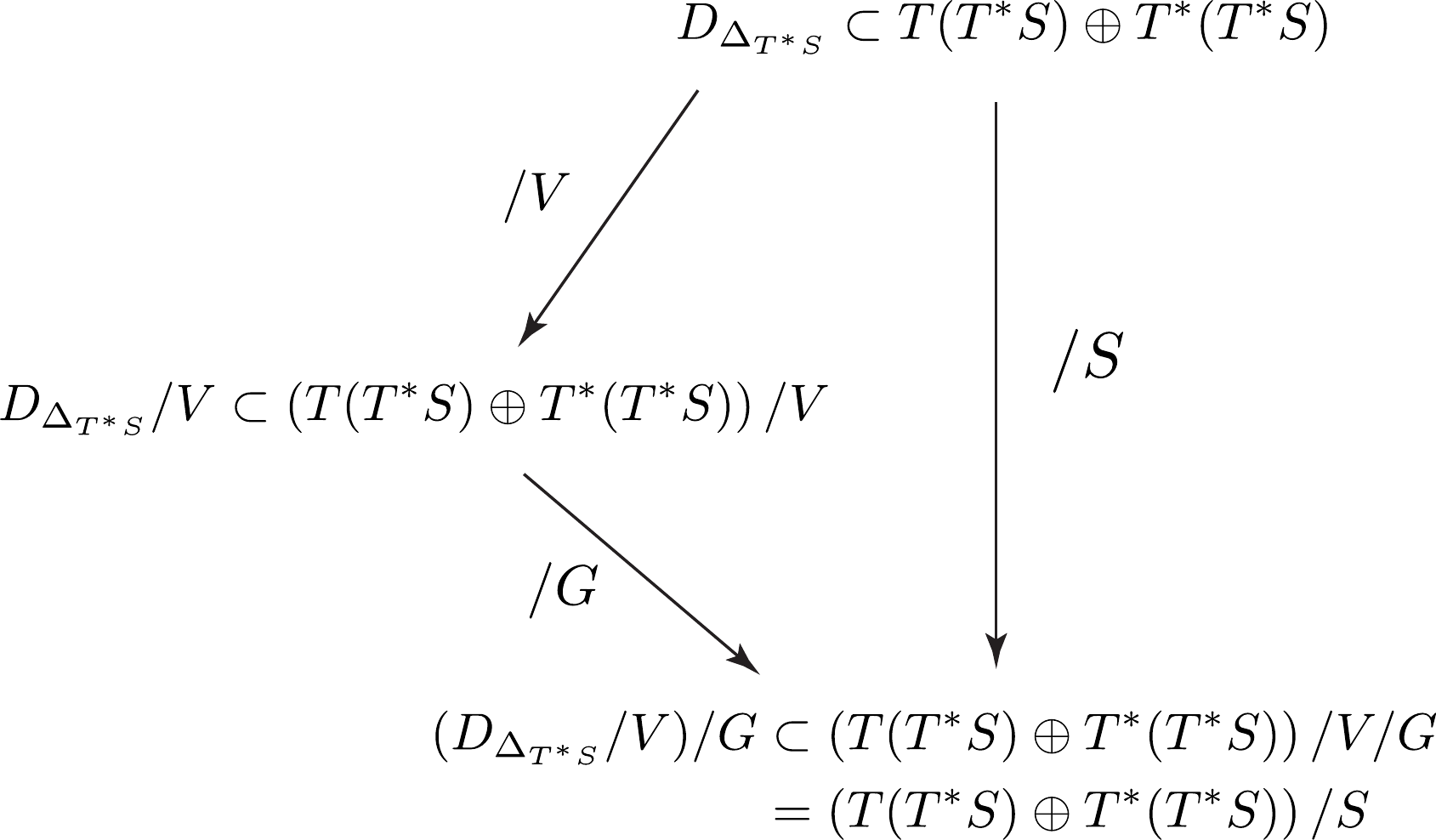}
\caption{Dirac reduction by stages for semidirect products}
\label{DiracBYStages}
\end{center}
\end{figure}

\subsection{Relation with the Lagrange-Dirac side}

The aim of this section is to relate the Dirac structure $ D_{\Delta _{T^*S}}\subset T(T^*S)\oplus T^*(T^*S)$ defined above for the Hamiltonian side and the family of Dirac structures with parameter $D_{ \Delta ^{a _0 }}\subset T(T^*G)\oplus T^*(T^*G)$ used in the Lagrangian side and in \S\ref{subsec_preliminary_comments}, by using exclusively a Dirac reduction point of view. 

\medskip

Recall that in the standard case of Hamiltonian and Lagrangian dynamics, which corresponds to the choice of the canonical Dirac structure, one passes from the canonical Hamiltonian description associated to $\bar H$ on $T^*S$ to the canonical Hamilton description associated to $H_{a _0 }$ on $T^*G$ by a symplectic reduction by the subgroup $V$ at the particular value $a _0 \in V ^\ast $ of the momentum map $ \mathbf{J}_V  : T^*S \rightarrow V ^\ast $. From the Dirac point of view, this means that one has to reduce the canonical Dirac structure on $T^*S$ in order to get the canonical Dirac structure on $T^*G$. The reduction approach that has to be used here to relate the noncanonical Dirac structures $D_{ \Delta _{T^*S}}$ and $D_{ \Delta ^{a _0 }_G}$ is therefore the one that parallels symplectic reduction, as developed in \cite{BS2001}.

\medskip

Roughly speaking the reduction process goes as follows. Given a Dirac structure $D$ on a manifold $M$ and a free and proper Lie group action of $G$ on $M$, one first assumes that  the action admits an equivariant momentum map relative to $D$, that is, there exists a $G$-equivariant map $ \mathbf{J} : M \rightarrow \mathfrak{g}  ^\ast $ such that $( \xi _M , \mathbf{d} \mathbf{J} _ \xi ) \in D$, for all $\xi \in \mathfrak{g}  $.  The process goes in two steps.

First, if $ \mu \in \mathfrak{g}  ^\ast $ is a regular value of $ \mathbf{J} $ then $M_{\mu}:= \mathbf{J}^{-1} (\mu) \subset M$ is a submanifold. If the dimension of the vector subspaces $D(m) \cap (T_m {M_{\mu}} \times T^*_m M|_{M_{\mu}})\subset T_m M_{\mu} \times T^*_m M|_{M_{\mu}}$ is independent of $m\in M_{\mu}$, then these vector spaces naturally induce a Dirac structure $D_{M_{\mu}} \subset TM_{\mu} \oplus T^*M_{\mu}$.

Second, one observes that the Dirac structure $D_{M_{\mu}}$ is $G _\mu $-invariant, where $G_ \mu =\{g \in G \mid \operatorname{Ad}^*_ g \mu = \mu \}$ is the coadjoint isotropy subgroup  of $ \mu $. Under appropriate conditions (we do not detail them because they are trivially satisfied in our example, and refer to \cite{BS2001}), one obtains a reduced Dirac structure $D _\mu \subset T(M_{\mu}/G _\mu ) \oplus  T^\ast (M_{\mu}/G _\mu )$ on $M_{\mu}/G _\mu = \mathbf{J} ^{-1} (\mu) /G _\mu $ described in terms of its local sections by 
\begin{equation}\label{2nd_red} 
\begin{aligned} 
(\mathfrak{D}_ \mu )_{loc}:&=\{(X, \alpha ) \in \mathfrak{X}  _{loc}(M_{\mu}/G _\mu ) \times \Omega ^1 _{loc}(M_{\mu}/G_ \mu )\mid \exists \,(Y, \beta  )\in \mathfrak{D}  _{loc}, \\
& \qquad \qquad \qquad \qquad \qquad \qquad \;\text{such that}\; T \pi \circ Y = X \circ \pi , \; \pi ^\ast \alpha = \beta  \},
\end{aligned} 
\end{equation}
see \cite{BS2001}.

\medskip

We shall now apply this reduction approach to the Dirac structure associated to
\[
\Delta _{T^*S}(g,p,u,a_0 )= \left( T_{(g,p,u, a _0 )} \pi \right) ^{-1} \left( \Delta _G^{a_0}(g) \times T _u V \right) = \left( T_{(g,p)} \pi _G \right) ^{-1} ( \Delta ^{ a _0 }(g)) \times T_{(u, a _0 )} T^*V
\]
on $T^*S$. Note that in general, if $ \mathbf{J} : P \rightarrow \mathfrak{g}  ^\ast $ is a momentum map for a symplectic action of $G$ on the symplectic vector space $(P, \omega )$, and if one considers the Dirac structure $D$ associated with $ \omega $ and a distribution $ \Delta \subset TP$ as in \eqref{DiracManifold}, then $ \mathbf{J} $ is also a Dirac momentum map, provided we have $ \xi _P (x) \in \Delta (x)$, for all $x \in P$ and $ \xi \in \mathfrak{g}  $. In our case, this hypothesis is verified since the $V$ action on $T^*S$ reads $(g,p,u, a _0 ) \mapsto (g,p,u +w, a _0 )$, so the infinitesimal generator $w_V$ at $(g,p,u, a _0 )$ is given by $(0,0,w, 0 )$ which belongs to $\Delta _{T^*M_{\mu}}(g,p,u,a_0 )$.
Choosing a momentum value $ a _0 \in V ^\ast $, we have $ \mathbf{J} _V ^{-1} ( a _0 )= T^*G \times V \times \{ a _0 \}\simeq T^*G \times V$.

According to the process recalled above, the Dirac structure obtained on $T^*G \times V$ is given by
\begin{align} \label{def_restricted} 
&\left( D_{ \Delta _{T^*S}} \right) _{T^*G \times V}(g,p,u)\nonumber\\
& \quad = D_{\Delta _{T^*S}}(g,p,u, a _0 ) \cap \left( T_{(g,p,u)} (T^*G \times V) \times T^*_{(g,p,u)}(T^*(G \times V)|_{T^*G \times V}) \right),
\end{align}
where we recall that $ a _0 \in V ^\ast $ is fixed and we note that the dimension condition is verified. First we observe that 
\[
\left( (g,p, u, a _0, \dot g, \dot p, \dot u, \dot a _0 ), (g,p, u, a _0, \alpha , v, b, w) \right)  \in D_{ \Delta _{T^*S}}(g,p,u, a _0 )
\]
if and only if $(g,p,u, a _0 , \dot g, \dot p, \dot u, \dot a _0 ) \in \Delta _{T^*S}$ and $(g,p, u, a _0, \alpha +\dot p, v-\dot g, b+ \dot a _0 , w-\dot u) \in \Delta _{T^*S} ^\circ$, which is equivalent to $(g, \dot g) \in \Delta ^{ a _0 }(g)$, $ \alpha +\dot p \in \Delta ^{a _0}(g) ^\circ$, $v=\dot g$, $b+\dot a _0 =0$, $w=\dot u$. Formula \eqref{def_restricted} means that in addition we have  $\dot a _0 =0$ and $w=0$. So we deduce that 
\begin{align}\label{first_red}  
&\left( D_{ \Delta _{T^*S}} \right) _{T^*G \times V}(g,p,u)\nonumber\\
&=\{(g,p,u,\dot g, \dot p, \dot u),(g,p,u, \alpha ,v, b) \in T_{(g,p,u)}(T^*G \times V) \oplus T^\ast _{(g,p,u)}(T^*G \times V)\mid\\
&\qquad \qquad  \qquad \qquad \qquad \qquad  \qquad \qquad \qquad \qquad  ((g,p,\dot g, \dot p),(g,p, \alpha ,v))\in D_{ \Delta ^{a _0 }_G}\}.\nonumber
\end{align} 

To carry out the second step, we recall that the isotropy group is $V_{ a _0 }=V$ and that $(T^*G \times V)/V= T^*G$. Then, a direct application of formula \eqref{2nd_red} to the Dirac structure \eqref{first_red} shows that the reduced Dirac structure is given by $D_{ \Delta ^{a _0 }_G}\subset T(T^*G) \oplus T^*(T^*G)$. We thus have obtained the following result.

\begin{theorem} The Dirac structure $D_{ \Delta _G ^{a _0 }} \subset T(T^*G) \oplus T^*(T^*G)$ associated to the nonholonomic constraint $ \Delta _G^{a _0 }\subset TG$ is obtained from the $S$-invariant Dirac structure $D_{ \Delta _{T^*S}} \subset T(T^*S) \oplus T^*(T^*S)$ defined in \eqref{Def_Delta_TstarS} by the Dirac reduction process of \cite{BS2001} with respect to the momentum value $ a _0 \in V ^\ast $.
\end{theorem} 

The relation between the various Dirac reductions involved is illustrated in the following diagram.
\begin{figure}[h]
\begin{center}
%\hspace{2cm}
\includegraphics[scale=.75]{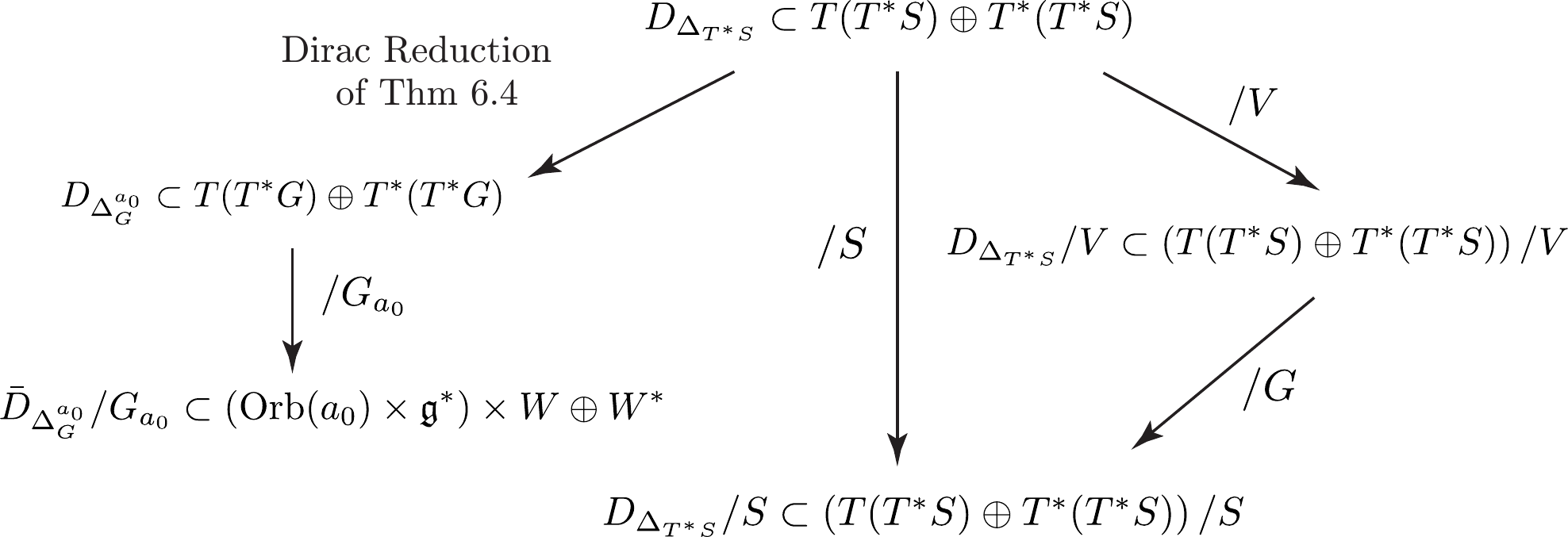}
\caption{Relation between Dirac reductions}
\label{DiracBYStages}
\end{center}
\end{figure}

\subsection{The Hamilton-d'Alembert principle in phase space}\label{PSVP_TstarS} 

We shall now develop the Hamilton-d'Alembert principle in phase space naturally associated to the approach presented in \S\ref{subsec_LDRBS}. It extends the Hamilton-d'Alembert principle developed in \S\ref{Hamiltonian_side} in the sense that it naturally implies the advection equation without having to assume it a priori.

In order to obtain this principle, it suffices to consider the Hamilton-d'Alembert principle in phase space $T^*S \cong T^{*}G \times T^{\ast}V$ (instead of $T^*G$ as in \S\ref{Hamiltonian_side}) for a Hamiltonian $\bar H=\bar H(g,p,u,a _0 ): T^*S \rightarrow \mathbb{R}  $, and with constraint distribution $ \Delta^{ a _0 } _G(g) \times T_uV \subset TS$ (see, \eqref{Def_Delta_TstarS}). Note that this principle is slightly more general than the Hamilton-d'Alembert principle for nonholonomic mechanics developed in \S\ref{Hamiltonian_side}, since the distribution constraint also depends on the fiber in $T^*S$ ($ a _0 $, here) and not only on the base point $(g,u) \in S$.

The principle thus reads
\begin{equation}\label{PSP_TstarS}
\delta \int_{ t _1 }^{ t _2 } \left\{ \left\langle p, \dot g \right\rangle + \left\langle a _0 , \dot u \right\rangle -\bar H(g,p, u, a _0 )\right\} dt=0,\quad (\dot g, \dot{u}) \in \Delta _G^{a_0} \times T_{u}V,
\end{equation} 
for arbitrary variations $ (\delta g(t), \delta p(t), \delta u(t), \delta a _0 (t))$ of the curve $(g(t), p(t), u(t), a _0 (t))$ in $T^{\ast}S$, with $ \delta g(t)$ and $ \delta u(t)$ vanishing at the endpoints and $ (\delta g, \delta{u}) \in \Delta _G^{a_0} \times T_{u}V$.

In our context, the Hamiltonian $\bar H$ is $S$-invariant and, in particular, does not depend on $u$ and is related to the given Hamiltonian function $H$ via the simple relation $\bar H(g,p,u, a _0 )=H(g,p,a _0)$, see Remark \ref{Link_RBS}. A similar proof as before, yields the following theorem.

\begin{theorem}\label{HPS_principle_second}  
Let $\bar H:T^*S  \rightarrow \mathbb{R}  $ be a $S$-invariant Hamiltonian and let $ \Delta ^{a _0 }_G \subset TG$ be a family of distribution verifying the $G$-invariance assumption \eqref{invariance_assumption}. Consider a curve $(g(t),p(t), u(t), a_{0}(t)) \in T^*S$, $t \in [t_{1},t_{2}]$, and define the curves $ \mu (t)= g (t) ^{-1} p (t) $ and $a (t) = g(t) ^{-1} a _0 $. Then the following are equivalent.
\begin{itemize}
\item[\bf{(i)}]
The Hamilton-d'Alembert principle in phase space
\[
\delta \int_{ t _1 }^{ t _2 } \left\{ \left\langle p, \dot g \right\rangle+ \left\langle a _0 , \dot u \right\rangle  - \bar{H}(g,p,u, a _0 )\right\} dt=0,\quad 
(\dot g, \dot{u}) \in \Delta ^{ a _0 }_G(g) \times T_{u}V 
\subset TS 
\]
holds, for variations $( \delta g (t), \delta u (t))$ of $(g(t), u(t))$ vanishing at the endpoints and such that $(\delta g , \delta{u}) \in \Delta ^{ a _0 }_G(g) \times T_{u}V$, and for arbitrary variations $ (\delta p (t), \delta a _0 (t))$ of $(p(t), a _0 (t))$.
\item[\bf{(ii)}]
The curve $(g(t),p(t) , u(t), a _0 (t))\in T^*S$, $t \in [t_{1},t_{2}]$ satisfies the implicit Hamiltonian system with constraint:
\[
\dot p+ \frac{\partial \bar H}{\partial g}\in \Delta _G (g)^\circ, \quad \dot g= \frac{\partial \bar H}{\partial p}\in \Delta _G(g), \quad \dot a _0 =0, \quad \dot u= \frac{\partial \bar H}{\partial a _0 }. 
\]
\item[\bf{(iii)}] The reduced Hamilton-d'Alembert principle with advected quantities
\begin{equation}\label{Lie_Poisson_VPSDP} 
\delta \int_{t_1}^{t_2} \left \{ \left\langle \mu , \xi  \right\rangle+ \left\langle a,v \right\rangle  -h(\mu , a) \right\}dt=0,\quad \xi \in \mathfrak{g}  ^ \Delta (a)
\end{equation} 
holds, for arbitrary variations $ \delta \mu (t) $, $ \delta a (t) $ and variations $ \delta \xi (t) $ and $ \delta v(t) $ of the form $\delta \xi  = \partial _t \zeta +[ \xi , \zeta ]$ and $\delta v= \partial _t w+ \xi w- \zeta v $, where $ \zeta \in \mathfrak{g}  ^ \Delta (a)$ vanishes at the endpoints and $ w \in V$ is arbitrary and vanishes at endpoints.
\item[\bf{(iv)}] The implicit Lie-Poisson-Suslov equations with parameters, together with the advection equation
\begin{equation}\label{implicit_LP_constrained}
\xi  =  \frac{\delta h}{\delta \mu }, \quad v= \frac{\delta h}{\delta{a}},\quad \xi \in\mathfrak{g}^{ \Delta}(a) , \quad \dot{\mu} - \operatorname{ad}_{\xi}^{\;\ast} \mu +v \diamond a\in \left( \mathfrak{g}^{ \Delta}(a)\right) ^\circ, \quad \dot a + a \xi =0
\end{equation} 
hold.
\end{itemize}
\end{theorem}

\begin{remark} {\rm It is important to compare the four points of the previous theorem with the four points of Theorem \ref{HPS_principle_first}, and also to note that in Theorem \ref{HPS_principle_second}, the advection equation is a consequence of the reduced Hamilton-d'Alembert principle with advected quantities. While Theorem \ref{HPS_principle_first} is associated to  the Dirac reduction approach of \S\ref{Hamiltonian_side} which starts with the Dirac structure $D_{ \Delta ^{ a _0 }_G}$ on $T^*G$, Theorem \ref{HPS_principle_second} is associated with the Dirac reduction approach \S\ref{sec_NH_SDP} which starts with the Dirac structure $D_{\Delta _{T^*S}}$ on $T^*S$.
} \end{remark}

\subsection{Dirac reduction for rolling ball type constraints on semidirect products}\label{sdp_case_Ham_bis}

All the results obtained in this Section apply to the case of the rolling ball type constraints on semidirect products described in \S\ref{subsec_the_sdp_case}.

Recall that in this case, one has to consider a parameter dependent constraint distribution $ \Delta _G^{a _0 } \subset TG$, where $a _0 \in V ^\ast $ and $G=K\,\circledS\, V$.

\paragraph{Dirac reduction.} The reduction of the induced Dirac structure $ D_{ \Delta ^{a _0 }_G} \subset T(T^*G) \oplus T^*(T^*G)$ was mentioned in Remark \ref{rolling_ball_remark}. Here we consider Dirac reduction that includes advection equation, along the approach described in the present section. Thus we need to consider the semidirect product
\[
S:=G \,\circledS\, V= (K \,\circledS\, V) \,\circledS\, V.
\]
Given a distribution $ \Delta _G^{a _0 }$ (e.g. of the form \eqref{case_I}--\eqref{case_III}), the associated distribution $ \Delta _{T^*S} \subset T(T^*S)$, is given by (see Definition \ref{Definition_delta_star})
\[
\Delta _{T^*S}(k,p,x, \pi ,u,a_0 )= \left( \left( T_{(k,p, x, \pi )} \pi_G \right) ^{-1}\left( \Delta _G^{a_0}(k,x)\right) \right)  \times T_{(u, a _0 )}T^*V.
\] 
By reduction, we get the reduced Dirac structure $D^{/ S}_{\Delta _{T^*S}}$ given by
\begin{equation}\label{condition_red_dirac_rollingtype_on_S} 
\begin{array}{c} 
\left( ( \mu ,b, a, \xi,X ,w, \rho , \sigma ,\tilde b),( \mu ,b,a, \beta , \gamma ,c,  \eta, Y,v ) \right) \in D^{/S}_{\Delta_{T^*S}}( \mu , b , a)\\
\Longleftrightarrow\\
\left( -\rho +\operatorname{ad} ^*_ \xi \mu - X \diamond b-w \diamond a- \beta , - \sigma - \xi b- \gamma \right) \in (\mathfrak{g}  ^ \Delta (a))^\circ\\
\tilde b+ \xi a+c =0, \;\; ( \xi , X)= ( \eta , Y), \;\; w=v,\;\; (\xi ,X) \in \mathfrak{g}  ^ \Delta (a),
\end{array} 
\end{equation}
where we used \eqref{condition_red_dirac_S}.
From \eqref{condition_red_dirac_rollingtype_on_S} it follows that the solution curves of the reduced Hamilton-Dirac system 
\[
\left(( \mu ,b, a, \xi, X , w, \dot \mu , \dot b, \dot a), \mathbf{d} ^{/S}h( \mu,b , a) \right) \in D^{/S}_{\Delta _{S \times \mathfrak{s}  ^\ast }}( \mu,b ,a)
\]
verify the equations
\begin{align*} 
&\left( \dot \mu - \operatorname{ad}^*_ \xi \mu +X \diamond b+w \diamond a, \dot b+ \xi b \right)  \in ( \mathfrak{g}  ^ \Delta (a))^\circ, \;\; \dot a + \xi a=0,\\
&(\xi, X) = \left( \frac{\delta h}{\delta \mu }, \frac{\delta h}{\delta b}\right)  , \;\; w= \frac{\delta h}{\delta a},\;\; (\xi,X) \in \mathfrak{g}  ^ \Delta (a).
\end{align*} 
We consistently recover the implicit Lie-Poisson-Suslov equations with advected quantities together with the advection equation, in the case of rolling ball type constraint as described in \S\ref{subsec_the_sdp_case}.

\paragraph{Phase space variational structures.} 
In the particular situation of \S\ref{subsec_the_sdp_case}, the Hamilton-d'Alembert principle in phase space reads
\[
\delta \int_{ t _1 }^{ t _2 } \left\{ \left\langle p, \dot k \right\rangle+ \left\langle \pi , \dot x \right\rangle + \left\langle a _0 , \dot u \right\rangle  -H(k,p,x,\pi , a _0 )\right\} dt=0,\quad (\dot k, \dot x) \in \Delta ^{ a _0 }_G(g,x)
\]
for variations $ \delta k (t) $, $ \delta x (t) $, $ \delta u (t)$ of $k(t)$, $ x (t) $, $u(t)$ vanishing at the endpoints and such that $(\delta k, \delta x)\in \Delta _G^{a_0}$ and $\delta{u} \in T_{u}V$, and for arbitrary variations $ \delta p (t) $, $ \delta \pi (t) $, $ \delta a _0 (t)$ of $p(t)$, $ \pi (t) $, $ a _0 (t)$.

At the reduced level, the Hamilton-d'Alembert principle may be reduced to
\begin{equation}\label{Lie_Poisson_VPSDP_SDP} 
\delta \int_{t_1}^{t_2} \left \{ \left\langle \mu , \xi  \right\rangle+ \left\langle b, X \right\rangle + \left\langle a,v \right\rangle  -h(\mu ,b, a) \right\}dt=0,\quad (\xi,X) \in \mathfrak{g}  ^ \Delta (a),
\end{equation} 
for arbitrary variations $ \delta \mu (t) $, $ \delta b(t) $, $ \delta a (t) $ and variations $ \delta \xi (t) $, $ \delta X(t) $ and $ \delta v(t) $ of the form $\delta \xi  = \partial _t \zeta +[ \xi , \zeta ]$, $ \delta X= \partial _t Z+ \xi Z- \zeta X$, and $\delta v= \partial _t w+ \xi w- \zeta v $, where $ (\zeta ,Z)\in \mathfrak{g}  ^ \Delta (a)$ vanishes at the endpoints and $ w \in V$ is arbitrary and vanishes at endpoints.

These variational structures should be compared with the ones in Remark \ref{sdp_case_Ham}. In Remark \ref{sdp_case_Ham}, the advection equation for $a$ is assumed a priori, whereas here, it is a consequence of the Hamilton-d'Alembert  principle.

\section{Examples}\label{Examples}

\subsection{The heavy top}
As a simple and illustrative example for the unconstrained case, let us consider the heavy top.
Let $ \mathcal{B} \subset  \mathbb{R}^3$ be the reference configuration of the top and $\rho_{0}(\mathbf{X} )$ be its mass density, where $\mathbf{X}\in \mathcal{B}$. The evolution of the heavy top is described by a curve $\Lambda  (t) \in SO(3)$ giving the orientation of the body in space. The Lagrangian is defined on the tangent bundle $TSO(3)$ and is given by kinetic energy minus potential energy,
\[
L( \Lambda , V)= \frac{1}{2} \int_ \mathcal{B} \rho_{ 0}(\mathbf{X} )\|V \mathbf{X} \|^2d^3\mathbf{X}-mgl \Lambda ^{-1}\mathbf{e}_{3} \cdot \boldsymbol{\chi},
\]
where $m$ is the total mass of the top, $g$ is the magnitude of the gravitational acceleration, $\boldsymbol{\chi} $ is the unit vector of the oriented line segment pointing from the fixed point about which the top rotates (the origin of a spatial coordinate system) to the center of mass of the body, and $l$ is its length. Whereas the kinetic energy is $SO(3)$ left-invariant, the potential energy is invariant only under the rotations $S^{1}$ about the $\mathbf{e}_{3}$--axis.

In order to apply Theorem \ref{theorem_HP_semidirect}, we shall view $ \mathbf{e} _3 $ as a parameter, denoted $ \boldsymbol{\Gamma} _0 = \mathbf{e} _3 $ and we shall write the Lagrangian as $L_{ \boldsymbol{\Gamma} _0 }:TSO(3) \rightarrow \mathbb{R}  $ in order to emphasize this dependence, see \cite{HMR1998a}, to which we refer for the Euler-Poincar\'e formulation of the heavy top.

\paragraph{Lagrangian side.} Using this notation, the Hamilton-Pontryagin principle \eqref{HP-ActInt} for the heavy top is given by
\[
\delta \int_{t_{1}}^{t_{2}} \left\{ L_{\boldsymbol{\Gamma}_{0}}(\Lambda ,V) + \left< p, \dot{\Lambda }-V\right> \right\}dt=0,
\]
for all variations of $(\Lambda ,V, p)$ in the Pontryagin bundle $TSO(3) \oplus T^{\ast}SO(3)$ such that $ \delta \Lambda $ vanishes at the endpoints.

Writing  $\boldsymbol{\Gamma}=\Lambda ^{-1} \boldsymbol{\Gamma}_{0}$, we obtain the reduced Lagrangian $\ell: \mathfrak{so}(3) \times \mathbb{R}^{3} \to \mathbb{R}$ given by
$$
\ell(\boldsymbol{\Psi}, \boldsymbol{\Gamma})=L(\Lambda ^{-1}V, \Lambda ^{-1} \boldsymbol{\Gamma}_{0}) =
\frac{1}{2}\left<\boldsymbol{\Psi}, \mathbb{I} \boldsymbol{\Psi} \right>-mgl \boldsymbol{\Gamma} \cdot \boldsymbol{\chi},
$$
where $ \mathbb{I}  $ is the inertia tensor.
Reduction of the Hamilton-Pontryagin principle yields the principle
\[
\delta \int_{t_{1}}^{t_{2}} \left\{ \ell(\boldsymbol{\Psi}, \boldsymbol{\Gamma}) + \left< \boldsymbol{\Pi}, \boldsymbol{\Omega} - \boldsymbol{\Psi}\right> \right\}dt=0,
\]
for arbitrary variations $\delta\boldsymbol{\Psi}, \delta \boldsymbol{\Pi}, \delta\boldsymbol{\Gamma}$, 
and variations of the form 
\begin{equation}\label{var_heavy_top} 
\begin{split}
\delta\boldsymbol{\Omega}=\dot{\boldsymbol{\Sigma}}+\boldsymbol{\Omega} \times \boldsymbol{\Sigma}\quad 
\textrm{and}\quad
\delta\boldsymbol{\Gamma}=\boldsymbol{\Gamma} \times \boldsymbol{\Sigma},
\end{split}
\end{equation}
where $\boldsymbol{\Omega}$ is the curve in $\mathbb{R}^{3}$ given such that $\hat{\boldsymbol{\Omega}}(t)=\Lambda ^{-1}(t)\dot{\Lambda }(t)$, and where $\boldsymbol{\Sigma}(t)$ is the curve in $\mathbb{R}^{3}$ vanishing at endpoints, which is defined such that $\hat{\boldsymbol{\Sigma}}(t)=\Lambda ^{-1}(t)\delta \Lambda (t)$ with $\delta \Lambda (t_{1})=\delta \Lambda (t_{2})=0$.
\medskip

From this principle, we obtain the implicit equations of motion for the heavy top
\begin{equation*}
\begin{split}
\boldsymbol{\Pi}=\frac{\partial \ell}{\partial \boldsymbol{\Psi}}=\mathbb{I}\boldsymbol{\Psi}, \qquad
\dot{\boldsymbol{\Pi}}=\boldsymbol{\Pi} \times \boldsymbol{\Omega} -mg \boldsymbol{\chi} \times \boldsymbol{\Gamma}, \qquad
\boldsymbol{\Omega}=\boldsymbol{\Psi}.
\end{split}
\end{equation*}
The definition $\boldsymbol{\Gamma} = \Lambda ^{-1} \boldsymbol{\Gamma} _0 $ yields the advection equation $\dot{ \boldsymbol{\Gamma}}=\boldsymbol{\Gamma} \times \boldsymbol{\Omega}$.

\medskip

As we mentioned in Remark \ref{Hybrid_HP_advected}, one can also obtain the equations by using the variational principle
\[
\delta \int_{ t _1 }^{ t _2 }\left\{ \ell( \boldsymbol{\Omega} , \boldsymbol{\Gamma} )+  \left\langle\boldsymbol{\Pi}, \Lambda ^{-1} \dot{ \Lambda }- \boldsymbol{\Omega} \right\rangle + \left\langle \mathbf{V} , \Lambda ^{-1} \boldsymbol{\Gamma} _0 - \boldsymbol{\Gamma} \right\rangle \right\} dt=0
\]
for arbitrary variations $ \delta \boldsymbol{\Omega} $, $ \delta \boldsymbol{\Pi} $, $ \delta \mathbf{V} $, $ \delta \boldsymbol{\Gamma} $ and variations $ \delta \Lambda $ vanishing at the endpoints.

\paragraph{Hamiltonian side.} By Legendre transformation, the Hamiltonian $H_{\boldsymbol{\Gamma} _0 }=H_{\boldsymbol{\Gamma} _0 }( \Lambda , p)$ is given by kinetic plus potential energy on $T^* SO(3)$ and we obtain the reduced expression
\[
h( \boldsymbol{\Pi} , \boldsymbol{\Gamma} )= 
\frac{1}{2}\left<\mathbb{I}^{-1} \boldsymbol{\Pi} , \boldsymbol{\Pi}  \right>+mgl \boldsymbol{\Gamma} \cdot \chi,
\]
where $\boldsymbol{\Gamma} = \Lambda ^{-1} \boldsymbol{\Gamma} _0$ and $\boldsymbol{\Pi} = \Lambda ^{-1} p$. By reducing the Hamilton principle in phase space for $H_{ \boldsymbol{\Gamma} _0 }$, it follows from \eqref{Lie_Poisson_VPSDP_first} that one obtains the Lie-Poisson variational principle as
\[
\delta \int_{ t _1 }^{ t _2 } \left\{\left\langle \boldsymbol{\Pi} , \boldsymbol{\Omega} \right\rangle -h( \boldsymbol{\Pi} , \boldsymbol{\Gamma} ) \right\} dt=0
\]
for arbitrary variations $ \delta \boldsymbol{\Pi} $ and variations $ \delta \boldsymbol{\Omega}$, 
$ \delta \boldsymbol{\Gamma} $ of the form \eqref{var_heavy_top}.
It yields the implicit Lie-Poisson equation for the heavy top. As before, the equation $\dot{ \boldsymbol{\Gamma}}=\boldsymbol{\Gamma} \times \boldsymbol{\Omega}$ is a consequence of the definition $ \boldsymbol{\Gamma} = \Lambda ^{-1} \boldsymbol{\Gamma} _0 $. In order to obtain the advection directly from the variational principle, one has to reduce the Hamilton principle in phase space $T^*(SO(3) \,\circledS\, \mathbb{R}  ^3 )$, as explained in \S\ref{PSVP_TstarS}. In this case, the Lie-Poisson variational principle is given by
\[
\delta \int_{ t _1 }^{ t _2 } \left\{\left\langle \boldsymbol{\Pi} , \boldsymbol{\Omega} \right\rangle + \left\langle \boldsymbol{\Gamma} , \mathbf{V} \right\rangle -h( \boldsymbol{\Pi} , \boldsymbol{\Gamma} ) \right\} dt=0,
\]
for arbitrary variations $ \delta \boldsymbol{\Pi} $, $ \delta \boldsymbol{\Gamma} $ and variations $ \delta \boldsymbol{\Omega} $ and $ \delta \mathbf{V} $ of the form
\begin{equation*}\label{var_heavy_top1} 
\begin{split}
\delta\boldsymbol{\Omega}=\dot{\boldsymbol{\Sigma}}+\boldsymbol{\Omega} \times \boldsymbol{\Sigma}\quad 
\textrm{and}\quad
\delta\mathbf{V} =\partial _t \mathbf{W} + \boldsymbol{\Omega} \times \mathbf{W} - \boldsymbol{\Sigma} \times \mathbf{V}.
\end{split}
\end{equation*}
It yields the heavy top equations in the implicit form
\[
\boldsymbol{\Omega} = \mathbb{I}  ^{-1} \boldsymbol{\Pi} , \quad \mathbf{V} = mgl \chi , \quad \dot{\boldsymbol{\Pi} }= \boldsymbol{\Pi} \times \boldsymbol{\Omega} + \boldsymbol{\Gamma} \times \mathbf{V} , \quad \dot { \boldsymbol{\Gamma} }+ \boldsymbol{\Omega} \times \boldsymbol{\Gamma} =0.
\]

\paragraph{Lie-Dirac reduction with advection.} Starting with the canonical Dirac structure $D_{can} \subset T(T^*SO(3)) \oplus T^*(T^*SO(3))$, as in \eqref{condition_red_dirac}, one carries out the Lie-Dirac reduction with respect to the isotropy subgroup $SO(3)_{ \mathbf{e} _3 }= S ^1 $ and then obtains the reduced Dirac structure $D^{/S ^1 }_{can}$ given at $(\boldsymbol{\Pi}, \boldsymbol{\Gamma} ) \in \mathfrak{so}(3) ^\ast \times S ^2 $, where $ S ^2 =\operatorname{Orb}( \mathbf{e} _3 )$, as shown below:
\begin{equation}\label{condition_red_dirac_heavytop} 
\begin{array}{c} 
\left( ( \boldsymbol{\Pi}  , \boldsymbol{\Gamma} , \boldsymbol{\Omega}  , \boldsymbol{\rho}  ),( \boldsymbol{\Pi}  , \boldsymbol{\Gamma} , \boldsymbol{\beta }  , \boldsymbol{\Psi }  ) \right) \in D^{/S ^1 }_{can}\\
\Longleftrightarrow\\
\boldsymbol{\beta} + \boldsymbol{\rho}  + \boldsymbol{\Omega} \times \boldsymbol{\Pi} =\mathbf{0} \quad \text{and} \quad \boldsymbol{\Psi}  = \boldsymbol{\Omega}.
\end{array} 
\end{equation}
One readily checks that the reduced Lagrange-Dirac system
\[
\left( ( \boldsymbol{\Pi} , \boldsymbol{\Gamma} , \boldsymbol{\Omega} , \dot{ \boldsymbol{\Pi} }), \mathbf{d} _D^{/S ^1 }\ell( \boldsymbol{\Psi} ,\boldsymbol{\Gamma})\right)  \in D^{/S ^1 }_{can}( \boldsymbol{\Pi} , \boldsymbol{\Gamma} )
\]
yields the heavy top equations in implicit form. Similar observations hold on the Hamiltonian side, following \S\ref{subsec_preliminary_comments}. In these formulations, one always has to assume the relation $ \boldsymbol{\Gamma} = \Lambda ^{-1} \mathbf{e} _3 $ or, equivalently, $\dot {\boldsymbol{\Gamma}} + \boldsymbol{\Omega} \times \boldsymbol{\Gamma} =0$.

The advection equation can be included in the Dirac formulation, as showed in \S\ref{sec_NH_SDP}, by starting with the canonical Dirac structure $D_{can}$ on $T^*S$, where $S:= SO(3) \,\circledS\, \mathbb{R}  ^3 $. In this case, the Lie-Dirac reduction by $S$ yields the reduced Dirac structure $D^{/S}_{can}\subset \mathfrak{s} ^\ast :=(\mathfrak{so}(s) \,\circledS\, \mathbb{R}  ^3 ) ^\ast $ given at $( \boldsymbol{\Pi} , \boldsymbol{\Gamma} )$ by
\begin{equation}\label{condition_red_dirac_S_heavytop} 
\begin{array}{c} 
\left( ( \boldsymbol{\Pi}  , \boldsymbol{\Gamma} , \boldsymbol{\Omega}  , \mathbf{w} , \boldsymbol{\rho}   , \mathbf{b} ),( \boldsymbol{\Pi}  , \boldsymbol{\Gamma} , \boldsymbol{\beta}   , \mathbf{c} , \boldsymbol{\Psi}   , \mathbf{v} )\right) \in D^{/S}_{can }(\boldsymbol{\Pi} ,\boldsymbol{\Gamma} )\\
\Longleftrightarrow\\
-\boldsymbol{\rho} + \boldsymbol{\Pi} \times \boldsymbol{\Omega} - \mathbf{w}  \times  \boldsymbol{\Gamma}  - \boldsymbol{\beta}  = \mathbf{0},\;\; \mathbf{b}  +  \boldsymbol{\Omega}  \times \boldsymbol{\Gamma}   + \mathbf{c}  =\mathbf{0} , \;\;  \boldsymbol{\Omega}  = \boldsymbol{\Psi}  ,\;\;  \mathbf{w} =\mathbf{v} .
\end{array} 
\end{equation} 
The reduced Hamilton-Dirac system
\[
\left(( \boldsymbol{\Pi}  , \boldsymbol{\Gamma} , \boldsymbol{\Omega}  , \mathbf{w} , \dot{ \boldsymbol{\Pi} }, \dot{ \boldsymbol{\Gamma} }), \mathbf{d} ^{/S}h(\boldsymbol{\Pi}  , \boldsymbol{\Gamma} ) \right) \in D^{/S}_{can}(\boldsymbol{\Pi} ,\boldsymbol{\Gamma} ).
\]
yields the heavy top equations in implicit form
\[
\dot{ \boldsymbol{\Pi}  }+  \boldsymbol{\Omega} \times \boldsymbol{\Pi}  + \mathbf{w}  \times  \boldsymbol{\Gamma}   =\mathbf{0} ,\;\; \dot{ \boldsymbol{\Gamma} }  +  \boldsymbol{\Omega}  \times\boldsymbol{\Gamma} =\mathbf{0} , \;\;  \boldsymbol{\Omega}  = \frac{\delta h}{\delta \boldsymbol{\Pi} }  ,\;\;  \mathbf{w} =\frac{\delta h}{\delta \boldsymbol{\Gamma} },
\]
as desired.

\subsection{Incompressible ideal fluids}
It is widely known that Euler equations for ideal fluids can be obtained from variational principles (see, for instance, \cite{Arnol'd1966} and \cite{Bretherton1970}). An extended variational principle for ideal continuum was developed by  \cite{CendraMarsden1987}, in which one can incorporate Clebsch potentials into Hamilton's variational principle. We follow the mathematical notations used in \cite{HMR1998a}. 

Let $\mathfrak{D}$ be a bounded domain in $\mathbb{R}^{3}$ with smooth boundary $\partial \mathfrak{D}$ and let $G\!=\!\mathrm{Diff}(\mathfrak{D})$ be the group of all smooth diffeomorphisms of $\mathfrak{D}$. Recall that the Lie algebra of $ \operatorname{Diff}( \mathfrak{D}  )$ is given by the space $ \mathfrak{X}  ( \mathfrak{D} )$ of all smooth vector fields on $ \mathfrak{D} $ parallel to the boundary, endowed with the Lie bracket $[ \mathbf{u} , \mathbf{v} ]= \mathbf{v} \cdot \nabla \mathbf{u}- \mathbf{u} \cdot \nabla \mathbf{v} $.

The motion of a fluid in the domain $ \mathfrak{D} $ can be described by a curve $\eta_{t}$ in $\operatorname{Diff}(\mathfrak{D})$ giving an evolutional sequence of diffeomorphisms from the reference configuration to the current configuration in $\mathfrak{D}$, that is, $x(X,t):=\eta_{t}(X) \in \mathfrak{D}$, where $x$ is the current Eulerian spatial location of the particle with label $X$. %\begin{minipage}{1.0\textwidth}
%\begin{figure}
%\begin{center}
%\hspace{1cm}
%\includegraphics[scale=.35, clip]{IdealFluid-new1.eps}
%\vspace{-0.8cm}
%\caption{¬ÌÌ^®}
%\label{fig-fluid}
%\end{center}
%\end{figure}
%\end{minipage}

The Lagrangian or material velocity of the fluid is defined by taking the time derivative of the Lagrangian trajectory keeping the particle label $X$ fixed:
%\[
%\begin{split}
$\mathbf{V}(X,t):=\frac{\partial \eta_{t}(X)}{\partial t}
:=\frac{\partial }{\partial t}\big|_{X} \eta_{t}(X)=\frac{\partial x(X,t)}{\partial t}$,
%\end{split}
%\] 
while the Eulerian or spatial velocity of the system along the path $x(X,t):=\eta_{t}(X)$ is defined by taking the time derivative of the path keeping the Eulerian point $x$ fixed as
$\mathbf{v}(x,t):=\mathbf{V}(X,t):=\frac{\partial }{\partial t}\big|_{x} \eta_{t}(X)=\mathbf{V}(\eta_{t}^{-1}(x),t)$.
We thus have the following relationship between the Lagrangian velocity $(\eta,\dot{\eta})$ and the Eulerian velocity $\mathbf{v}$:
\[
\mathbf{v}_t =\dot{\eta}_t  \circ \eta_t ^{-1}, \quad \text{i.e.}, \quad  \mathbf{v}_{t}= \mathbf{V}_{t} \circ \eta_{t}^{-1},
\]
where $\mathbf{V}_{t}(X):=\mathbf{V}(X,t)$ and $\mathbf{v}_{t}(x):=\mathbf{v}(x,t)$. Note that the Eulerian velocity $ \mathbf{v}_t$ belongs to the Lie algebra $ \mathfrak{X}  ( \mathfrak{D} )$ of $\operatorname{Diff}( \mathfrak{D} )$. 

In the incompressible case, the fluid motion is described by a curve in the subgroup $ \operatorname{Diff}_{vol}( \mathfrak{D} )=\{ \eta \in \operatorname{Diff}( \mathfrak{D} )\mid J \eta =1\} $ of volume preserving diffeomorphisms of $ \mathfrak{D} $, where $ J \eta $ denotes the Jacobian of the diffeomorphism $ \eta $. Its Lie algebra is given by the space $ \mathfrak{X}  _{vol}( \mathfrak{D} )$ of divergence free smooth vector fields parallel to the boundary. Note that for simplicity we have considered here a bounded domain  $\mathfrak{D} \subset \mathbb{R}  ^3 $ with smooth boundary, endowed with the Riemannian metric given by the standard inner product on $ \mathbb{R}  ^3 $. The approach extends easily to arbitrary compact Riemannian manifolds with smooth boundary. 

%%%%%%%%%%%%%%%%%%%%%%%%%%
\paragraph{Implicit Euler-Poincar\'e equations for incompressible ideal fluids.}
The (reduced) Lagrangian of the incompressible fluid is $\ell: \mathfrak{X}  _{vol}( \mathfrak{D} ) \rightarrow \mathbb{R}  $,
\[
\ell( \mathbf{v} )= \int_ \mathfrak{D} \frac{1}{2} | \mathbf{v} | ^2 d ^3 x,
\]
where $| \mathbf{v} |$ denotes the norm of the vector field $\mathbf{v} $ relative to the inner product on $ \mathbb{R}  ^3 $. 
In this simple example, there are no advected quantities, so the 
reduced Hamilton-Pontryagin principle in Eulerian coordinates reads
\[
\delta \int_{t_{1}}^{t_{2}} \int_{\mathfrak{D}} \left[\frac{1}{2} | \mathbf{u} | ^2+\left<\Pi , \mathbf{v}-\mathbf{u} \right> \right] d^{3}x \,dt=0,
\]
for arbitrary variations $ \delta \Pi $ and $ \delta \mathbf{u}$ and variations $\delta{\bf{v}}=\dot{\bf{w}}-[\mathbf{v},\mathbf{w}]$, where $ \mathbf{w}= \delta{\eta} \circ \eta^{-1}$ is an arbitrary curve in $ \mathfrak{X}  _{vol}( \mathfrak{D} )$ vanishing at the endpoints. Here we have made the identification $\mathfrak{X}  _{vol}( \mathfrak{D} ) ^\ast = \mathfrak{X}  _{vol}( \mathfrak{D} )$ (regular dual), by using the duality pairing
\[
\left\langle \Pi , \mathbf{v} \right\rangle =\int_ \mathcal{D} \Pi \cdot \mathbf{v}\, d ^3 x, \quad \Pi  , \mathbf{v} \in \mathfrak{X}  _{vol}(\mathfrak{D} ).
\]
One obtains the implicit Euler-Poincar\'e equations
\[
\dot { \Pi }+ \operatorname{ad}^*_ \mathbf{v}\Pi  =0, \quad  \frac{\delta\ell}{\delta \mathbf{u} }= \Pi , \quad \mathbf{u} = \mathbf{v}.
\]
Since $\Pi = \delta \ell/ \delta \mathbf{u} = \mathbf{u} $ and $\operatorname{ad}^*_ \mathbf{v} \Pi = \mathbb{P} ( \mathbf{v} \cdot \nabla \Pi + \nabla \mathbf{v} ^\mathsf{T} \Pi )$, where $ \mathbb{P}$ is the Hodge projector onto divergence free vector fields parallel to the boundary, one recovers the Euler equations $ \partial _t \mathbf{v} + \mathbf{v} \cdot \nabla \mathbf{v} = - \nabla p$.

\medskip

By applying the alternative variational principle of Remark \ref{Hybrid_HP}, we get
\[
\delta \int_{t_{1}}^{t_{2}} \int_{\mathfrak{D}} \left[\frac{1}{2} | \mathbf{v} | ^2+\left<\Pi , \dot \eta \circ \eta ^{-1} -\mathbf{v} \right> \right] d^{3}x \,dt=0,
\]
for arbitrary variations $ \delta \mathbf{v} $, $ \delta \Pi  $, and variations $ \delta \eta $ vanishing at the endpoints. The stationarity conditions are
\[
\mathbf{v} = \dot \eta \circ \eta ^{-1} , \quad \frac{\delta\ell}{\delta \mathbf{u} }= \Pi , \quad \dot { \Pi }+ \operatorname{ad}^*_ \mathbf{v}\Pi  =0.
\]

On the Hamiltonian side, the Lie-Poisson variational principle reads
\[
\delta \int_{t_{1}}^{t_{2}} \int_{\mathfrak{D}} \left[\left<\Pi ,\mathbf{v} \right>-\frac{1}{2} |\Pi| ^2\right] d^{3}x \,dt=0,
\]
for arbitrary variations $ \delta \Pi $ and variations $\delta{\bf{v}}=\dot{\bf{w}}-[\mathbf{v},\mathbf{w}]$, where $ \mathbf{w} $ is an arbitrary curve in $ \mathfrak{X}  _{vol}( \mathfrak{D} )$ vanishing at the endpoints. One can also easily implement the variational principle given in Remark \ref{Triv_HamPrinPhSp} to the present case.

\paragraph{Lie-Dirac reduction for the ideal fluid.} It can be implemented by choosing the canonical Dirac structure $D_{can} \subset T(T^* \operatorname{Diff}_{vol} (\mathfrak{D} ))\oplus T^*(T^* \operatorname{Diff}_{vol}(\mathfrak{D} ))$, where we identify the cotangent space at $ \eta $ with the tangent space, i.e., $T^*_ \eta  \operatorname{Diff}_{vol}( \mathfrak{D} )= T_ \eta  \operatorname{Diff}_{vol}( \mathfrak{D} )= \{\mathbf{V}  _ \eta : \mathfrak{D} \rightarrow T\mathfrak{D} \mid \mathbf{V} _\eta \circ \eta ^{-1} \in \mathfrak{X}  _{vol}( \mathfrak{D} ) \}$. According to the Lie-Dirac reduction process, the reduced Dirac structure on the reduced Pontryagin bundle
\[
\mathfrak{X}_{vol}( \mathfrak{D} ) ^\ast \times (W \oplus W ^\ast ) 
\]
is given as follows:
\[
D_{can}^{/\operatorname{Diff}_{vol} ( \mathfrak{D} )}(\Pi )= \{ ((\mathbf{v}  ,\boldsymbol{\rho} ),(\boldsymbol{\beta} ,\mathbf{u}  )) \in W \oplus W^{\ast} \mid 
\boldsymbol{\beta}   + \boldsymbol{\rho} - \mathrm{ad}_{\mathbf{v} }^{\ast}\Pi =0 \},
\]
where $W= \mathfrak{X}  _{vol}( \mathfrak{D} ) \times \mathfrak{X}  _{vol}( \mathfrak{D} ) ^\ast$.

The reduced Lagrange-Dirac and Hamilton-Dirac systems associated to $D_{can}^{/\operatorname{Diff}_{vol}( \mathfrak{D} ) }$ read
\[
\left( (\Pi , \mathbf{v} ,\dot{ \Pi }),\mathbf{d} _D^{/ \operatorname{Diff}_{vol}( \mathfrak{D} )}\ell( \mathbf{u} )\right)  \in D_{can}^{/\operatorname{Diff}_{vol}( \mathfrak{D} )}( \Pi ) \;\text{and}\; \left( (\Pi , \mathbf{v} ,\dot{ \Pi }),\mathbf{d}^{/ \operatorname{Diff}_{vol}}h( \Pi  )\right)   \in D_{can}^{/\operatorname{Diff}_{vol} }( \Pi ), 
\]
each of which yields the ideal fluid equations in implicit form.

\subsection{Compressible MHD}

As an illustrative example of the Hamilton-Pontryagin principle with advected parameter on diffeomorphism groups, we consider the case of compressible magnetohydrodynamics (MHD), whose Euler-Poincar\'e formulation has been given in \cite{HMR1998a}. In this case, the configuration Lie group is $G= \operatorname{Diff}( \mathfrak{D} )$ and the space of advected quantities is $V^*=  \mathrm{Den}(\mathfrak{D}) \times \Omega ^2_{cl}( \mathfrak{D})\ni ( \rho \otimes d ^3 x,B)$, where $ \rho \otimes d ^3 x$ is the mass density and the closed two-form $B$, where $ \mathbf{d} B=0$, is the magnetic field. Since we assumed that $ \mathfrak{D} \subset \mathbb{R}  ^3 $, we can identify the closed two-form $B=\sum_{i<j}B_{ij} dx ^i \wedge dx ^j $ with the divergence free vector field $ \mathbf{B} $ given by $ B_{ij}= \epsilon _{ijk} \mathbf{B} _k $.

The diffeomorphism group acts on these advected quantities by pull-back as
\[
\rho \otimes d ^3 x \mapsto \eta ^\ast (\rho \otimes d ^3 x) \quad\text{and}\quad B \mapsto \eta ^\ast B,
\]
and we also denote by $ \mathbf{B} \mapsto \mathbf{B}\cdot \eta  $ the representation induced on the magnetic vector field $ \mathbf{B} $.

\paragraph{Lagrangian side.} Given the mass density $ \rho _0 \otimes d ^3 x$ and the magnetic field $ \mathbf{B} _0 $ in the reference configuration, the (reduced) Lagrangian is obtained by reducing the Lagrangian $L_{ \rho _0 , \mathbf{B} _0 }:T \operatorname{Diff}( \mathfrak{D} ) \rightarrow \mathbb{R}  $ in material representation as 
\[
\ell( \mathbf{v} , \rho , \mathbf{B} )= \int_ \mathfrak{D} \left( \frac{1}{2} \rho | \mathbf{v} | ^2 - \rho e( \rho ) - \frac{1}{2} | \mathbf{B} | ^2 \right) d ^3 x,
\]
where $e( \rho )$ is the internal energy of the fluid and
\begin{equation}\label{representation_rho_B}
\rho \otimes d ^3 x= \eta _\ast ( \rho _0 \otimes d ^3 x)\quad\text{and}\quad  \mathbf{B} = \mathbf{B} _0 \cdot \eta ^{-1}.
\end{equation} 
The reduced Hamilton-Pontryagin principle reads
\[
\delta \int_{t_{1}}^{t_{2}} \int_{\mathfrak{D}} \left[ \frac{1}{2} \rho | \mathbf{u} | ^2 - \rho e( \rho ) - \frac{1}{2} | \mathbf{B} | ^2 +\left<\Pi , \mathbf{v}-\mathbf{u} \right> \right] d^{3}x \,dt=0,
\]
for variations of the form
\begin{equation}\label{variations_MHD} 
\begin{aligned} 
\delta{\bf{v}}=\dot{\bf{w}}-[\mathbf{v},\mathbf{w}],\qquad \quad 
 \delta{(\rho  d^{3}x)}&=-\pounds_{\mathbf{w}}(\rho d^{3}x)=-\nabla \cdot (\rho\mathbf{w})\,d^{3}x\\
\delta ( \mathbf{B} \cdot d \mathbf{S} )&= - \pounds _ {\mathbf{w} }( \mathbf{B} \cdot d \mathbf{S} )=  \operatorname{curl}( \mathbf{w} \times \mathbf{B} ) \cdot d \mathbf{S},
\end{aligned}
\end{equation} 
where $ \mathbf{w} \in \mathfrak{X}  ( \mathfrak{D} )$ is an arbitrary curve vanishing at the endpoints. Application of this principle yields the equations
\begin{equation*} 
\begin{array}{ccll}
\vspace{1mm}&\;(i)\;\;  & \displaystyle\Pi=\frac{\delta {\ell}}{\delta \mathbf{u}}=\rho \mathbf{u} \quad &\textrm{:  momentum density}, \\
&\;(ii)\;\; &\mathbf{v}=\mathbf{u} \quad \quad &\textrm{:  second-order vector field}, \\ 
&\;(iii)\;\; &\displaystyle\frac{\partial{\Pi}}{\partial t}+\mathrm{ad}^{\ast}_{\mathbf{v}} \,\Pi =\rho \nabla \frac{\delta \ell}{\delta \rho }+ \mathbf{B} \times \operatorname{curl} \frac{\delta \ell}{\delta \mathbf{B} } \quad &\textrm{:  equations of motion}.
\end{array}
\end{equation*} 
From the definition \eqref{representation_rho_B}, we get the advection equations
\[
\partial _t \rho + \operatorname{div}( \rho \mathbf{v}   )=0 \quad\text{and}\quad \partial _t \mathbf{B} +\operatorname{curl}( \mathbf{B} \times \mathbf{v} )=0.
\] 
Using the expression $ \frac{\delta \ell}{\delta \rho }  = \frac{1}{2} | \mathbf{u} | ^2 - e( \rho )- \rho \frac{\partial e}{\partial \rho }$ and $ \frac{\delta \ell}{\delta \mathbf{B} } = \mathbf{B} $, together with the advection equation for the mass density, one gets from the conditions $(i)-(iii)$ the MHD equations
\[
\partial _t \mathbf{v} + \mathbf{v} \cdot \nabla \mathbf{v} = - \frac{1}{\rho }\left( \nabla p + \mathbf{B} \times \operatorname{curl}\mathbf{B} \right), \quad p=- \frac{\partial e}{\partial \rho ^{-1} } . 
\]

The alternative variational principle of Remark \ref{Hybrid_HP_advected} yields
\begin{align*} 
&\delta \int_{ t _1 }^{ t _2 }\int_{\mathfrak{D}}  \left[  \frac{1}{2} \rho | \mathbf{v} | ^2 - \rho e( \rho ) - \frac{1}{2} | \mathbf{B} | ^2 + \left\langle\Pi  , \dot \eta \circ \eta ^{-1}  - \mathbf{v}  \right\rangle \right .\\
&\qquad \qquad  \qquad \qquad \left .\phantom{ \frac{1}{2} }+ v\left(\eta _\ast ( \rho _0 \otimes d ^3 x) - \rho \otimes d ^3 x \right) + \left\langle \mathbf{V} ,  \mathbf{B}_0 \cdot \eta ^{-1} - \mathbf{B}  \right\rangle \right]  d ^3 x\,dt=0,
\end{align*} 
for arbitrary variations $ \delta \mathbf{v}  $, $ \delta\Pi  $, $ \delta v$, $ \delta \mathbf{V} $, $ \delta \rho $, $ \delta \mathbf{B} $ and variations $\delta \eta$ vanishing at the endpoints. It yields the stationarity conditions
\[
\mathbf{v}  = \dot \eta \circ \eta  ^{-1}, \quad \eta _\ast ( \rho _0 \otimes d ^3 x) = \rho \otimes d ^3 x, \quad \mathbf{B} _0 \cdot \eta ^{-1} = \mathbf{B}  , \quad \Pi  = \frac{\delta \ell}{\delta \mathbf{v} }= \rho \mathbf{v}
\]
\[
v= \frac{\delta \ell}{\delta \rho }=\frac{1}{2} | \mathbf{v} | ^2 - e( \rho )- \rho \frac{\partial e}{\partial \rho } , \quad 
\mathbf{V} = \frac{\delta \ell}{\delta \mathbf{B}  }= - \mathbf{B} , \quad 
\dot{\Pi } =  \mbox{ad}_{\mathbf{v} }^{*} \Pi  + \rho \nabla v+ \mathbf{B} \times \operatorname{curl} \mathbf{V}  . 
\]

\paragraph{Hamiltonian side.} Consider the Hamiltonian $H_{ \rho _0 , \mathbf{B} _0 }:T^* \operatorname{Diff}( \mathfrak{D} )\rightarrow \mathbb{R}  $ obtained from $L_{ \rho _0 ,\mathbf{B} _0 }$ by Legendre transform. The reduced Hamiltonian is
\[
h( \Pi , \rho ,\mathbf{B} )= \int_ \mathfrak{D} \left(  \frac{1}{2\rho }| \Pi | ^2 + \rho e( \rho ) + \frac{1}{2} | \mathbf{B} |^2  \right)  d ^3 x.
\]

By reducing Hamilton's phase space principle for $H_{ \rho _0 , \mathbf{B} _0 }$, we have the Lie-Poisson variational principle (see \eqref{Lie_Poisson_VPSDP_first}) as
\[
\delta \int_{t_1}^{t_2} \int_ \mathfrak{D} \left[  \left\langle \Pi  , \mathbf{v}   \right\rangle -\left( \frac{1}{2\rho }| \Pi | ^2 + \rho e( \rho ) + \frac{1}{2} | \mathbf{B} |^2 \right)  \right] d ^3 x\, dt=0,
\]
for arbitrary variations $ \delta \Pi $ and variations $ \delta \mathbf{v} $ , $ \delta \mathbf{B} $ and $ \delta \rho $ of the form \eqref{variations_MHD}. 

\medskip

As explained in \S\ref{PSVP_TstarS}, one can obtain the advection equations naturally by reducing the phase space principle for the Hamiltonian defined on the phase space of the semidirect product. We obtain the reduced Lie-Poisson variational principle
\[
\delta \int_{t_1}^{t_2} \int_ \mathfrak{D} \left[\left\langle \Pi  , \mathbf{v}   \right\rangle + \left\langle \rho , v \right\rangle + \left\langle \mathbf{B} , \mathbf{V} \right\rangle -\left( \frac{1}{2\rho }| \Pi | ^2 + \rho e( \rho ) + \frac{1}{2} | \mathbf{B} |^2 \right)  \right] d ^3 x\, dt=0,
\]
for arbitrary variations $ \delta \Pi $, $ \delta \rho $, $ \delta \mathbf{B} $ and variations $ \delta \mathbf{v} $, $ \delta v$, $ \delta \mathbf{V} $ of the form
\begin{equation*}\label{variations_MHD_sdp} 
\begin{aligned} 
\delta{\bf{v}}&=\dot{\bf{w}}-[\mathbf{v},\mathbf{w}]\\
 \delta v&=\partial _t w+ \mathbf{v} \cdot \nabla w- \mathbf{w} \cdot \nabla v\\
\delta \mathbf{V} &= \partial _t \mathbf{W} +  \operatorname{curl} \mathbf{W} \times \mathbf{v} - \operatorname{curl}  \mathbf{V} \times \mathbf{w} 
\end{aligned}
\end{equation*}

\paragraph{Lie-Dirac reduction with advection.}  Starting with the canonical Dirac structure $D_{can} \subset T(T^*\operatorname{Diff}( \mathfrak{D})) \oplus T^*(T^*\operatorname{Diff}( \mathfrak{D}))$, one performs the Lie-Dirac reduction with respect to the isotropy group $\operatorname{Diff}( \mathfrak{D}  )_{ \rho _0 , \mathbf{B} _0 }$ and it follows from \eqref{condition_red_dirac} that one obtains the reduced Dirac structure $D^{/\operatorname{Diff}( \mathfrak{D}  )_{ \rho _0, \mathbf{B} _0 }}_{can}$ given at $(\Pi  , \rho , \mathbf{B} ) \in \mathfrak{X}( \mathfrak{D} ) ^\ast \times \operatorname{Orb}(\rho  _0 , \mathbf{B} _0  )$, as follows
\begin{equation*}\label{condition_red_dirac_MHD} 
\begin{array}{c} 
\left( (\Pi   , \rho , \mathbf{B} , \mathbf{v}  , \boldsymbol{\rho}  ),(\Pi  , \rho , \mathbf{B} , \boldsymbol{\beta }  , \mathbf{u}  ) \right) \in D^{/\operatorname{Diff}( \mathfrak{D}  )_{ \rho _0, \mathbf{B} _0 }}_{can}(\Pi   , \rho , \mathbf{B})\\
\Longleftrightarrow\\
\boldsymbol{\beta} + \boldsymbol{\rho}  - \operatorname{ad}^*_ \mathbf{v} \Pi   =0 \quad \text{and} \quad \mathbf{u}  =\mathbf{v}.
\end{array} 
\end{equation*}
One readily checks that the reduced Lagrange-Dirac system
\[
\left( (\Pi , \rho , \mathbf{B} ,\mathbf{v} , \dot{ \Pi }), \mathbf{d} _D^{/\operatorname{Diff}( \mathfrak{D}  )_{ \rho _0, \mathbf{B} _0 } }\ell(\mathbf{u} , \rho , \mathbf{B} )\right)  \in D^{/\operatorname{Diff}( \mathfrak{D}  )_{ \rho _0, \mathbf{B} _0 }  }_{can}( \Pi  , \rho , \mathbf{B}  )
\]
yields the MHD equations in implicit form. Similar observations hold on the Hamiltonian side, following \S\ref{subsec_preliminary_comments}. In these formulations, one always has to assume the advection equations $\partial _t \rho + \operatorname{div}( \rho \mathbf{v}   )=0$ and $\partial _t \mathbf{B} - \operatorname{curl}( \mathbf{v} \times \mathbf{B} )=0$.

As shown in \S\ref{sec_NH_SDP}, these advection equations can be included in the Dirac formulation by starting with the canonical Dirac structure $D_{can}$ on $T^*S$, $S:= \operatorname{Diff}( \mathfrak{D} ) \,\circledS\, V$. In this case, Lie-Dirac reduction by $S$ yields the Dirac structure $D^{/S}_{can}\subset \mathfrak{s} ^\ast \times (W \oplus W ^\ast )$, which is given at $( \Pi  , \rho , \mathbf{B}  )$ by
\begin{equation*}\label{condition_red_dirac_S_MHD} 
\begin{array}{c} 
\left( (\Pi , \rho , \mathbf{B}, \mathbf{v}  , w_1 ,w_2 , \boldsymbol{\rho} , b_1 , b _2 ),( \Pi  ,\rho , \mathbf{B}  , \boldsymbol{\beta}   ,c _1 , c _2  ,\mathbf{u}   ,v _1 , v _2  )\right) \in D^{/S}_{can }(\Pi , \rho , \mathbf{B} )\\
\Longleftrightarrow\\
-\boldsymbol{\rho} +\operatorname{ad}^*_ \mathbf{v} \Pi  -\rho\nabla w _1  -  \mathbf{B}\times \operatorname{curl} w _2    - \boldsymbol{\beta}  = 0,\;\; b _1 + \operatorname{div}(\mathbf{v} \rho )+ c _1 =0,\\
b _2 +\operatorname{curl}( \mathbf{B} \times   \mathbf{v}) + c _2 =0,\;\;
\mathbf{v} = \mathbf{u}   ,\;\;  w _1 = v _1 , \;\; w _2 = v _2.
\end{array} 
\end{equation*} 
The reduced Hamilton-Dirac system
\[
\left((\Pi   , \rho , \mathbf{B} , \mathbf{v} , w _1 , w _2 , \dot{ \Pi }, \dot{ \rho }, \dot{ \mathbf{B} }), \mathbf{d} ^{/S}h(\Pi , \rho , \mathbf{B} )\right) \in D^{/S}_{can}(\Pi , \rho , \mathbf{B} ).
\]
yields the MHD equations in implicit form
\[
\begin{array}{c}
\vspace{0.2cm}\dot{\Pi }- \operatorname{ad}^*_ \mathbf{v} \Pi + \rho\nabla w _1  + \mathbf{B}\times \operatorname{curl} w _2 =0,\;\; \dot{\rho  } + \operatorname{div}( \rho \mathbf{v} )=0 ,\;\;\dot{ \mathbf{B} }+ \operatorname{curl}( \mathbf{B} \times \mathbf{v} ) =0 , \\
\displaystyle\mathbf{v} = \frac{\delta h}{\delta \Pi }  ,\;\;  w _1 =\frac{\delta h}{\delta \rho  },\;\; w _2 = \frac{\delta h}{\delta \mathbf{B} }.
\end{array}
\]

The approach described here extends to all kinds of fluid models with advected quantities that admit an Euler-Poincar\'e variational formulation, as described in \cite{HMR1998a}. Extension to the case of {\it complex fluids} (see \cite{GBRa2009}) will be considered elsewhere.

\subsection{The Chaplygin rolling ball}

We shall now consider an example with nonholonomic constraints. More precisely, this is a situation that needs the theory described in \S\ref{subsec_the_sdp_case}, where the configuration Lie group is itself a semidirect product, $G=K \,\circledS \,V$, and the constraint has the special form described there, namely, of rolling ball type.

Consider a spherical ball of mass $m$, radius $r$ and whose mass distribution is inhomogeneous so that its center of mass lies anywhere in the ball as it rolls without slipping on a horizontal plane in the presence of gravity. 
The motion of the ball is described by a curve in the Euclidean group
\[
( \Lambda (t), \mathbf{x} (t))\in SE(3)= SO(3) \,\circledS\, \mathbb{R}  ^3,
\]
where $ \Lambda (t) \in SO(3)$ is the attitude of the body relative to its reference configuration and $\mathbf{x} (t)\in \mathbb{R}  ^3 $ is the coordinate of the center of mass. Let us denote by $l$ the fixed distance between the geometric center of the ball to its center of mass and by $ \chi $ the unit vector in the body frame that points from the geometric center to the center of mass. The rolling constraint reads
\begin{equation}\label{roll_constraint_ball} 
\dot{\mathbf{x}} (t)=\dot{ \Lambda }(t) \left( \Lambda ^{-1} (t) r \mathbf{e}_3+ l \boldsymbol{\chi} \right) .
\end{equation} 
See e.g. \cite{Ho2008} for an account on the geometric dynamics of Chaplygin's ball. When the center of mass coincides with the geometric center of the ball, (i.e. $l=0$), the system is called Chaplygin's sphere, also known as Routh's ball, see \cite{Bloch2003} and references therein.

Formula \eqref{roll_constraint_ball} defines the nonholonomic constraint distribution $\Delta _{SE(3)}^{\boldsymbol{\Gamma}  _0 }( \Lambda , \mathbf{x} )\subset T_{( \Lambda , \mathbf{x} )} SE(3)$ for Chaplygin's ball, where we used the notation $\boldsymbol{\Gamma}  _0 = \mathbf{e}_3$. Moreover, one observes that the invariance property \eqref{inv_prop_G} is satisfied where $S$ is $SE(3)$, $(k,x)$ is $( \Lambda , \mathbf{x} )\in SE(3)$ and $a_0$ is $\boldsymbol{\Gamma}_0 \in V^*= \mathbb{R}  ^3 $. This is an example of constraint of intermediate case (type II), i.e. given as in \eqref{case_II}. The theory developed in \S\ref{subsec_the_sdp_case} applies here with $S=G\,\circledS\, V= (K \,\circledS\, V) \,\circledS\, V= (SO(3) \,\circledS\, \mathbb{R}  ^3 ) \,\circledS\, \mathbb{R}  ^3 $.

\paragraph{Lagrangian formulation.} The Lagrangian in material representation is given by kinetic energy minus potential energy, that is,
\begin{equation}\label{Lagr_Chaplygin_top} 
L_{\mathbf{e}_3}:TSE(3) \rightarrow \mathbb{R}  , \quad 
L_{\mathbf{e}_3}( \Lambda , \dot \Lambda , \mathbf{x} , \dot{\mathbf{x}})= \frac{1}{2} \int_ \mathcal{B} \rho (\mathbf{X} )\|\dot \Lambda \mathbf{X} \|^2d^3\mathbf{X} + \frac{m}{2}|\dot{ \mathbf{x} }| ^2 -mgl \mathbf{e}_3 \cdot \Lambda \boldsymbol{\chi},
\end{equation} 
where $\mathbf{X} $ is a point in the body volume (a sphere) $ \mathcal{B} \in \mathbb{R}  ^3 $ and $ \rho (\mathbf{X} )$  is the mass density.

Note that $L_{ \mathbf{e}_3}$ is not $SE(3)$-invariant, it is only invariant under the isotropy subgroup of $ \mathbf{e}_3$, that is, $SO(2) \,\circledS\, \mathbb{R}  ^3 $, where $SO(2)$ is the group of rotation around the $\mathbf{z}$-axis. Using the notation $\mathbf{e} _3= \boldsymbol{\Gamma} _0 \in \mathbb{R}  ^3 $ we define
\[
L: TSE(3) \times \mathbb{R}  ^3 \rightarrow \mathbb{R}  , \quad L( \Lambda , \dot{ \Lambda }, \mathbf{x} , \dot{ \mathbf{x} },\boldsymbol{\Gamma}  _0 ):=L_{\boldsymbol{\Gamma} _0 }( \Lambda , \dot \Lambda , \mathbf{x} , \dot{ \mathbf{x} }),
\]
where $ \boldsymbol{\Gamma} _0 $ is interpreted as a new parameter that can be arbitrarily fixed. We now let $SE(3)$ act on $TSE(3) \times \mathbb{R}  ^3 $ by the left action
\[
( \psi , z)(  \Lambda , \dot{ \Lambda }, \mathbf{x} , \dot { \mathbf{x} }, \boldsymbol{\Gamma} _0):= \left(  \psi \Lambda , \psi \dot{ \Lambda }, \mathbf{z} + \psi \mathbf{x} , \psi \dot{ \mathbf{x} }, \psi \boldsymbol{\Gamma}  _0 \right) 
\]
with the result that $L:TSE(3) \times \mathbb{R}  ^3 \rightarrow \mathbb{R}  $ is now $SE(3)$-invariant and therefore satisfies the hypothesis of the Lagrangian reduction theorem with advected quantities.

The reduced Lagrangian is given by
\begin{equation*}\label{red_Lagr_Chaplygin} 
\ell: \mathfrak{se}(3) \times \mathbb{R}  ^3 \rightarrow \mathbb{R}  , \quad \ell( \Omega , \mathbf{X} , \boldsymbol{\Gamma}  )= \frac{1}{2} \boldsymbol{\Omega} \cdot\mathbb{I}  \boldsymbol{\Omega} + \frac{m}{2}|\mathbf{X} |^2- mgl \boldsymbol{\Gamma} \cdot \boldsymbol{\chi} ,
\end{equation*} 
where
\[
\Omega = \Lambda ^{-1} \dot{ \Lambda }, \quad \mathbf{X} = \Lambda ^{-1} \dot{\mathbf{x}},\quad \boldsymbol{\Gamma}  = \Lambda ^{-1} \mathbf{e}_3,
\]
and the vector $ \boldsymbol{\Omega} \in \mathbb{R}  ^3 $ is defined by $\hat{ \boldsymbol{\Omega} }= \Omega$, where $\hat{\,}: \mathbb{R}  ^3   \rightarrow\mathfrak{so} (3) $ is the usual Lie algebra isomorphism. The constraint subspace $ \mathfrak{g} ^\Delta(a)\subset \mathfrak{g} $ reads
\[
\mathfrak{se}(3)^ \Delta ( \boldsymbol{\Gamma} )= \Delta _{SE(3)}(e,0,\boldsymbol{\Gamma}  )=\left\{(\Omega  , \mathbf{X} )\in \mathfrak{se}(3)\mid \mathbf{X} = \Omega (r \boldsymbol{\Gamma}  +l \boldsymbol{\chi} ) \right \}.
\]
Note that we have $\Omega(r\boldsymbol{\Gamma}  +l \boldsymbol{\chi}  )=\boldsymbol{\Omega}  \times (r \boldsymbol{\Gamma}  +l \boldsymbol{\chi} )$.

At the unreduced level, the Lagrange-d'Alembert-Pontryagin principle reads
\begin{equation}\label{unreduced_HP_ball} 
\delta \int_{t _1 }^ { t _2 }\left\{L_{ \mathbf{e}_3}( \Lambda , V, \mathbf{x} , \mathbf{w} )+ \left\langle p, \dot \Lambda -V \right\rangle + \left\langle\boldsymbol{\pi}  , \dot{ \mathbf{x}}- \mathbf{w}  \right\rangle \right\}dt=0
\end{equation} 
for variations $(\delta \Lambda , \delta \mathbf{x} ,\delta{V},\delta \mathbf{w} ,\delta{p}, \delta \boldsymbol{\pi} )$ of  $(\Lambda ,\mathbf{x} ,V,\mathbf{w} ,p, \boldsymbol{\pi}  )$ in $TSE(3)\oplus T^{\ast}SE(3)$ vanishing at the endpoints, with $(\delta \Lambda , \delta \mathbf{x} )\in \Delta _{SE(3)}^{\mathbf{e}_3}(\Lambda ,\mathbf{x} ) $ and with the constraint $(V,\mathbf{w} )\in \Delta _{SE(3)}^{\mathbf{e}_3}(\Lambda ,\mathbf{x} )$, that is,
\[
\delta \mathbf{x} =\delta \Lambda  \left( \Lambda ^{-1}  r \mathbf{e}_3+ l \boldsymbol{\chi}  \right) \quad \text{and} \quad  \mathbf{w} =V \left( \Lambda ^{-1} r \mathbf{e}_3+ l\boldsymbol{\chi}  \right) .
\]
At the reduced level, it follows from \eqref{ConstHP-ActInt_constrained_SDP} that we have the reduced Lagrange-d'Alembert-Pontryagin principle
\begin{equation}\label{red_HP_ball} 
\delta \int_{t_1}^{t_2} \left\{ \ell(\psi , \mathbf{Y} , \boldsymbol{\Gamma} )+ 
\left\langle \Pi , \Omega  - \psi \right\rangle+\left\langle \boldsymbol{\lambda}  , \mathbf{X}  - \mathbf{Y}  \right\rangle \right\} ~dt=0,
\end{equation} 
with respect to variations of the form
\begin{equation}\label{red_HP_ball_variations} 
\delta{\boldsymbol{\Omega} }=\frac{\partial \boldsymbol{\zeta} }{\partial t}+\boldsymbol{\Omega}  \times \boldsymbol{\zeta}  , \quad \delta \mathbf{X} =\frac{\partial \mathbf{Z} }{\partial t}+\boldsymbol{\Omega}  \times \mathbf{Z} - \boldsymbol{\zeta}  \times \mathbf{X} ,\quad \delta{\boldsymbol{\Gamma} }=-\boldsymbol{\zeta}  \times \boldsymbol{\Gamma} ,
\end{equation} 
where $(\zeta,\mathbf{Z} ) \in \mathfrak{se}(3)^{ \Delta}(\boldsymbol{\Gamma})$ vanishes at the endpoints and $ (\psi  ,\mathbf{Y} ) $ verifies the constraint $(\psi  ,\mathbf{Y} )\in \mathfrak{se}(3)^{ \Delta}(\boldsymbol{\Gamma})$, that is,
\[
\mathbf{Z} = \zeta (r \boldsymbol{\Gamma}  +l \boldsymbol{\chi}  ) \quad \text{and} \quad \mathbf{Y} = \psi (r \boldsymbol{\Gamma} + l \boldsymbol{\chi}  ). 
\]
The relations with the variables of the unreduced level are given as follows:
\begin{align*} 
(\psi , \mathbf{Y} )&= ( \Lambda , \mathbf{x} ) ^{-1} ( \Lambda , V,\mathbf{x} ,\mathbf{w} )=( \Lambda ^{-1} v, \Lambda ^{-1} \mathbf{w} )\in \mathfrak{se}(3),\\
(\Pi , \boldsymbol{\lambda} )&=( \Lambda , \mathbf{x} ) ^{-1} ( \Lambda , p, \mathbf{x} , \boldsymbol{\pi}  )= ( \Lambda ^{-1} p, \Lambda ^{-1} \boldsymbol{\pi}  )\in \mathfrak{se}(3) ^\ast, \\
(\Omega , \mathbf{X} )&= ( \Lambda , \mathbf{x} ) ^{-1} ( \Lambda , \dot{ \Lambda }, \mathbf{x} , \dot{ \mathbf{x} })= ( \Lambda ^{-1}\dot \Lambda , \Lambda ^{-1} \dot{ \mathbf{x} })\in \mathfrak{se}(3),\\
\boldsymbol{\Gamma}  &= \Lambda ^{-1} \mathbf{e}_3.
\end{align*} 
The reduced Lagrange-d'Alembert-Pontryagin principle yields the equations of Chaplygin's ball in implicit form as follows (see also  \eqref{implicit_EP_constrained_SDP}):
%
%\begin{align*} 
\begin{equation*}
\begin{split}
& \boldsymbol{\Pi} = \frac{\delta \ell}{\delta \boldsymbol{\psi} }=\mathbb{I}   \boldsymbol{\psi} , \quad \boldsymbol{\lambda}  = \frac{\delta \ell}{\delta \mathbf{Y} }= m \mathbf{Y}, \quad ({\Omega}, \mathbf{X} )=( {\psi}, \mathbf{Y} )\in\mathfrak{se}(3)^ \Delta ( \boldsymbol{\Gamma}  ), \quad \mathbf{Y} = \psi (r\boldsymbol{\Gamma}  +l \boldsymbol{\chi}  ), \\
&\left(  \partial _t {\Pi} - {\Pi} \times \boldsymbol{\Omega}  + {X} \times \boldsymbol{\lambda}  - \frac{\delta \ell}{\delta \boldsymbol{\Gamma}  }\times \boldsymbol{\Gamma}   , \dot { \boldsymbol{\lambda} }+ \boldsymbol{\Omega}  \times \boldsymbol{\lambda}   \right) \in\mathfrak{se}(3)^ \Delta ( \boldsymbol{\Gamma}  )^\circ.
\end{split}
\end{equation*}
%\end{align*} 

Since the constraint is of the form \eqref{case_II}, with $\phi(\boldsymbol{\Gamma}  )=(r\boldsymbol{\Gamma}  +l \boldsymbol{\chi}  )$, we can rewrite this equation using \eqref{case_II_equations} and we find
\[
\left( \partial _t + \boldsymbol{\Omega} \times \right) \left(\boldsymbol{\Pi }  + \phi(\boldsymbol{\Gamma} ) \times \boldsymbol{\lambda} \right)=\frac{\delta\ell}{\delta \boldsymbol{\Gamma} } \times \boldsymbol{\Gamma}  + \partial _t \phi(\boldsymbol{\Gamma} ) \times \boldsymbol{\lambda}.
\]
The equation of motion is thus given by
\[
\left( \partial _t + \boldsymbol{\Omega} \times \right) \left(\mathbb{I}   \boldsymbol{\Omega}   + \phi(\boldsymbol{\Gamma} ) \times m \mathbf{X}  \right)=-mgl \boldsymbol{\chi} \times \boldsymbol{\Gamma}  + \partial _t \phi(\boldsymbol{\Gamma} ) \times m \mathbf{X} , 
\]
where $\mathbf{X} = \boldsymbol{\Omega} \times  \phi (\boldsymbol{\Gamma} )$ and $ \partial _t \phi( \boldsymbol{\Gamma} )=-r \boldsymbol{\Omega} \times \boldsymbol{\Gamma} $, thanks to the advection equation $ \partial _t \boldsymbol{\Gamma} =- \boldsymbol{\Omega} \times \boldsymbol{\Gamma} $.
This recovers the equations of motion of Chaplygin's ball (see, for instance, (12.2.9) in \cite{Ho2008}), where it is shown how these equations can be equivalently rewritten in a more compact form.

Note that the alternative variational structure described in Remark \ref{Hybrid_HP_advected_constraints_sdp} yields
\[
\delta \int_{ t _1 }^{ t _2 } \{ \ell(\boldsymbol{\Omega} ,\mathbf{X} ,\boldsymbol{\Gamma}  )+ \left\langle \Pi  , \Lambda  ^{-1} \dot \Lambda  - \boldsymbol{\Omega}  \right\rangle+ \left\langle \boldsymbol{\lambda} , \Lambda  ^{-1} \dot { \mathbf{x} }-\mathbf{X}  \right\rangle + \left\langle v, \Lambda ^{-1} \mathbf{e} _3  - \boldsymbol{\Gamma} \right\rangle  \} \,dt=0,
\]
with $(\dot \Lambda , \dot { \mathbf{x} }) \in \Delta _{SE(3)}^{ \mathbf{e} _3 }( \Lambda , \mathbf{x} )$, i.e. $\dot{\mathbf{x}} =\dot\Lambda  \left( \Lambda ^{-1}  r \mathbf{e}_3+ l \boldsymbol{\chi}  \right) $,
for arbitrary variations $ \delta \boldsymbol{\Omega} (t) $, $ \delta \mathbf{X} (t)$, $ \delta\Pi  (t) $, $ \delta \boldsymbol{\lambda} (t)$, $ \delta v(t)$, $ \delta \boldsymbol{\Gamma}  (t)$ of $ \boldsymbol{\Omega}  (t) $, $\mathbf{X} (t)$, $\Pi  (t) $, $\boldsymbol{\lambda} (t)$, $v(t)$, $\boldsymbol{\Gamma} (t) $ and variations $\delta \Lambda (t)$, $ \delta \mathbf{x} (t)$ of $\Lambda (t)$, $\mathbf{x} (t)$ vanishing at the endpoints and verifying
\[
(\delta  \Lambda , \delta  { \mathbf{x} }) \in \Delta _{SE(3)}^{ \mathbf{e} _3 }( \Lambda , \mathbf{x} )\quad \text{i.e.} \quad  \delta {\mathbf{x}} =\delta \Lambda  \left( \Lambda ^{-1}  r \mathbf{e}_3+ l \boldsymbol{\chi}  \right).
\]
This reduced Lagrange-d'Alembert-Pontryagin principle recovers the one developed in \cite[\S12.2]{Ho2008}.

\paragraph{Hamiltonian formulation.} The variational structures on the Hamiltonian side were shown in \S\ref{Hamiltonian_side} and \S\ref{PSVP_TstarS}. The particular case of a semidirect product was given in Remark \ref{sdp_case_Ham} and Remark \ref{sdp_case_Ham_bis}.

Let us denote by $H_{ \mathbf{e} _3 }:T^{ \ast}SE(3) \rightarrow \mathbb{R}  $, the Hamiltonian of Chaplygin's ball obtained from the Lagrangian \eqref{Lagr_Chaplygin_top} by Legendre transform. The Hamilton-d'Alembert principle in phase space reads
\[
\delta \int_{ t _1 }^{ t _2 } \left\{ \left\langle p, \dot \Lambda  \right\rangle+ \left\langle \boldsymbol{\pi}  , \dot{ \mathbf{x} } \right\rangle  -H_{ \mathbf{e} _3 }(\Lambda ,p, \mathbf{x}  , \boldsymbol{\pi}  )\right\} dt=0,\quad (\dot \Lambda  , \dot{ \mathbf{x} }) \in \Delta ^{ \mathbf{e} _3 }_{SE(3)}(\Lambda ,\mathbf{x} ),
\]
for variations $ \delta \Lambda  (t) , \delta \mathbf{x}  (t) $ of $\Lambda (t), \mathbf{x} (t) $ vanishing at the endpoints and such that $(\delta \Lambda , \delta \mathbf{x} )\in \Delta _{SE(3)}^{\mathbf{e} _3}$, and for arbitrary variations $ \delta p (t) , \delta \boldsymbol{\pi}  (t) $ of $p(t), \boldsymbol{\pi}  (t) $. By reduction, we have the reduced Hamilton-d'Alembert principle for Chaplygin's ball
\begin{equation*}\label{Lie_Poisson_VPSDP_first_SDP} 
\delta \int_{t_1}^{t_2} \left \{ \left\langle\Pi  , \boldsymbol{\Omega}   \right\rangle + \left\langle \boldsymbol{\lambda} , \mathbf{X}  \right\rangle -h(\Pi  ,\boldsymbol{\lambda} , \Gamma ) \right\}dt=0,\quad (\boldsymbol{\Omega} ,\mathbf{X} ) \in \mathfrak{se}(3)  ^ \Delta (\boldsymbol{\Gamma} ),
\end{equation*} 
for arbitrary variations $ \delta\Pi  (t) , \delta \boldsymbol{\lambda} (t) $ and variations $ \delta \boldsymbol{\Omega}  (t) $, $ \delta \mathbf{X} $ and $ \delta\boldsymbol{\Gamma} (t) $ of the form $\delta{\boldsymbol{\Omega} }=\partial _t\boldsymbol{\zeta} +\boldsymbol{\Omega} \times \boldsymbol{\zeta}$, $\delta \mathbf{X} =\partial _t \mathbf{Z} + \boldsymbol{\Omega} \times \mathbf{Z} - \boldsymbol{\zeta} \times \mathbf{X}$ and $\delta \boldsymbol{\Gamma} =- \boldsymbol{\zeta} \times \boldsymbol{\Gamma}$, where $( \boldsymbol{\zeta} ,\mathbf{Z} ) \in \mathfrak{se}(3)  ^ \Delta (\boldsymbol{\Gamma} )$ and vanishes at the endpoints.

We shall now apply another approach developed in \S\ref{PSVP_TstarS}, more precisely in Remark \ref{sdp_case_Ham_bis}, which implies the advection equation without having to assume it a priori. The Hamilton-d'Alembert principle in phase space reads
\[
\delta \int_{ t _1 }^{ t _2 } \left\{ \left\langle p, \dot \Lambda  \right\rangle+ \left\langle \boldsymbol{\pi}  , \dot { \mathbf{x} } \right\rangle + \left\langle \boldsymbol{\Gamma} _0 , \dot { \mathbf{u} } \right\rangle  -H(\Lambda ,p,\mathbf{x} ,\boldsymbol{\pi}  , \boldsymbol{\Gamma}  _0 )\right\} dt=0,\quad (\dot \Lambda , \dot { \mathbf{x} }) \in \Delta ^{ \boldsymbol{\Gamma}  _0 }_{SE(3)}(\Lambda ,\mathbf{x} )
\]
for variations $ \delta \Lambda (t) $, $ \delta \mathbf{x} (t) $, $ \delta \mathbf{u}  (t)$ of $\Lambda (t)$, $ \mathbf{x}  (t) $, $\mathbf{u} (t)$ vanishing at the endpoints and such that $(\delta \Lambda , \delta \mathbf{x} )\in \Delta _{SE(3)}^{\boldsymbol{\Gamma} _0}$, and for arbitrary variations $ \delta p (t) $, $ \delta \boldsymbol{\pi}  (t) $, $ \delta \boldsymbol{\Gamma}  _0 (t)$ of $p(t)$, $ \boldsymbol{\pi}  (t) $, $ \boldsymbol{\Gamma}  _0 (t)$.

Then, the reduced Hamilton-d'Alembert principle in \eqref{Lie_Poisson_VPSDP_SDP} reads
\begin{equation*} 
\delta \int_{t_1}^{t_2} \left \{ \left\langle \Pi  , \boldsymbol{\Omega}   \right\rangle+ \left\langle \boldsymbol{\lambda} , \mathbf{X}  \right\rangle + \left\langle \boldsymbol{\Gamma} ,\mathbf{v}  \right\rangle  -h(\Pi  ,\boldsymbol{\lambda} , \boldsymbol{\Gamma} ) \right\}dt=0,\quad (\boldsymbol{\Omega} ,\mathbf{X} ) \in \mathfrak{se}(3)  ^ \Delta (\boldsymbol{\Gamma} ),
\end{equation*} 
for arbitrary variations $ \delta \Pi  (t) $, $ \delta \boldsymbol{\lambda} (t) $, $ \delta \boldsymbol{\Gamma}  (t) $ and variations $ \delta \boldsymbol{\Omega}  (t) $, $ \delta \mathbf{X} (t) $ and $ \delta v(t) $ of the form $\delta \boldsymbol{\Omega}   = \partial _t \boldsymbol{\zeta}  +\boldsymbol{\Omega} \times\boldsymbol{\zeta} $, $ \delta \mathbf{X} = \partial _t \mathbf{Z} + \boldsymbol{\Omega} \times  \mathbf{Z} - \boldsymbol{\zeta} \times \mathbf{X}$, and $\delta \mathbf{v} = \partial _t \mathbf{w} + \boldsymbol{\Omega} \times  \mathbf{w} - \boldsymbol{\zeta} \times  \mathbf{v}  $, where $ (\boldsymbol{\zeta}  ,\mathbf{Z} )\in \mathfrak{se}(3)^ \Delta (\boldsymbol{\Gamma} )$ vanishes at the endpoints and $\mathbf{w}  \in \mathbb{R}  ^3 $ is arbitrary and vanishes at endpoints.

\paragraph{Dirac formulation of Chaplygin's rolling ball.} Given the rolling constraint $\Delta _{SE(3)}^{\mathbf{e} _3 }\subset SE(3) $ of Chaplygin's ball as in \eqref{roll_constraint_ball}, we consider the induced Dirac structure $D_{ \Delta ^{\mathbf{e} _3 }_{SE(3)}} $ on $T^*SE(3)$. Reduction by the isotropy group $SE(3)_ { \mathbf{e} _3 }$ yields the Dirac structure $D_{\text{Chaplygin}}^{ \mathbf{e} _3 }:=D^{/SE(3)_ { \mathbf{e} _3 }}_{\Delta _{SE(3)}^{\mathbf{e} _3 }}$ given by
\begin{equation*}\label{condition_red_dirac_rollingtype_Chaplygin} 
\begin{array}{c} 
\left( ( \boldsymbol{\Pi}  ,\mathbf{b} , \boldsymbol{\Gamma} , \boldsymbol{\Omega} ,\mathbf{X}  , \boldsymbol{\rho}  , \boldsymbol{\sigma}  ),( \boldsymbol{\Pi}  ,\mathbf{b} ,\boldsymbol{\Gamma} , \boldsymbol{\beta}  , \boldsymbol{\gamma}  , \boldsymbol{\Psi} , \mathbf{Y}  ) \right) \in D_{\text{Chaplygin}}^{ \mathbf{e} _3 }\\
\Longleftrightarrow\\
\left( \boldsymbol{\beta}  + \boldsymbol{\rho}  + \boldsymbol{\Omega} \times  \boldsymbol{\Pi} + \mathbf{X}  \times  \mathbf{b} , \boldsymbol{\gamma}  + \boldsymbol{\sigma}  - \boldsymbol{\Omega} \times \mathbf{b} \right)  \in (\mathfrak{g}  ^\Delta (\boldsymbol{\Gamma} ))^\circ \quad \text{and} \quad (\boldsymbol{\Psi}  , \mathbf{Y} ) = (\boldsymbol{\Omega} , \mathbf{X} )  \in \mathfrak{g}  ^\Delta (\boldsymbol{\Gamma} ),
\end{array} 
\end{equation*}
where we recall that $\mathfrak{g}  ^\Delta (\boldsymbol{\Gamma} )= \mathfrak{se}(3)^ \Delta ( \boldsymbol{\Gamma} )=\left\{(\Omega  , \mathbf{X} )\in \mathfrak{se}(3)\mid \mathbf{X} = \Omega (r \boldsymbol{\Gamma}  +l \boldsymbol{\chi} ) \right \}$.

It follows that the equations for the Chaplygin rolling ball can be written as the Lagrange-Dirac system
\[
\left( ( \boldsymbol{\Pi}  , \mathbf{b} , \boldsymbol{\Gamma} ,\boldsymbol{\Omega}  , \mathbf{X} , \dot {\boldsymbol{\Pi} },\dot {\mathbf{b}} ),\mathbf{d} ^{/ SE(3)_{ \mathbf{e} _3  }}_D\ell (\boldsymbol{\Psi}  , \mathbf{Y} , \boldsymbol{\Gamma} )\right) \in D_{\text{Chaplygin}}^{ \mathbf{e} _3 },
\]
or, as a Hamilton-Dirac system
\[
\left( ( \boldsymbol{\Pi}  , \mathbf{b} , \boldsymbol{\Gamma} ,\boldsymbol{\Omega}  , \mathbf{X} , \dot {\boldsymbol{\Pi} },\dot {\mathbf{b}} ),\mathbf{d} ^{/ SE(3)_{ \mathbf{e} _3  }}h (\boldsymbol{\Pi}  , \mathbf{b} , \boldsymbol{\Gamma} )\right) \in D_{\text{Chaplygin}}^{ \mathbf{e} _3 }.
\]

The advection equation can be included in the Dirac formulation, by considering the induced Dirac structure $D_{\Delta _{T^*S}}$ on $T^*S=T^*(SE(3) \,\circledS\, \mathbb{R}  ^3 )$, following the approach developed in \S\ref{sec_NH_SDP}. The Hamilton-Dirac formulation of Chaplygin's ball is
\[
\left(( \boldsymbol{\Pi}  ,\mathbf{b} , \boldsymbol{\Gamma} ,  \boldsymbol{\Omega} , \mathbf{X}  , \mathbf{w} , \dot { \boldsymbol{\Pi} }, \dot { \mathbf{b} }, \dot{\boldsymbol{\Gamma} }), \mathbf{d} ^{/S}h( \boldsymbol{\Pi} ,\mathbf{b}  , \boldsymbol{\Gamma} ) \right) \in D_{ \text{Chaplygin}},
\]
where $D_{ \text{Chaplygin}}:= D^{/S}_{ \Delta _{T^*S}}$.

\subsection{The Euler disk}

Consider a flat circular disk with homogeneous mass distribution which rolls without slipping on a horizontal plane, and whose orientation is allowed to tilt away from the vertical plane. Denote its mass and radius by $m$ and $r$, respectively. Because its mass distribution is homogeneous, the center of mass coincides with the geometric center of the disk.

Let $(\mathbf{E} _1, \mathbf{E} _2 , \mathbf{E} _3 )$ denote an orthonormal basis in the reference configuration whose origin is attached to the disk at its center of symmetry, so that the disk lies in the $(\mathbf{E} _1 , \mathbf{E} _2 )$ plane in body coordinates. Let $(\mathbf{e} _1,\mathbf{e} _2 , \mathbf{e} _3 )$ be an orthonormal basis of $\mathbb{R}  ^3$ such that $(\mathbf{e} _1, \mathbf{e} _2)$ lies in the horizontal plane and $\mathbf{e} _3$ is a vertical unit vector.

The motion of the circular disk is given by a curve $(\Lambda (t), \mathbf{x} (t)) \in  SE(3)=SO(3) \,\circledS\, \mathbb{R}  ^3 $, where $ \Lambda (t) \in SO(3)$ is the attitude of the disk relative to its reference configuration and $\mathbf{x} (t)\in \mathbb{R}  ^3 $ is the coordinate of the center of mass.

We shall now describe the nonholonomic constraint. Since the vector $ \Lambda (t) \mathbf{E}  _3 $ is normal to the disk and the vector $\mathbf{e} _3 \times \Lambda (t) \mathbf{E} _3 $ is tangent to the edge of the disk at the point of contact on the plane, the vector $ \mathbf{u} (t) := \Lambda (t) \mathbf{E} _3 \times  (\mathbf{e} _3 \times \Lambda (t) \mathbf{E} _3)$ points from the contact point in the direction of the geometric center. Therefore, the vector $ \boldsymbol{\sigma} (t) $ relating the contact point to the geometric center of the disk in the spatial frame is given by
\[
\boldsymbol{\sigma} (t)= r \frac{\mathbf{u} (t) }{| \mathbf{u} (t) |}= r \frac{\Lambda (t) \mathbf{E} _3 \times  (\mathbf{e} _3 \times \Lambda (t) \mathbf{E} _3)}{|\Lambda (t) \mathbf{E} _3 \times  (\mathbf{e} _3 \times \Lambda (t) \mathbf{E} _3)|}. 
\]
In the body frame, this vector is $\mathbf{s} (t):=  \Lambda (t) ^{-1} \boldsymbol{\sigma} (t) $, so the rolling constraint at the contact point is therefore
\begin{equation}\label{roll_constraint_disk} 
\dot{ \mathbf{x} }(t) =\dot{\Lambda} (t) \mathbf{s} (t) = \dot{\Lambda} (t)\Lambda (t)^{-1}\boldsymbol{\sigma} (t) .
\end{equation} 
We have followed the notations of \cite{Ho2008} to which we refer for more information.

Formula \eqref{roll_constraint_disk} defines the nonholonomic constraint distribution $\Delta _{SE(3)}^{\boldsymbol{\Gamma}  _0 }( \Lambda , \mathbf{x} )\subset T_{( \Lambda , \mathbf{x} )} SE(3)$ for the rolling disk, where we used the notation $\boldsymbol{\Gamma}  _0 = \mathbf{e}_3$. Moreover, one observes that the invariance property \eqref{inv_prop_G} is satisfied where $G$ is $SE(3)$, $(k,x)$ is $( \Lambda , \mathbf{x} )\in SE(3)$ and $a_0$ is $\boldsymbol{\Gamma}_0 \in V^*= \mathbb{R}  ^3 $.

The Lagrangian in material representation is given by kinetic energy minus potential energy, that is,
\[
L_{\mathbf{e}_3}:TSE(3) \rightarrow \mathbb{R}  , \quad 
L_{\mathbf{e}_3}( \Lambda , \dot \Lambda , \mathbf{x} , \dot{\mathbf{x}})= \frac{1}{2} \int_ \mathcal{B} \rho (\mathbf{X} )\|\dot \Lambda \mathbf{X} \|^2d^3\mathbf{X} + \frac{m}{2}|\dot{ \mathbf{x} }| ^2 -mg\mathbf{e}_3 \cdot \boldsymbol{\sigma}( \Lambda , \mathbf{e} _3  ) ,
\]
where $\mathbf{X} $ is a point in the body volume (a sphere) $ \mathcal{B} \in \mathbb{R}  ^3 $, $ \rho (\mathbf{X} )$  is the mass density, and
\[
\boldsymbol{\sigma} ( \Lambda, \mathbf{e} _3  )= r \frac{\Lambda \mathbf{E} _3 \times  (\mathbf{e} _3 \times \Lambda  \mathbf{E} _3)}{|\Lambda \mathbf{E} _3 \times  (\mathbf{e} _3 \times \Lambda \mathbf{E} _3)|}.
\]

The Lagrangian $L_{ \mathbf{e}_3}$ is only invariant under the isotropy subgroup of $ \mathbf{e}_3$, that is, $SO(2) \,\circledS\, \mathbb{R}  ^3 $. As before, we use the notation $\mathbf{e} _3= \boldsymbol{\Gamma} _0 \in \mathbb{R}  ^3 $ and $L_{\boldsymbol{\Gamma} _0 }$, where $ \boldsymbol{\Gamma} _0 $ is interpreted as a new parameter that can be arbitrarily fixed. We obtain the reduced Lagrangian
\[
\ell: \mathfrak{se}(3) \times \mathbb{R}  ^3 \rightarrow \mathbb{R}  , \quad \ell( \Omega , \mathbf{X} , \boldsymbol{\Gamma}  )= \frac{1}{2} \boldsymbol{\Omega} \cdot\mathbb{I}  \boldsymbol{\Omega} + \frac{m}{2}|\mathbf{X} |^2- mg \boldsymbol{\Gamma} \cdot \mathbf{s} ( \Gamma ) ,
\]
where
\[
\Omega = \Lambda ^{-1} \dot{ \Lambda }, \quad \mathbf{X} = \Lambda ^{-1} \dot{\mathbf{x}},\quad \boldsymbol{\Gamma}  = \Lambda ^{-1} \mathbf{e}_3,\quad \mathbf{s} ( \boldsymbol{\Gamma} )= \Lambda ^{-1} \boldsymbol{\sigma} ( \Lambda , \boldsymbol{\Gamma}  _0 )= r \frac{ \mathbf{E} _3 \times  (\boldsymbol{\Gamma} \times  \mathbf{E} _3)}{|\mathbf{E} _3 \times  (\boldsymbol{\Gamma} \times  \mathbf{E} _3)|}
\]
and the vector $ \boldsymbol{\Omega} \in \mathbb{R}  ^3 $ is defined by $\hat{ \boldsymbol{\Omega} }= \Omega$, where $\hat{\,}: \mathbb{R}  ^3   \rightarrow\mathfrak{so} (3) $ is the usual Lie algebra isomorphism. The subspace $ \mathfrak{g}^{\Delta}(a)\subset \mathfrak{g} $ reads
\[
\mathfrak{se}(3)^{\Delta}(\boldsymbol{\Gamma}  )= \Delta _{SE(3)}(e,0,\boldsymbol{\Gamma}  )=\left\{(\Omega  , \mathbf{X} )\in \mathfrak{se}(3)\mid \mathbf{X} = \Omega\,\mathbf{s} ( \boldsymbol{\Gamma} ) \right \}.
\]

At the unreduced level, the Lagrange-d'Alembert-Pontryagin principle is given as in \eqref{unreduced_HP_ball} 
where the variations verify
\[
\delta \mathbf{x} =\delta \Lambda \Lambda ^{-1} \boldsymbol{\sigma} ( \Lambda , \mathbf{e} _3 ) \quad \text{and} \quad  \mathbf{w} =V  \Lambda ^{-1} \boldsymbol{\sigma}( \Lambda , \mathbf{e} _3 ).
\]
At the reduced level, it is given as in \eqref{red_HP_ball}-\eqref{red_HP_ball_variations} where now $(\zeta,\mathbf{Z} )$ and $ (\psi  ,\mathbf{Y} ) $ verify the constraints:
\[
\mathbf{Z} = \zeta \,\mathbf{s}( \boldsymbol{\Gamma} ) \quad \text{and} \quad \mathbf{Y} = \psi\, \mathbf{s} ( \boldsymbol{\Gamma} ) . 
\]

As shown in \eqref{implicit_EP_constrained_SDP}, the reduced Lagrange-d'Alembert-Pontryagin principle yields the implicit Euler-Poincar\'e-Suslov equations with advected parameters for  the Euler disk:
\begin{align*} 
& \;\; \boldsymbol{\Pi} = \frac{\delta \ell}{\delta \boldsymbol{\psi} }=\mathbb{I}   \boldsymbol{\psi} , \quad \boldsymbol{\lambda}  = \frac{\delta \ell}{\delta \mathbf{Y} }= m \mathbf{Y} ,  \\
& \;\; \mathbf{Y} = \psi\, \mathbf{s} ( \boldsymbol{\Gamma} )= r\psi  \frac{ \mathbf{E} _3 \times  (\boldsymbol{\Gamma} \times  \mathbf{E} _3)}{|\mathbf{E} _3 \times  (\boldsymbol{\Gamma} \times  \mathbf{E} _3)|}, \quad ( \Omega , \mathbf{X} )=( \psi , \mathbf{Y} ) \in \mathfrak{se}(3)^ \Delta ( \boldsymbol{\Gamma}  ), \\
& \left(  \partial _t \Pi - \Pi \times \boldsymbol{\Omega}  +X \times \boldsymbol{\lambda}  - \frac{\delta \ell}{\delta \boldsymbol{\Gamma}  }\times \boldsymbol{\Gamma}   , \dot { \boldsymbol{\lambda} }+ \boldsymbol{\Omega}  \times \boldsymbol{\lambda}   \right) \in(\mathfrak{se}(3)^ \Delta ( \boldsymbol{\Gamma}  ))^\circ.
\end{align*} 
Since the constraints are of the form \eqref{case_II}, with $\phi(\boldsymbol{\Gamma}  )=\mathbf{s} ( \boldsymbol{\Gamma} )$, we can rewrite the equations of motion by using \eqref{case_II_equations} and we find
\[
\left( \partial _t + \boldsymbol{\Omega} \times \right) \left(\mathbb{I}   \boldsymbol{\Omega}   + \mathbf{s} (\boldsymbol{\Gamma} ) \times m \mathbf{X}  \right)=-mgl \boldsymbol{\chi} \times \boldsymbol{\Gamma}  + \partial _t \mathbf{s} (\boldsymbol{\Gamma} ) \times m \mathbf{X} , 
\]
where $\mathbf{X} = \boldsymbol{\Omega} \times  \mathbf{s}  (\boldsymbol{\Gamma} )$.
This recovers the equations of motion of Euler's disk (see, for instance, (12.2.38) in \cite{Ho2008}).
\medskip

The variational and Dirac formulations mentioned in the example of Chaplygin's rolling ball can be carried out in a similar way for Euler's disk. We shall not repeat them here.
\medskip

Another point of view on the equations of motion for the Euler disk was shown in \cite{CeDi2007} using the Lagrange-Poincar\'e-d'Alembert equations.

\subsection{Second-order Rivlin-Ericksen fluids}

In this section, we show that equations of motion for second-order Rivlin-Ericksen fluids can be formulated as an infinite dimensional nonholonomic systems with rolling ball type constraints in the framework developed in \S\ref{subsec_the_sdp_case}. The role of the group $G=SO(3) \,\circledS\, \mathbb{R}  ^3 $ is played by the semidirect product group $G= \operatorname{Diff}_{vol}(\mathfrak{D} )\,\circledS\, S _2 ( \mathfrak{D} )$, where $V= S_2 ( \mathfrak{D} )$ is the space of $2$-covariant symmetric tensor fields.

A well-known constitutive expression for the stress in an incompressible non-Newtonian fluid is provided by the representation of the extra stress as a function of the Rivlin-Ericksen tensors $\mathbf{A}  _1, \mathbf{A} _2 , ...$ (see \cite{RiEr1955}).

For second-order Rivlin-Ericksen fluids, the Cauchy-stress tensor is given (in matrix notation) by
\[
\tau = - p \mathbf{I} + \mu \mathbf{A} _1 + \alpha _1 \mathbf{A} _2 + \alpha _2 \mathbf{A} _1 ^2 ,
\]
where $p$ is the pressure, $ \mu $ is the coefficient of viscosity, $ \alpha _1 $, $ \alpha _2 $ are material moduli, and the Rivlin-Ericksen tensors $ \mathbf{A} _1 , \mathbf{A} _2 $ are given in terms of $ \nabla \mathbf{v} $ by
\begin{equation}\label{def_A_1_2} 
\mathbf{A} _1 := \nabla \mathbf{v} + \nabla \mathbf{v} ^\mathsf{T}  \quad\text{and}\quad \mathbf{A} _2 = \partial _t \mathbf{A} _1 + \mathbf{v}\cdot  \nabla \mathbf{A} _1 + \mathbf{A} _1 \nabla \mathbf{v} + \nabla \mathbf{v} ^\mathsf{T} \mathbf{A} _1 .
\end{equation} 
For simplicity, we will denote $ \mathbf{A} _1 $ by $ \mathbf{A} $. We refer to  e.g. \cite{Jo1990} and \cite{TrRa2008} for more information about second-order fluids.

In \cite{DuFo1974} (see also \cite{FoRa1979}), it was found that if we impose the second-order Rivlin-Ericksen fluids to be compatible with thermodynamics, in the sense that all motions of the fluid meet the Clausius-Duhem inequality and the assumption that the specific Helmholtz free energy of the fluid takes its minimum value in equilibrium, then the material moduli must satisfy
$ \mu \geq 0$, $ \alpha _1 \geq 0$, and $ \alpha _1 + \alpha _2 =0$. If these conditions are satisfied, the second-order Rivlin-Ericksen equations coincide with the averaged Euler equations when $ \mu =0$ (see \cite{HMR1998b}).
If $\alpha_1 = \alpha _2  = 0$ then one recovers the constitutive law for the Navier-Stokes equations.
The appropriate restriction to be assumed on the moduli coefficients is at the origin of an extensive discussion in the literature.
For example, the belief of many mechanicians and rheologists is that for fluids that exhibit stress relaxation one has $ \alpha _1 <0$. 
We refer to \cite{DuRa1995} for a review of discussions regarding these coefficients.
Below, we shall only deal with the variational structures underlying these equations and we shall not assume any conditions on these coefficients.

\medskip

\paragraph{Second-order fluids on Riemannian manifolds.} In order to formulate the problem in an intrinsic way on a Riemannian manifold $( \mathfrak{D} , g)$, we shall consider the tensor $ \mathbf{A} $ as a $2$-covariant symmetric tensor field on $ \mathfrak{D}$, i.e. we write
\[
\mathbf{A} = (\nabla \mathbf{v} + \nabla \mathbf{v} ^\mathsf{T})^ \flat,
\]
where $ \nabla $ is the Levi-Civita covariant derivative associated with a Riemannian metric $g$ and the flat operator indicates that we are lowering the contravariant index down by using $g$. For simplicity we assume that $ \mathfrak{D} $ has no boundary.

We then recognize that the right hand side of the second equality in \eqref{def_A_1_2} can be intrinsically written as the Lie derivative of the $2$-covariant tensor $ \mathbf{A}$, denoted $ \pounds _ \mathbf{v} \mathbf{A}$. Using this notations, the equations for a second-order Rivlin-Ericksen fluid associated to the constitutive relation \eqref{def_A_1_2} on a Riemannian manifold $( \mathfrak{D} , g)$ read
\[
\partial _t \mathbf{v} + \mathbf{v} \cdot \nabla \mathbf{v} = - \nabla p + \mu \Delta \mathbf{v} + \alpha _1 \operatorname{div}( \partial _t \mathbf{A} + \pounds _ \mathbf{v} \mathbf{A}  )^\sharp + \alpha _2 \operatorname{div} (\mathbf{A}  ^2 )^\sharp,\quad \operatorname{div} \mathbf{v} =0, 
\]  
where $ \operatorname{div}$ denotes the (metric) divergence of a $2$-covariant symmetric tensor field, i.e., it is the one-form $ \operatorname{div} \mathbf{A} \in \Omega ^1 ( \mathfrak{D} )$ defined as
\[
( \operatorname{div}\mathbf{A} )_i:= \nabla _k \mathbf{A} _i ^k, \quad \text{where}
\quad \mathbf{A} ^k _i = g ^{kj} \mathbf{A} _{ij}.
\]
The tensor $ \mathbf{A}^2 $ is defined by contracting one covariant index by using the Riemannian metric $g$, that is, $( \mathbf{A} ^2 )_{ij}= \mathbf{A} _{ik} g^{kl} \mathbf{A} _{lj}$. The gradient $ \nabla p$ is computed with respect to $g$ and the Laplacian is $ \Delta = (\operatorname{div}(\nabla \mathbf{v} ) ^\flat )^\sharp$.

\paragraph{Nonholonomic formulation.}

Let us denote by $S_2( \mathfrak{D} )$ the space of $2$-covariant symmetric tensor fields on $ \mathfrak{D} $. We consider the right representation of $ K=\operatorname{Diff}_{vol}( \mathfrak{D} )$ on $V=S_2( \mathfrak{D} )$ given by pull-back, i.e., $ \mathbf{A} \mapsto \eta ^\ast \mathbf{A} $, and denote by $G= \operatorname{Diff}_{vol}( \mathfrak{D} ) \,\circledS\, S_2( \mathfrak{D} )$ the associated semidirect product. The Lie algebra is $ \mathfrak{g} = \mathfrak{X}_{vol}( \mathfrak{D} ) \,\circledS\, S_2( \mathfrak{D} )$ 
and we make the identification $ \mathfrak{g} ^\ast = \mathfrak{g} $ by using the $ L ^2$ Riemannian pairing
\[
\left\langle ( \mathbf{w} , \mathbf{B} ), ( \mathbf{v} ,\mathbf{A} ) \right\rangle =\int_ \mathfrak{D} g( \mathbf{w} , \mathbf{v} ) \mu +\int_\mathfrak{D} \bar g( \mathbf{B} , \mathbf{A} ) \mu ,\quad \mathbf{v} , \mathbf{w} \in \mathfrak{X}  _{vol}( \mathfrak{D} ), \quad \mathbf{A} , \mathbf{B} \in S _2 ( \mathfrak{D} ),
\]
where $ \mu $ is the Riemannian volume form and $\bar g$ is the vector bundle metric induced by $g$ on $2$-covariant symmetric tensor fields.

Let us consider the nonholonomic constraint
\begin{equation}\label{constraint2ndorder}
\mathfrak{g} ^ \Delta = \left\{  ( \mathbf{v} , \mathbf{A} )\in \mathfrak{g} \,\left|\, \mathbf{A} = 2(\operatorname{Def} \mathbf{v})^\flat  = (\nabla \mathbf{v} + \nabla \mathbf{v} ^\mathsf{T}) ^\flat \right .\right\} . 
\end{equation} 
In addition to this constraint, in order to obtain the equations for the second-order fluid, we shall also need to include external forces in the formulation. This can be easily done both at the level of the variational structures and at the level of the Dirac structures, therefore we shall not comment on the abstract formulation with forces and directly include them in this example.

Given an arbitrary Lagrangian $\ell=\ell( \mathbf{v} , \mathbf{A} ): \mathfrak{g} \rightarrow \mathbb{R}  $ and a force field $ f: \mathfrak{g} \rightarrow \mathfrak{g} ^\ast $, $f( \mathbf{v} , \mathbf{A})=\left(  f ^1 ( \mathbf{v} , \mathbf{A} ), f ^2 ( \mathbf{v} , \mathbf{A} )\right) $, the associated {\it implicit Euler-Poincar\'e-Suslov equations with external forces} read,
\begin{equation}\label{nh_general_sd_order} 
\left( \partial _t \frac{\delta \ell}{\delta \mathbf{v}  }  + \operatorname{ad}_{\mathbf{v} }^{\;\ast} \frac{\delta \ell}{\delta \mathbf{v} } +\mathbf{A}  \diamond \frac{\delta \ell}{\delta \mathbf{A} }- f ^1 ( \mathbf{v} , \mathbf{A} ) , \partial _t \frac{\delta \ell}{\delta \mathbf{A} } + \pounds _ \mathbf{v}  \frac{\delta \ell}{\delta \mathbf{A} }-f ^2 ( \mathbf{v} , \mathbf{A} )\right) \in \left( \mathfrak{g}^{ \Delta}\right) ^\circ,
\end{equation}
together with $(\mathbf{v} ,\mathbf{A} )\in\mathfrak{g}^{ \Delta}$ ( see Theorem \ref{SDP_Schneider}).
Note that in this example there is no advected quantity. We shall now compute concretely these equations.
Using the formula
\begin{equation}\label{stokes} 
\int_ \mathfrak{D}  \bar g ( \mathbf{B} ,\nabla \mathbf{v}^\flat  )\mu =-\int_ \mathfrak{D}  g((\operatorname{div}\mathbf{B}  ) ^\sharp, \mathbf{v}) \mu +\int_{ \partial \mathfrak{D}  }\mathbf{B}  ( \mathbf{n} , \mathbf{v} ) \mu _{ \partial },
\end{equation} 
on a manifold with smooth boundary, and observing that $\bar g ( \mathbf{B} ,\nabla \mathbf{v}^\flat  )=\bar g ( \mathbf{B} ,(\nabla \mathbf{v}^\mathsf{T})^\flat  )$, we deduce the formula
\[
\int_ \mathfrak{D}  \bar g ( \mathbf{B} ,(\operatorname{Def}\mathbf{v}) ^\flat  )\mu =-\int_ \mathfrak{D}  g((\operatorname{div}\mathbf{B}  ) ^\sharp, \mathbf{v}) \mu +\int_{ \partial \mathfrak{D} }\mathbf{B}  ( \mathbf{n} , \mathbf{v} ) \mu _{ \partial }.
\]
Since we assumed that $ \mathfrak{D} $ has no boundary, we obtain the following expression for the annihilator
\[
(\mathfrak{g} ^ \Delta) ^\circ  = \{ ( \mathbf{w} , \mathbf{B} )\mid \mathbf{w} =2 \mathbb{P}  ( \operatorname{div} \mathbf{B} ) ^\sharp \} \subset \mathfrak{g} ^\ast , 
\]
where $ \mathbb{P}  $ is the Hodge projector onto divergence free vector fields.

It remains to compute the diamond operator. Given $ \mathbf{A} , \mathbf{B} \in S _2 ( \mathfrak{D} )$, we have
\begin{align*} 
\int_ \mathfrak{D}\bar g( \mathbf{B} , \pounds _ \mathbf{v} \mathbf{A} ) \mu &= \int_ \mathfrak{D}\bar g( \mathbf{B} , \nabla _ \mathbf{v} \mathbf{A} +2 \mathbf{A} ( \nabla _{\_\,}\mathbf{v} , \_\,) ) \mu=  \int_ \mathfrak{D}\bar g( \mathbf{B} , \nabla _ \mathbf{v} \mathbf{A} )+ 2\bar g ( \mathbf{B} \cdot \mathbf{A} ,\nabla \mathbf{v}^\flat  ) \mu\\
&=\int_ \mathfrak{D}\bar g( \mathbf{B} , \nabla _ \mathbf{v} \mathbf{A} )\mu -2\int_ \mathfrak{D}  g((\operatorname{div}(\mathbf{B} \cdot \mathbf{A}) ) ^\sharp, \mathbf{v}) \mu +\int_{ \partial \mathfrak{D}  }(\mathbf{B} \cdot \mathbf{A} ) ( \mathbf{n} , \mathbf{v} ) \mu _{ \partial },
\end{align*}
where we defined the contraction $(\mathbf{B} \cdot \mathbf{A} )_{ij}=\mathbf{B} _{ik}g^{kl} \mathbf{A} _{lj}$ and we used the formula \eqref{stokes}, which is also valid when $ \mathbf{B} $ is not symmetric and with the convention $( \nabla \mathbf{v} ) ^\flat _{ij}= (\nabla _ i \mathbf{v} )^k g_{kj}$.

We can thus write \eqref{nh_general_sd_order} as
\begin{equation}\label{intermediate_step} 
\begin{aligned} 
&\partial _t \frac{\delta \ell}{\delta \mathbf{v} }+ \mathbb{P}  \left( \mathbf{v} \cdot \nabla  \frac{\delta \ell}{\delta \mathbf{v} }+ \nabla \mathbf{v} ^\mathsf{T} \frac{\delta \ell}{\delta \mathbf{v} }+ \bar g\left( \frac{\delta \ell}{\delta \mathbf{A} }  , \nabla _{\_\,}\mathbf{A} \right)^\sharp - 2 \operatorname{div}\left(\frac{\delta \ell}{\delta \mathbf{A} } \cdot \mathbf{A}\right)  ^\sharp-f ^1 ( \mathbf{v} , \mathbf{A} )\right) \\
& \qquad \qquad \qquad \qquad = 2 \mathbb{P} \operatorname{div} \left( \partial _t \frac{\delta \ell}{\delta \mathbf{A} } + \pounds _ \mathbf{v}  \frac{\delta \ell}{\delta \mathbf{A} }-f ^2 ( \mathbf{v} , \mathbf{A} )\right) ^\sharp.
\end{aligned}
\end{equation}  
In order to recover the second-order fluid equations, we shall consider the Lagrangian and the force fields as:
\[
\ell( \mathbf{v} , \mathbf{A} )=\int_ \mathfrak{D}\left(  \frac{1}{2} | \mathbf{v} | ^2 +\frac{1}{4}\alpha_1  | \mathbf{A} | ^2\right)  \mu, \quad f ^1 ( \mathbf{v} , \mathbf{A} )=0, \quad -2f ^2 ( \mathbf{v} , \mathbf{A} )=( \alpha _2 - \alpha _1 ) \mathbf{A} ^2 +\mu  \mathbf{A} .
\]
Since $ \delta \ell/ \delta \mathbf{v} =\mathbf{v} $ and $ \delta \ell / \delta \mathbf{A} = \alpha_1  \mathbf{A} /2$, the third and fourth term in \eqref{intermediate_step} contribute to the pressure and we obtain that \eqref{intermediate_step} reduces to the second-order Rivlin-Ericksen fluid equation
\[
\partial _t \mathbf{v} + \mathbf{v} \cdot \nabla \mathbf{v} = - \nabla p + \mu \Delta \mathbf{v} + \alpha _1 \operatorname{div}( \partial _t \mathbf{A} + \pounds _ \mathbf{v} \mathbf{A}  )^\sharp + \alpha _2 \operatorname{div} (\mathbf{A}  ^2 )^\sharp,\quad \operatorname{div} \mathbf{v} =0,
\]
as desired. Note the above choice of force field is not unique. For example, the choice $f ^1 ( \mathbf{v} , \mathbf{A} )=\mu \Delta \mathbf{v} + ( \alpha _2 -\alpha _1 )\operatorname{div}(  \mathbf{A}  ^2) ^\sharp$ and $f ^2 ( \mathbf{v} , \mathbf{A} )= 0$ also yields the desired equations.

\begin{remark}{\rm Note that the second-order Rivlin-Ericksen fluid equation can also be obtained by a holonomic variational principle. Indeed, in the special case $ \alpha _1 + \alpha _2 =0= \mu $, since they coincide with the averaged Euler equations, they arise by the usual Euler-Poincar\'e variational principle for the Lagrangian $\ell( \mathbf{v} )=\int_ \mathfrak{D}\left(  \frac{1}{2} | \mathbf{v} | ^2 +\alpha _1  | \operatorname{Def} \mathbf{v} | ^2 \right)\mu $. In the general case, the equations can still be obtained by a holonomic variational principle, by simply adding the missing terms (arising from $ \alpha _2 \neq - \alpha _1 $) as an external force in the Euler-Poincar\'e equations. However, this process is rather artificial and breaks the second Rivlin-Ericksen tensor $ \mathbf{A} _2 =\partial _t \mathbf{A} + \pounds _ \mathbf{v} \mathbf{A} $ in an unphysical way. It appears that our nonholonomic approach is more appealing in this sense, since it involves the representation of the group $\operatorname{Diff}_{vol}( \mathfrak{D} )$ on $S_2( \mathfrak{D} )$ dictated by the expression of the second Rivlin-Ericksen tensor $ \mathbf{A} _2 =\partial _t \mathbf{A} + \pounds _ \mathbf{v} \mathbf{A} $ as a Lie derivative, and it explains the occurrence of the force field $ \operatorname{div} (\partial _t \mathbf{A} + \pounds _ \mathbf{v} \mathbf{A})$ as a consequence of the nonholonomic constraint $ \mathbf{A} = 2(\operatorname{Def} \mathbf{v} ) ^\flat $.}
\end{remark}

\paragraph{Variational structures.} For second-order fluids, the variational structures are similar with those of the Chaplygin rolling ball, except that there is no advected quantity. There are however external forces in the second-order fluids.

For example, the Lagrange-d'Alembert-Pontryagin principle reads
\[
\delta \int_{ t _1 }^{ t _2}\ell( \mathbf{u} , \mathbf{Y} )+ \left\langle \Pi , \mathbf{v} - \mathbf{u} \right\rangle + \left\langle \mathbf{B} , \mathbf{A} - \mathbf{Y} \right\rangle dt+\int_{ t _1 }^{ t _2 } \left\langle f ^2 (\mathbf{u}, \mathbf{Y} ),\delta \mathbf{W}\right\rangle  dt =0,
\]
where $( \mathbf{u} , \mathbf{Y} )$ satisfy the constraint, i.e., $ \mathbf{Y} = ( \nabla \mathbf{u} + \nabla \mathbf{u} ^\mathsf{T}) ^\flat $ and for variations $\delta{\bf{v}}=\partial _t \bf{w}-[\mathbf{v},\mathbf{w}]$ and $ \delta \mathbf{A} = \partial _t \mathbf{W} + \pounds _ \mathbf{v} \mathbf{W} - \pounds _ \mathbf{w} \mathbf{A} $, where $ \mathbf{w}$ and $\mathbf{W} $ vanish at the endpoints and satisfy the constraint $ \mathbf{W} = ( \nabla \mathbf{w} + \nabla \mathbf{w} ^\mathsf{T}) ^\flat $.

The alternative variational structure in Remark \ref{Hybrid_HP_advected_constraints_sdp} reads
\[
\delta \int_{ t _1 }^{ t _2}\left( \ell( \mathbf{v} , \mathbf{A} ) + \left\langle \Pi , \dot \eta \circ \eta ^{-1} - \mathbf{v} \right\rangle + \left\langle \mathbf{B} , \eta _\ast \dot a - \mathbf{A} \right\rangle \right) dt+ \int_{ t _1 }^{ t _2} \left\langle f ^2 (\mathbf{v} , \mathbf{A} ) , \eta_\ast \delta a \right\rangle dt,
\]
where $( \dot \eta , \dot a ) \in \Delta _G( \eta , a)$, for arbitrary variations $ \delta \mathbf{v} $, $ \delta \mathbf{A} $, $ \delta \Pi $, $ \delta \mathbf{B} $, and variations $ \delta \eta , \delta a $ vanishing at the endpoints as well as satisfying the constraint $( \delta \eta , \delta a) \in \Delta _G( \eta , a)$. Note that $ \Delta _G\subset T( \operatorname{Diff}_{vol}( \mathfrak{D} ) \,\circledS\, S _2 ( \mathfrak{D} ))$ is the right-invariant distribution induced by $ \mathfrak{g} ^\Delta \subset \mathfrak{g}$.

The variational structures on the Hamiltonian side can be written in the same way as in the other examples.

\paragraph{Dirac formulation of second-order Rivlin-Ericksen fluids.}  Given the constraint $\Delta _{G}\subset TG$ of second-order Rivlin-Ericksen fluid as in \eqref{constraint2ndorder}, we consider the induced Dirac structure $D_{ \Delta _G } $ on $T^*G$, where $G= \operatorname{Diff}_{vol}( \mathfrak{D} ) \,\circledS\, S_2( \mathfrak{D} )$. Reduction by $G$ yields the Dirac structure $D^{/G}_{\Delta _{G}}$ given by
\begin{equation*}\label{condition_red_dirac_rollingtype_Chaplygin} 
\begin{array}{c} 
\left( ( \Pi   ,\mathbf{B} , \mathbf{v} , \mathbf{A} , \boldsymbol{\rho}  , \boldsymbol{\Sigma}  ),(\Pi  ,\mathbf{B}  , \boldsymbol{\beta}  , \boldsymbol{\Gamma}  , \mathbf{u}  , \mathbf{Y}  ) \right) \in D^{/G}_{\Delta _{G}}\\
\Longleftrightarrow\\
\left( \boldsymbol{\beta}  + \boldsymbol{\rho}  +\operatorname{ad}^*_ \mathbf{v} \Pi   + \mathbf{A} \diamond \mathbf{B} , \boldsymbol{\Gamma}  + \boldsymbol{\Sigma}  + \pounds _ \mathbf{v} \mathbf{B}  \right)  \in (\mathfrak{g}  ^\Delta )^\circ \quad \text{and} \quad (\mathbf{u}  , \mathbf{Y} ) = (\mathbf{v}  , \mathbf{A} )  \in \mathfrak{g}  ^\Delta\\
\Longleftrightarrow\\
\boldsymbol{\beta}  + \boldsymbol{\rho}  +\operatorname{ad}^*_ \mathbf{v} \Pi   + \mathbf{A} \diamond \mathbf{B}= 2 \mathbb{P}  ( \operatorname{div}(\boldsymbol{\Gamma}  + \boldsymbol{\Sigma}  + \pounds _ \mathbf{v} \mathbf{B} ))^\sharp, \quad  \mathbf{Y} = 2( \operatorname{Def} \mathbf{u} ) ^\flat  , \quad \mathbf{u} = \mathbf{v} , \quad \mathbf{Y} = \mathbf{A},
\end{array} 
\end{equation*}
where we recall that $\mathfrak{g}  ^\Delta (\boldsymbol{\Gamma} )= \mathfrak{se}(3)^ \Delta ( \boldsymbol{\Gamma} )=\left\{(\Omega  , \mathbf{X} )\in \mathfrak{se}(3)\mid \mathbf{X} = \Omega (r \boldsymbol{\Gamma}  +l \boldsymbol{\chi} ) \right \}$.
\medskip

It follows that the second-order Rivlin-Ericksen fluid equations can be written as the Lagrange-Dirac system
\[
\left( ( \Pi   ,\mathbf{B} , \mathbf{v} , \mathbf{A} ,\dot{ \Pi } , \dot{ \mathbf{B} } ),\mathbf{d} ^{/ G}_D\ell (\mathbf{u}  , \mathbf{Y}  )\right) \in D^{/G}_{\Delta _{G}},
\]
or, as a Hamilton-Dirac system
\[
\left( ( \Pi   ,\mathbf{B} , \mathbf{v} , \mathbf{A} ,\dot{ \Pi } , \dot{ \mathbf{B} } ),\mathbf{d} ^{/ G}h ( \Pi   ,\mathbf{B} )\right) \in D^{/G}_{\Delta _{G}}.
\]

%%%
%%%
%%%

\paragraph{Acknowledgements.} FGB is partially supported by a 'Projet Incitatif de Recherche' from ENS-Paris and by the ANR project GEOMFLUID; HY is partially supported by JSPS (26400408), JST-CREST, Waseda University (2014B-162), and the IRSES project ``Geomech" (246981) within the 7th European Community Framework Programme.

\end{document}